\documentclass[11pt]{amsart}
\usepackage{amsmath, amssymb}
\usepackage{enumerate}
\usepackage{color}
\usepackage[colorlinks=true, pdfstartview=FitV, linkcolor=blue,citecolor=blue, urlcolor=blue]{hyperref}
\usepackage[all, knot]{xy}
\usepackage{tikz}
\usepackage{diagbox}
\usepackage{float}
\usetikzlibrary{arrows.meta}
\renewcommand{\baselinestretch}{1.2}
\numberwithin{equation}{section}
\newtheorem{Thm}{Theorem}[section]
\newtheorem{Prop}[Thm]{Proposition}

\newtheorem{Lem}[Thm]{Lemma}
\theoremstyle{definition}
\newtheorem{defn}[Thm]{Definition}
\newtheorem{Ex}[Thm]{Example}
\newtheorem{Rmk}[Thm]{Remark}
\newcommand{\Z}{\mathbb{Z}}
\newcommand{\BB}{B}
\newcommand{\ep}{\varepsilon}
\newcommand{\wt}{{\rm wt}}

\title[Young wall construction of level-1 highest weight crystals]
{Young wall construction of level-1 highest weight crystals over  
$U_q(D_4^{(3)})$ and $U_q(G_2^{(1)})$}

\author[Zhaobing Fan]{Zhaobing Fan${}^*$}
\address{Harbin Engineering University,
	Harbin, China}
\email{fanzhaobing@hrbeu.edu.cn}
\thanks{${}^*$The corresponding author. {All authors contribute equally.}}

\author[Shaolong Han]{Shaolong Han}
\address{Harbin Engineering University,
	Harbin, China}
\email{algebra@hrbeu.edu.cn}
\thanks{}

\author[Seok-Jin Kang]{Seok-Jin Kang}
\address{Korea Research Institute of Arts and Mathematics,
	Asan-si, Chungcheongnam-do, 31551, Republic of Korea}
\email{soccerkang@hotmail.com}
\thanks{}

\author[Yong-Su Shin]{Yong-Su Shin${}^*$}
\thanks{}
\address{Department of Mathematics, Sungshin Women's University,  Seoul, 02844, Republic of  Korea}
\email{ysshin@sungshin.ac.kr}



\begin{document}

\begin{abstract}
With the help of path realization and affine energy function,
we give a Young wall construction of level-1 highest weight crystals
$B(\lambda)$ over $U_{q}(G_{2}^{(1)})$ and $U_{q}(D_{4}^{(3)})$.
Our construction is based on four different shapes of colored 
blocks, $\mathbf O$-block, $\mathbf I$-block, $\mathbf L$-block and $\mathbf{LL}$-block, obtained by cutting the unit cube in three different ways.
\end{abstract}

\maketitle

\begin{center}
	{\it In Memory of Georgia Benkart}
\end{center}

\tableofcontents

\section*{Introduction}
In \cite{Kas90,Kas91}, Kashiwara introduced the \textit{crystal basis} theory over quantum groups $U_q(\mathfrak g)$. The combinatorial features in crystal bases reveal the internal structure of integrable highest weight modules over quantum groups and their tensor product decomposition. 
Therefore, the problem of concrete realization of crystal bases has become one of the most important
research topics in representation theory.

There have been a lot of works on the realization of the highest weight crystal $B(\lambda)$ over $U_{q}(\mathfrak g)$
and the crystal $B(\infty)$ of the negative half $U_{q}^{-}(\mathfrak g)$.
For example, in \cite{KN94}, Kashiwara and Nakashima gave the constructions of crystal bases in terms of generalized Young tableaux for all finite dimensional irreducible modules over $U_q(\mathfrak g)$ of types $A_n$, $B_n$, $C_n$ and $D_n$. 
In \cite{KMN1,KMN2}, Kang {\it et al.} introduced the notion of {\it perfect crystals} for quantum affine algebras
and gave an explicit realization of the highest weight crystals $B(\lambda)$ in terms of {\it paths}. 
In \cite{Kang03}, Kang defined the concept of \textit{Young walls} as a new combinatorial model for realizing level-1 highest weight crystals $B(\lambda)$ for quantum affine algebras of types $A_n^{(1)}$, $A_{2n}^{(2)}$, $D_{n+1}^{(2)}$, $A_{2n-1}^{(2)}, D_n^{(1)}$ and $B_n^{(1)}$. The Young walls consist of
colored blocks with various shapes, and can be viewed as generalizations of colored
Young diagrams. 
In \cite{HKL04}, Hong, Kang, and Lee complemented the Young wall realization of level-$1$ crystals for the quantum affine algebra of type $C_n^{(1)}$. In \cite{KL06, KL09}, Kang and Lee generalized the Young wall realization of highest weight crystals of classical quantum affine algebras to arbitrary level.
For the quantum affine algebras of exceptional types $G_2^{(1)}$ and $D_4^{(3)}$, Kim and Shin gave a combinatorial realization of crystals $B(\lambda)$ and $B(\infty)$ using strip bundles (cf. \cite{KS19,KS22}).
The polyhedral realizations of crystal bases can be found in \cite{NZ97,Nak99}.

The aim of this paper is to give a new realization of level-$1$ highest weight crystals
$B(\lambda)$ for the quantum affine algebras $U_q(D_4^{(3)})$ and $U_q(G_2^{(1)})$. 
We use \textit{perfect crystal} theory to obtain this realization, which is different from Kim and Shin's constructions. In \cite{KMOY,M05,JM11,MMO10,Ya98}, Kashiwara, Jayne, Misra {\it et al}. have studied the perfect crystals
for quantum affine algebras $U_q(D_4^{(3)})$ and $U_q(G_2^{(1)})$. 
The key point of our construction is to build a level-1 perfect crystal in terms of colored blocks.
For this purpose, we introduce four different shapes of colored blocks, 
 $\mathbf O$-block, $\mathbf I$-block, $\mathbf L$-block and $\mathbf{LL}$-block,
 obtained by cutting the unit cube in three different ways.
These new blocks are 1/2 part, 1/4 part, 3/8 part and 3/4 part of the unit cube. 
We paint these blocks with colors $0$, $1$, $2$ and work with these colored blocks 
to construct \textit{Young columns}. 
Unlike the case of classical quantum affine algebras, we cannot give a realization of perfect crystals in terms of Young columns in a straightforward way. To overcome this difficulty, we introduce the concept of equivalence classes of Young columns, which leads to a realization of level-$1$ perfect crystal $B_1$ (resp. $B'_1$) of $U_{q}(D_{4}^{(3)})$ (resp. $U_{q}(G_{2}^{(1)})$).

As in the path realization of highest weight crystals over $U_q(D_4^{(3)})$ and $U_q(G_2^{(1)})$,
we define the ground-state wall in terms of our colored blocks. 
Then we proceed
to define the Young wall as a wall built by stacking colored blocks of various shapes on the ground-state wall. We define the {\it reduced} Young wall using the \textit{affine energy function}.
We calculate the values of energy function on $C_1\otimes C_1$ (resp. $C'_1\otimes C'_1$), 
where $C_1$ (resp. $C'_1$) is the perfect crystal of $U_{q}(D_{4}^{(3)})$ (resp. $U_{q}(G_{2}^{(1)})$) consisting of all equivalence classes of Young columns. 
Then we can use the values of energy function to give a precise description of adjacent columns in reduced Young walls. 
From this point of view, the  combinatorial structure of reduced Young walls is very clear. 
We can define the \textit{tensor product rule} on the set of Young walls. 
Using the affine energy function, we show that the set of all reduced Young walls 
$\mathcal{Y}(\lambda)$ is closed under the action of Kashiwara operators. 
It is straightforward to prove that the set $\mathcal{Y}(\lambda)$ possess an affine crystal structure. 
With the help of path realization, we prove that $\mathcal{Y}(\lambda)$ is isomorphic to $B(\lambda)$ as crystals.

The paper is organized as follows. In Section \ref{Perfect crystal}, we review the definition of quantum affine algebras, perfect crystals and path realization of $B(\lambda)$. In Section \ref{Perfect crystal of G_2^1}, we recall the level-$1$ perfect crystals of quantum affine algebras $U_q(D_4^{(3)})$ and $U_q(G_2^{(1)})$. In Section \ref{Young wall}, we use $\mathbf O$-block, $\mathbf I$-block, $\mathbf L$-block and $\mathbf{LL}$-block to construct Young columns and give a realization of level-$1$ perfect crystals using equivalence classes of Young columns. Then we give the definition of Young walls and reduced Young walls. In Section \ref{Young wall realization}, we show that the
level-$1$ highest weight crystals $B(\lambda)$ are realized as the affine crystals
consisting of reduced Young walls. Finally, we give an
explicit description of the top part of level-$1$ highest weight crystals in terms of reduced Young walls.

\vskip 1em

\noindent
{\bfseries  Acknowledgements.} 
Z. Fan was partially supported by the NSF of China grant 12271120, the NSF of Heilongjiang Province grant JQ2020A001, and the Fundamental Research Funds for the central universities. S.-J.  Kang was supported by China grant YZ2260010601. Y.S. Shin was supported by the Basic Science Research Program of the NRF(Korea) under the grant No.2022R1F1A10633741.

\section{Perfect crystals of quantum affine algebras}\label{Perfect crystal} 
Let us fix some basic notations here.
\begin{itemize}
\item[\textendash] $I = \{0,1,...,n\}$: index set.
\item[\textendash] $A=(a_{ij})_{i,j\in I}$: affine Cartan matrix.
\item[\textendash] $D={\rm diag}\{s_i\in \mathbb Z_{>0}\mid i\in I\}$: diagonal matrix such that $DA$ is symmetric.
\item[\textendash] $P^{\vee} =(\oplus_{i\in I}\mathbb Z h_i) \oplus
\Z d$: dual weight lattice.
\item[\textendash] $\mathfrak{h}=\mathbb C \otimes_{\mathbb Z}
P^{\vee}$: Cartan subalgebra.
\item[\textendash] $P=\{\lambda \in
\mathfrak{h}^* \mid \lambda(P^{\vee})\subset \Z\}$: affine weight lattice.
\item[\textendash] $\Pi^\vee=\{h_i\mid i\in I\}\subset P^\vee$: the set of simple coroots.
\item[\textendash] $\Pi=\{\alpha_i\mid i\in I\}\subset P$: the set of simple roots.
\item[\textendash] $\mathfrak g$: affine Lie algebra associated to the Cartan datum $(A, P^{\vee}, P, \Pi^\vee, \Pi)$.
\item[\textendash] $\delta$, $c$, $\Lambda_i\ (i\in I)$: null root, canonical central element, fundamental weight.
\item[\textendash] $P^{+}$: the set of affine dominant integral weights.
\item[\textendash] $\bar{P}=\oplus_{i\in I}\mathbb Z \Lambda_i$: the set of classical weights.
\item[\textendash] $\bar{P}^{+}$: the set of classical dominant integral weights.
\item[\textendash] $l=\lambda(c)$: level of the affine (classical) dominant integral weight $\lambda$.
	\end{itemize}

We set
\begin{align*}
q_i = q^{s_i},\quad \left[n\right]_i=\dfrac{q_i^n-q_i^{-n}}{q_i-q_i^{-1}},\quad
\left[m\right]_i!=\left[m\right]_i\left[m-1\right]_i \cdots
\left[1\right]_i,\quad \left[0\right]_i!=1,
\end{align*}
where $i\in I$, $n\in\mathbb Z$ and $m\in \mathbb Z_{> 0}$.

The {\it quantum affine algebra} $U_q(\mathfrak{g})$ is the associative algebra over $\mathbb Q(q)$ with unity
generated by $e_i, f_i$ $(i\in I)$ and $q^h$ $( h \in P^{\vee})$
satisfying the following defining relations.
\begin{enumerate}
	
	\item $q^0 = 1,\ q^h q^{h^{\prime}} = q^{h + h^{\prime}}$  for all  $h,h^{\prime} \in P^{\vee}$,
	
	\item  $q^h e_i q^{-h} = q^{\alpha_i(h)}e_i,\ q^h f_i q^{-h} = q^{-\alpha_i(h)}f_i$  for  $h \in P^{\vee}$,
	
	\item  $e_i f_j - f_j e_i = \delta_{ij} \dfrac{K_i - K_i^{-1}}{q_i - q_i^{-1}}$  for  $i,j \in I$,
	
	\item $\displaystyle\sum_{k=0}^{1-a_{ij}}(-1)^k e_{i}^{(1-a_{ij}-k)} e_j e_{i}^{(k)} = 0 $ for $i \ne j$,

	\item $\displaystyle\sum_{k=0}^{1-a_{ij}}(-1)^k f_{i}^{(1-a_{ij}-k)} f_j f_{i}^{(k)} = 0 $  for $i \ne j$,
\end{enumerate}
where $e_i^{(k)}= \frac{e_i^k}{\left[k\right]_{i}!}$,
$f_i^{(k)} = \frac{f_i^k}{\left[ k \right]_{i}!}$ and $K_i = q_i^{h_i}$.

We denote by $U_q'(\mathfrak{g})$ the subalgebra of
$U_q(\mathfrak{g})$ generated by $e_i,f_i,K_i^{\pm1}$ $(i \in I)$.

\begin{defn}
	An {\it affine crystal} (resp. a {\it classical crystal}) is
	a set $\BB$ together with the maps $\wt : \BB \rightarrow P$
	(resp. $\wt: \BB \rightarrow \bar{P}$), $\tilde{e}_i,
	\tilde{f}_i: \BB \rightarrow \BB \cup \{ 0 \}$ and $\varepsilon_i,
	\varphi_i : \BB \rightarrow \mathbb{Z} \cup \{-\infty\}$  $(i \in I)
	$ satisfying the following conditions.
	\begin{enumerate}
		\item  $\varphi_{i}(b) = \varepsilon_{i}(b) + \langle h_i, \wt
		(b) \rangle$ \ for all  $ i \in I$,
		
		\item  $ \wt(\tilde {e}_i b) = \wt (b) + \alpha_i$ \ if  $\tilde{e}_i b \in \BB$,
		
		\item  $ \wt(\tilde{f}_i b) = \wt(b) - \alpha_i$ \ if  $\tilde{f}_ib \in \BB$,
		
		\item  $\varepsilon_i(\tilde{e} _ib) = \varepsilon_i(b) - 1$, \
		$\varphi_i(\tilde{e}_i b) = \varphi_i(b) + 1$ \ if  $\tilde{e}_ib \in
		\BB$,

		\item  $\varepsilon_i(\tilde{f} _ib) = \varepsilon_i(b) + 1$, \
		$\varphi_i(\tilde{f}_i b) = \varphi_i(b) - 1$ \ if  $\tilde{f}_ib \in
		\BB$,
		
		\item  $\tilde{f}_i b = b'$\ if and only if \ $b =
		\tilde{e}_i b'$ for  $b,b'\in \BB,\ i \in I $,
		
		\item  If $\varphi_i(b) = -\infty$  for  $b\in \BB$, then
		$\tilde{e}_i b = \tilde{f}_i b = 0$.
		
	\end{enumerate}
\end{defn}

\begin{defn}
	Let $B$ and $B'$ be affine or classical crystals.
	A {\it crystal morphism} $\Psi : B \rightarrow B'$ is a map
	$\Psi : B \cup \{0\} \rightarrow B' \cup \{0\}$ satisfying the following
	conditions.
	\begin{enumerate}
		\item  $\Psi(0)=0$,
		\item if $b \in B$ and $\Psi(b) \in B'$, then $\wt(\Psi(b))=\wt(b)$, $\varepsilon_i(\Psi(b))=\varepsilon_i(b)$, and $\varphi_i(\Psi(b))=\varphi_i(b)$ for all $i \in I$,
		\item if $b, b' \in B$, $\Psi(b), \Psi(b') \in B'$ and $\tilde{f}_i b=b'$, then $\tilde{f}_i \Psi(b)=\Psi(b')$ and $\Psi(b)=\tilde{e}_i \Psi(b')$ for all $i \in I$.
	\end{enumerate}
	A crystal morphism $\Psi : B \rightarrow B'$ is called an {\it isomorphism} if it is a bijection from $B \cup \{0\}$ to $B' \cup \{0\}$.
\end{defn}

For two crystals $B$ and $B' $, we define the {\it tensor
	product} $B \otimes B'$ to be the set $B \times B' $
whose crystal structure is given by
\begin{equation*}
	\begin{aligned}
		\tilde {e}_i(b \otimes b')&= 
\begin{cases} \tilde {e}_i b \otimes b' &\text{if \;$\varphi_i(b)\ge \varepsilon_i(b')$}, \\[-.5ex] 
			b\otimes \tilde {e}_i b' & \text{if\; $\varphi_i(b) < \varepsilon_i(b')$}, 
\end{cases} \\[-.5ex] 
		\tilde {f}_i(b \otimes b')&= 
\begin{cases} \tilde {f}_i b \otimes b'
			&\text{if \; $\varphi_i(b) > \varepsilon_i(b')$}, \\[-.5ex] 
			b\otimes \tilde {f}_i b' & \text{if \; $\varphi_i(b) \le
				\varepsilon_i(b')$},
\end{cases}\\
%
		\wt(b \otimes b')&= \wt(b)+\wt(b'),\\
		\ep_{i}(b \otimes b')&= \max(\ep_{i}(b), \ep_i(b') -
		\langle h_i, \wt(b) \rangle ),\\
		\varphi_{i}(b \otimes b')& = \max(\varphi_{i}(b'), \varphi_i(b)
		+\langle h_i, \wt(b')\rangle ).
	\end{aligned}
\end{equation*}

Let $\BB$ be a classical crystal. For an element $b \in \BB$, we
define
\begin{eqnarray*}
	\varepsilon(b) = \displaystyle\sum_{i \in I} \varepsilon_i(b)\Lambda_i, &
	\varphi(b) = \displaystyle\sum_{i \in I} \varphi_i(b)\Lambda_i.
\end{eqnarray*}

\begin{defn}
	Let $l$ be a positive integer. A classical crystal $\BB$ is
	called a \textit{perfect crystal of level $l$} if
	
	\begin{enumerate}
		\item  there exists an irreducible finite dimensional $U'_q (\mathfrak{g})$-module with a crystal basis whose crystal graph is isomorphic to $\BB$,
		
		\item  $\BB \otimes \BB$ is connected,
		
		\item  there exists a classical weight $ \lambda_0 \in \bar{P}$
		such that 
		$$\displaystyle \wt(\BB) \subset \lambda_0 + \sum_{i \ne 0}
		\mathbb{Z}_{\le0} \: \alpha_i, \quad
		\#(\BB_{\lambda_0})=1,$$ 
		where
		$\BB_{\lambda_0}=\{ b \in \BB ~ | ~ \wt(b)=\lambda_0 \}$,
		
		\item  for any $b \in \BB ,\ \varepsilon(b)(c)
		\ge l $,
		
		\item  for any $ \lambda \in \bar{P} ^{+}$ with $\lambda(c)= l$, there
		exist  unique $ b^{\lambda}\in B $ and $b_{\lambda} \in B$ such that
		$\varepsilon(b^{\lambda})  = \varphi(b_{\lambda})= \lambda.$
		
	\end{enumerate}
\end{defn}

The vectors $b^\lambda$ and $b_\lambda$ are called the \textit{minimal vectors}.

\begin{Thm} [{\cite{KMN1}}]
	Let $\BB$ be a perfect crystal of level $l$ ($l \in
	\mathbb{Z}_{>0}$). For any $\lambda \in \bar{P}^+$ with
	$\lambda(c)= l$, there exists a unique classical crystal
	isomorphism
	\begin{equation*}
		\begin{aligned}
			\Psi :\BB(\lambda) \stackrel{\sim}{\longrightarrow}
			\BB(\ep(b_\lambda)) \otimes \BB &
			&\text{given by}& &  u_\lambda  \longmapsto  u_{\ep(b_\lambda)} \otimes b_\lambda ,
		\end{aligned}
	\end{equation*}
	where $u_\lambda$ is the highest weight vector in $\BB(\lambda)$ and $b_\lambda$ is the unique vector in $\BB$ such that $\varphi(b_{\lambda})=\lambda $.
\end{Thm}
Set 
$$
\lambda_0 =\lambda, \quad \lambda_{k+1}=\ep(b_{\lambda_k}), \quad b_0=b_{\lambda_0}, \quad b_{k+1}=b_{\lambda_{k+1}}.
$$
Applying the above crystal isomorphism repeatedly, we obtain a sequence of crystal isomorphisms
\begin{equation*}
	\begin{array}{ccccccc} \BB(\lambda) &
		\stackrel{\sim}{\longrightarrow} & \BB(\lambda_1) \otimes \BB &
		\stackrel{\sim}{\longrightarrow} & \BB(\lambda_2)\otimes \BB \otimes
		\BB & \stackrel{\sim}{\longrightarrow} & \cdots
		\\
		u_\lambda & \longmapsto & u_{\lambda_1} \otimes b_0 & \longmapsto &
		u_{\lambda_2} \otimes b_1 \otimes b_0 & \longmapsto & \cdots.
\end{array} \end{equation*}
In this process, we obtain an infinite sequence $\mathbf{p}_\lambda =(b_k)^{\infty} _{k=0} \in \BB^{\otimes \infty}
$, which is called the \textit{ground-state path of weight} $\lambda$.

\medskip

Set 
$$
\mathcal{P}(\lambda):=\{\mathbf{p}=(p(k))^{\infty}_{k=0} \in \BB^{\otimes \infty} \;|\; p(k) \in \BB,\ p(k)=b_k \ \text{for all} \ k \gg 0  \}.
$$
The elements in $\mathcal{P}(\lambda)$ are called the $\lambda$-\textit{paths}. The following result gives the {\it path realization} of $\BB(\lambda)$.

\begin{Prop}[{\cite{KMN1}}]\label{prop:path realization}
	There exists an isomorphism of classical crystals
	\begin{equation*}
			\Psi_{\lambda} :\BB(\lambda) \stackrel{\sim}{\longrightarrow}
			\mathcal{P}({\lambda})\ \
			\text{given by}\ \ u_\lambda \longmapsto\mathbf{p}_\lambda,
	\end{equation*}
	where $u_{\lambda}$ is the highest weight vector in $B(\lambda)$.
	\label{path_realization}
\end{Prop}

\section{Perfect crystals of $U_q(D_4^{(3)})$ and $U_q(G_2^{(1)})$} \label{Perfect crystal of G_2^1}
Let us recall the  Cartan datum for affine Lie algebras of types $D_4^{(3)}$ and $G_2^{(1)}$. Let $I=\{0,1,2\}$. We denote by
$$
\{h_0, h_1, h_2\},\quad \{\alpha_0, \alpha_1, \alpha_2\},\quad \{\Lambda_0, \Lambda_1, \Lambda_2\}
$$
the set of 
simple coroots, simple roots and fundamental weights, respectively.

The Cartan matrix for affine Lie algebras of types $D_4^{(3)}$ and $G_2^{(1)}$ are given as follows.

\begin{equation*}
D_4^{(3)}:\ \left(
\begin{array}{rrr}
	2  & -1 & 0  \\
	-1 & 2  & -3 \\
	0  & -1 & 2
\end{array}
\right) \quad\quad		
G_2^{(1)}:\ \left(
\begin{array}{rrr}
	2  & -1 & 0  \\
	-1 & 2  & -1 \\
	0  & -3 & 2
\end{array}
\right)
\end{equation*}

and their Dynkin diagrams are depicted as follows.
\[
D_4^{(3)}:\ 	
\xymatrix@R=.0ex{
\circ\ar@{-}[r]&\circ&\circ\ar@3{->}[l]	\\
	0&1&2
}\quad\quad
G_2^{(1)}: \
\xymatrix@R=.0ex{
	\circ\ar@{-}[r]&\circ&\circ\ar@3{<-}[l]\\
	0&1&2
}
\]

The null root $\delta$ and the canonical central element $c$ are 
given by
\begin{align*}
&D_4^{(3)}:\ \delta=\alpha_0+2\alpha_1+\alpha_2, \quad c=h_0+2h_1+3h_2,\\
&G_2^{(1)}:\ \delta=\alpha_0+2\alpha_1+3\alpha_2, \quad c=h_0+2h_1+h_2.
\end{align*}

The level-$l$ perfect crystals $B_{l}$ of $U_q(D_4^{(3)})$ and $U_q(G_2^{(1)})$ were constructed in \cite{KMOY} and \cite{JM11}, respectively.

Now we consider the level-$1$ dominant integral weight $\lambda$, i.e., $l=\lambda(c)=1$. Then the level-$1$ dominant integral weight of $U_q(D_4^{(3)})$ is $\Lambda_0$ and the level-$1$ dominant integral weights of $U_q(G_2^{(1)})$ are $\Lambda_0$ and $\Lambda_2$.

\begin{Ex}[{\cite{{KMOY}}}] \label{ex:level 1 perfect crystal D43}
	 Let $B_1$ be the level-$1$ perfect crystal of $U_q(D_4^{(3)})$. We denote the elements of $B_1$ by 
	\begin{figure}[ht] 
		\vskip -1pc
		\begin{align*}
			u_1&=(1,0,0,0,0,0),& 
			u_2&=(0,1,0,0,0,0),& 
			u_3&=(0,0,2,0,0,0),\\
		    u_0&=(0,0,1,1,0,0),&
		u_{\bar{3}}&=(0,0,0,2,0,0),&    
			u_{\bar{2}}&=(0,0,0,0,1,0),\\
		u_{\bar{1}}&=(0,0,0,0,0,1), &
			u_\phi&=(0,0,0,0,0,0).
		\end{align*}
		\vskip -1pc
		\caption{The elements of $B_1$ }
	\end{figure}
	
	
	The crystal graph of $B_1$ is given as follows.
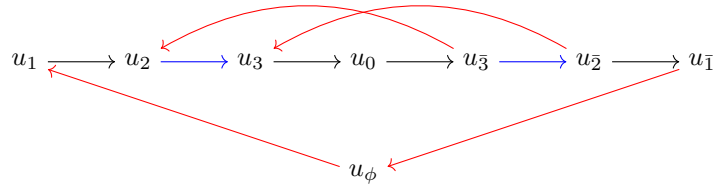
\begin{figure}[H] 
	{\small
		\centering
			\begin{tikzpicture}[scale=1.5]
				\node (u1) at (0,0) {$u_1$};	
				\node (u2) at (1,0) {$u_2$};
				\node (u3) at (2,0) {$u_3$};
				\node (u0) at (3,0) {$u_0$};
				\node (u3bar) at (4,0) {$u_{\bar{3}}$};
				\node (u2bar) at (5,0) {$u_{\bar{2}}$};
				\node (u1bar) at (6,0) {$u_{\bar{1}}$};
				\node (phi) at (3,-1) {$u_\phi$};
				\draw[black,->]  (u1) -- (u2);
				\draw[blue,->]  (u2) -- (u3);
				\draw[black,->]  (u3) -- (u0);
				\draw[black,->]  (u0) -- (u3bar);
				\draw[blue,->]  (u3bar) -- (u2bar); 
				\draw[black,->]  (u2bar) -- (u1bar);
				\draw[red, ->, bend right]  (u2bar) to node [swap] {} (u3);
				\draw[red, ->, bend right]  (u3bar) to node [swap] {} (u2);
					\draw[red,->]  (u1bar) -- (phi);
				\draw[red,->]  (phi) -- (u1);
					\end{tikzpicture}
	}
	\vskip -1pc
	\caption{The perfect graph of $B_1$}
\end{figure} 	

Here and in the sequel, we use the red arrows to represent the action of Kashiwara operator $\tilde{f}_0$, the black arrows to represent the action of Kashiwara operator $\tilde{f}_1$ and the blue arrows to represent the action of Kashiwara operator $\tilde{f}_2$.

For $\lambda=\Lambda_0$, the minimal element is $u_\phi$ and the $\Lambda_0$-ground state path is $\mathbf p_{\Lambda_0}=(\cdots\otimes \phi\otimes \phi\otimes \phi)$. 
\end{Ex}

\begin{Ex}[{\cite{{JM11}}}] \label{ex:level 1 perfect crystal G21}
	 Let $B'_1$ be the level-$1$ perfect crystal of $U_q(G_2^{(1)})$. We denote the elements of $B'_1$ by 	
	\begin{align*}
		v_0& =(0,0,0,0,0,0), & v_1&=(1,0,0,0,0,0), &  v_2&=(0,1,0,0,0,0),\\
		v_3&=(0,\frac{2}{3},\frac{2}{3},0,0,0), & v_4&=(0,\frac{1}{3},\frac{4}{3},0,0,0), & v_5&=(0,\frac{1}{3},\frac{1}{3},1,0,0),\\
		v_6&=(0,0,2,0,0,0), & v_7&=(0,0,1,1,0,0), & v_{\bar{7}}&=(0,\frac{1}{3},\frac{1}{3},\frac{1}{3},\frac{1}{3},0),\\
		v_{\bar{6}}&=(0,0,0,2,0,0), & v_{\bar{5}}&=(0,0,1,\frac{1}{3},\frac{1}{3},0), & v_{\bar{4}}&=(0,0,0,\frac{4}{3},\frac{1}{3},0),\\
		v_{\bar{3}}&=(0,0,0,\frac{2}{3},\frac{2}{3},0), & v_{\bar{2}}&=(0,0,0,0,1,0), & v_{\bar{1}}&=(0,0,0,0,0,1).
	\end{align*}
	The crystal graph of $B'_1$ is given as follows.
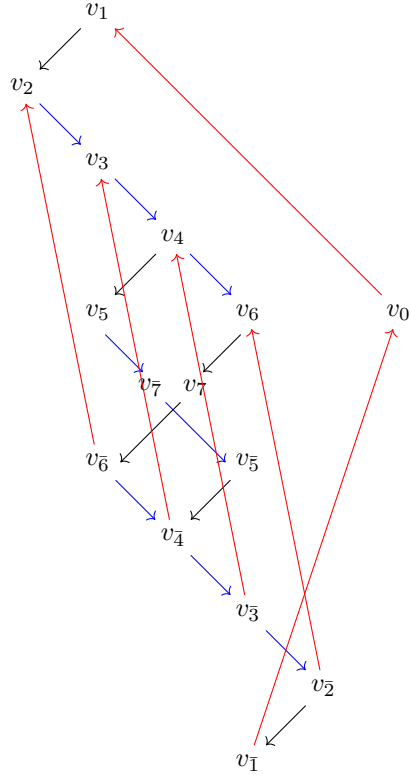
\begin{figure}[H] 
	{\small
		\centering
	\begin{tikzpicture}[scale=1]
	\node (v1) at (0,0) {$v_1$};	
	\node (v2) at (-1,-1) {$v_2$};
	\node (v3) at (0,-2) {$v_3$};
	\node (v4) at (1,-3) {$v_4$};
	\node (v5) at (0,-4) {$v_5$};
	\node (v6) at (2,-4) {$v_6$};
    \node (v7) at (1.3,-5) {$v_7$}; 
    \node (v7bar) at (0.7,-5) {$v_{\bar{7}}$};
    \node (v6bar) at (0,-6) {$v_{\bar{6}}$};
    \node (v5bar) at (2,-6) {$v_{\bar{5}}$};
    \node (v4bar) at (1,-7) {$v_{\bar{4}}$};
    \node (v3bar) at (2,-8) {$v_{\bar{3}}$};
    \node (v2bar) at (3,-9) {$v_{\bar{2}}$};
    \node (v1bar) at (2,-10) {$v_{\bar{1}}$};
    \node (v0) at (4,-4) {$v_0$};
    \draw[black,->]  (v1) -- (v2);
    \draw[black,->]  (v2bar) -- (v1bar);
    \draw[blue,->]  (v2) -- (v3);
    \draw[blue,->]  (v3) -- (v4);
    \draw[blue,->]  (v4) -- (v6); 
    \draw[blue,->]  (v6bar) -- (v4bar);
    \draw[blue,->]  (v4bar) -- (v3bar);
    \draw[blue,->]  (v3bar) -- (v2bar);
    \draw[black,->]  (v4) -- (v5);
    \draw[black,->]  (v5bar) -- (v4bar);
    \draw[red,->]  (v1bar) -- (v0);
    \draw[red,->]  (v0) -- (v1);
    \draw[blue,->]  (0.1,-4.3) -- (0.6,-4.8);
     \draw[black,->]  (1.9,-4.3) -- (1.4,-4.8);
     \draw[blue,->]  (0.9,-5.2) -- (1.7,-6);
     \draw[black,->]  (1.1,-5.2) -- (0.3,-6);    
     \draw[red,->]  (v6bar) -- (v2);  
    \draw[red,->]  (v4bar) -- (v3);   
     \draw[red,->]  (v3bar) -- (v4);   
     \draw[red,->]  (v2bar) -- (v6); 
\end{tikzpicture}
	}
\vskip -1pc
\caption{The perfect graph of $B_1$}
\end{figure} 

For $\lambda=\Lambda_0$, the minimal element is $v_0$ and the $\Lambda_0$-ground state path is $\mathbf p_{\Lambda_0}=(\cdots\otimes v_0\otimes v_0\otimes v_0)$. 

For $\lambda=\Lambda_2$, the minimal element is $v_{\bar{7}}$ and the $\Lambda_2$-ground state path is $\mathbf p_{\Lambda_2}=(\cdots\otimes v_{\bar{7}}\otimes v_{\bar{7}}\otimes v_{\bar{7}})$.

\end{Ex}

\section{Young wall model for $U_q(D_4^{(3)})$ and $U_q(G_2^{(1)})$}\label{Young wall}
\subsection{Young columns}
We shall cut the unit cube in three ways to obtain the blocks of four different shapes. We describe the cutting process graphically as follows.

\begin{figure}[H]
\small
\begin{center}
	\begin{tikzpicture}[scale=0.8]
		\draw (0,0)--(1,0)--(1,1)--(0,1)--(0,0);
		\draw (1,0)--(1.5,0.5)--(1.5,1.5)--(1,1);
		\draw (0,1)--(0.5,1.5)--(1.5,1.5);
		\draw [->] (2,0.8)--(3,0.8);
		\node at (2.5,1.2) {cut};
		\begin{scope}[shift={(3.5,0)}]	
			\draw (0,0)--(1,0)--(1,1)--(0,1)--(0,0);
			\draw (1,0)--(1.5,0.5)--(1.5,1.5)--(1,1);
			\draw (0,1)--(0.5,1.5)--(1.5,1.5);
		\draw (0,0.5)--(0.5,0.5);
		\draw (0.5,0)--(0.5,1)--(1,1.5);
			    \draw [->] (2,2)--(3,2.7);
				\draw [->] (2,1)--(3,1);
				\draw [->] (2,0)--(3,-0.7);
			    \node at (2.6,2) {split};		
				\node at (2.6,0.6) {split};
				\node at (2.6,-1) {split};
		\end{scope}	
		\begin{scope}[shift={(7,2)}]	
			\draw (0,0.5)--(0.5,0.5)--(0.5,1)--(0,1)--(0,0.5);
			\draw (0.5,0.5)--(1,1)--(1,1.5)--(0.5,1);
			\draw (0,1)--(0.5,1.5)--(1,1.5);
			
		\end{scope}	
	\begin{scope}[shift={(7,0.5)}]	
	\draw (0,0)--(0.5,0)--(0.5,0.5)--(0,0.5)--(0,0);
	\draw (0.5,0)--(1,0.5)--(1,1)--(0.5,0.5);
	\draw (0,0.5)--(0.5,1)--(1,1);

	\end{scope}	
		\begin{scope}[shift={(6.5,-2)}]	
		\draw (0.5,0)--(1,0)--(1,1)--(0.5,1)--(0.5,0);
		\draw (1,0)--(1.5,0.5)--(1.5,1.5)--(1,1);
		\draw (0.5,1)--(1,1.5)--(1.5,1.5);
		\node at (3.5,0.5) {$\mathbf O$-block};
	\end{scope}	
	\end{tikzpicture}
\end{center}

\begin{center}
	\begin{tikzpicture}[scale=0.7]
	\node at (2.5,1.2) {};
	\end{tikzpicture}
	\end{center} 
%
\begin{center}
	\begin{tikzpicture}[scale=0.8]
		\draw (0,0)--(1,0)--(1,1)--(0,1)--(0,0);
		\draw (1,0)--(1.5,0.5)--(1.5,1.5)--(1,1);
		\draw (0,1)--(0.5,1.5)--(1.5,1.5);
		\draw [->] (2,0.8)--(3,0.8);
		\node at (2.5,1.2) {cut};
		\begin{scope}[shift={(3.5,0)}]	
		\draw (0,0)--(1,0)--(1,1)--(0,1)--(0,0);
		\draw (1,0)--(1.5,0.5)--(1.5,1.5)--(1,1);
		\draw (0,1)--(0.5,1.5)--(1.5,1.5);
		\draw (0.25,1.25)--(0.75,1.25);
		\draw (1,1.5)--(0.5,1)--(0.5,0.5)--(1,0.5)--(1.25,0.75)--(1.25,0.25);
			\draw [->] (2,2)--(3,2.7);
			\draw [->] (2,1)--(3,1);
			\draw [->] (2,0)--(3,-0.7);
			\node at (2.6,2) {split};		
			\node at (2.6,0.6) {split};
			\node at (2.6,-1) {split};
		\end{scope}	
		\begin{scope}[shift={(7,2)}]	
		\draw (0.25,0.25)--(0.75,0.25)--(0.75,1.25)--(0.25,1.25)--(0.25,0.25);
		\draw (0.75,0.25)--(1,0.5)--(1,1.5)--(0.75,1.25);
		\draw (0.25,1.25)--(0.5,1.5)--(1,1.5);
			\node at (3,0.85) {$\mathbf I$-block};
		\end{scope}	
		\begin{scope}[shift={(7,0.5)}]	
		\draw (0,0)--(1,0)--(1,0.5)--(0.5,0.5)--(0.5,1)--(0,1)--(0,0);
		\draw (0,1)--(0.25,1.25)--(0.75,1.25)--(0.5,1);
		\draw (0.5,0.5)--(0.75,0.75)--(0.75,1.25);
		\draw (0.75,0.75)--(1.25,0.75)--(1,0.5);
		\draw (1.25,0.75)--(1.25,0.25)--(1,0);
			\node at (3,0.5) {$\mathbf L$-block};
		\end{scope}	
		\begin{scope}[shift={(6.5,-2)}]	
		\draw (0.5,0.5)--(1,0.5)--(1,1)--(0.5,1)--(0.5,0.5);
		\draw (0.5,1)--(1,1.5)--(1.5,1.5)--(1,1);
		\draw (1.5,1.5)--(1.5,0.5)--(1.25,0.25)--(0.75,0.25)--(0.75,0.5);
		\draw (1,0.5)--(1.25,0.75)--(1.25,0.25);
		\end{scope}	
	\end{tikzpicture}
\end{center}

\begin{center}
	\begin{tikzpicture}[scale=0.7]
		\node at (2.5,1.2) {};
	\end{tikzpicture}
\end{center} 

\begin{center}
	\begin{tikzpicture}[scale=0.8]
		\draw (0,0)--(1,0)--(1,1)--(0,1)--(0,0);
		\draw (1,0)--(1.5,0.5)--(1.5,1.5)--(1,1);
		\draw (0,1)--(0.5,1.5)--(1.5,1.5);
		\draw [->] (2,0.8)--(3,0.8);
		\node at (2.5,1.2) {cut};
		\begin{scope}[shift={(3.5,0)}]	
			\draw (0,0)--(1,0)--(1,1)--(0,1)--(0,0);
			\draw (1,0)--(1.5,0.5)--(1.5,1.5)--(1,1);
			\draw (0,1)--(0.5,1.5)--(1.5,1.5);
			\draw (0.25,1.25)--(0.75,1.25)--(0.5,1)--(0.5,0);
			\draw [->] (2,1.2)--(3,1.7);
			\node at (2.6,1.1) {split};
			\draw [->] (2,0.4)--(3,-0.1);
			\node at (2.6,-0.4) {split};
		\end{scope}	
		\begin{scope}[shift={(7,1.2)}]	
			\draw (0.5,1.5)--(0.25,1.25)--(0.75,1.25)--(0.5,1)--(1,1)--(1.5,1.5)--(0.5,1.5);
		\draw (0.25,1.25)--(0.25,0.25)--(0.5,0.25);
		\draw (0.5,1)--(0.5,0)--(1,0)--(1,1);
		\draw (1,0)--(1.5,0.5)--(1.5,1.5);
		\node at (3,0.75) {$\mathbf{LL}$-block};
		\end{scope}	
		\begin{scope}[shift={(7.3,-1.2)}]	
				\draw (0,0)--(0.5,0)--(0.5,1)--(0,1)--(0,0);
			\draw (0,1)--(0.25,1.25)--(0.75,1.25)--(0.5,1);
			\draw (0.75,1.25)--(0.75,0.25)--(0.5,0);		
				\end{scope}	
		
	\end{tikzpicture}
\end{center}
\caption{The cutting process of the unit cube}
\end{figure}
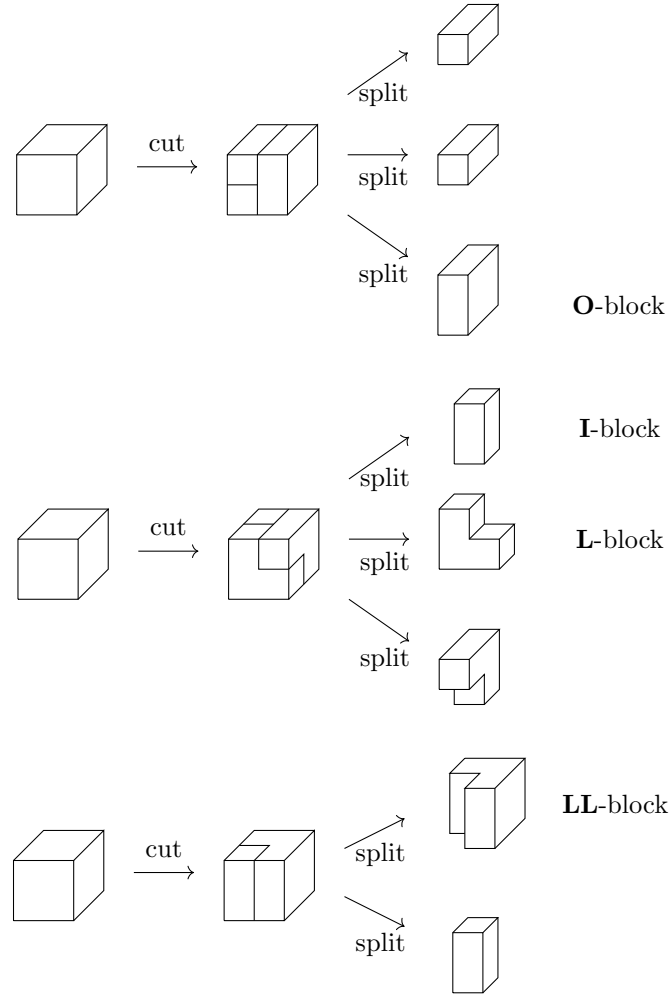

Using the first way of cutting, we obtain two $1/4$ parts and one $1/2$ part of the unit cube. Using the second way of cutting, we obtain one $1/4$ parts and two $3/8$ parts of the unit cube. Using the third way of cutting, we obtain one $3/4$ parts and one $1/4$ parts of the unit cube.

We shall use $\mathbf O$-block, $\mathbf I$-block, $\mathbf L$-block and $\mathbf{LL}$-block in the above picture to construct Young columns and Young walls.

We will use the following colored blocks.
\begin{figure}[H] 
	{\small
		\centering
	\begin{tikzpicture}[scale=0.8]
\begin{scope}[shift={(0,0)}]	
	\draw (0,0)--(0.5,0)--(0.5,1)--(0,1)--(0,0);
	\draw (0.5,0)--(0.75,0.25)--(0.75,1.25)--(0.25,1.25)--(0,1);
	\draw (0.5,1)--(0.75,1.25);
		\node at (0.25,0.5){$0$};
\end{scope}			
		
\begin{scope}[shift={(2,0)}]	
	\draw (0,0)--(0.5,0)--(0.5,1)--(0,1)--(0,0);
	\draw (0.5,0)--(1,0.5)--(1,1.5)--(0.5,1);
	\draw (0,1)--(0.5,1.5)--(1,1.5);
			\node at (0.25,0.5){$1$};
\end{scope}		

\begin{scope}[shift={(4,0)}]	
	\draw (0,0)--(0.5,0)--(0.5,1)--(0,1)--(0,0);
\draw (0.5,0)--(0.75,0.25)--(0.75,1.25)--(0.25,1.25)--(0,1);
\draw (0.5,1)--(0.75,1.25);
\node at (0.25,0.5){$1$};
\end{scope}		

\begin{scope}[shift={(6,0)}]	
	\draw (0,0)--(0.5,0)--(0.5,1)--(0,1)--(0,0);
	\draw (0.5,0)--(0.75,0.25)--(0.75,1.25)--(0.25,1.25)--(0,1);
	\draw (0.5,1)--(0.75,1.25);
	\node at (0.25,0.5){$2$};
\end{scope}		

\begin{scope}[shift={(8,0)}]	
\draw (0,0)--(1,0)--(1,0.5)--(0.5,0.5)--(0.5,1)--(0,1)--(0,0);
\draw (0,1)--(0.25,1.25)--(0.75,1.25)--(0.5,1);
\draw (0.5,0.5)--(0.75,0.75)--(0.75,1.25);
\draw (0.75,0.75)--(1.25,0.75)--(1,0.5);
\draw (1.25,0.75)--(1.25,0.25)--(1,0);
	\node at (0.25,0.5){$2$};
\end{scope}	

	\begin{scope}[shift={(10,0)}]	
	\draw (0.5,1.5)--(0.25,1.25)--(0.75,1.25)--(0.5,1)--(1,1)--(1.5,1.5)--(0.5,1.5);
	\draw (0.25,1.25)--(0.25,0.25)--(0.5,0.25);
	\draw (0.5,1)--(0.5,0)--(1,0)--(1,1);
	\draw (1,0)--(1.5,0.5)--(1.5,1.5);
	\node at (0.75,0.5){$2$};
\end{scope}		
	\end{tikzpicture}
	}
\caption{The colored blocks}
\end{figure}
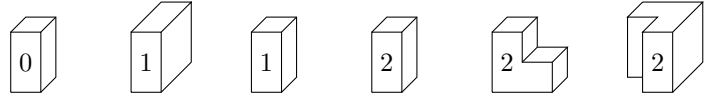 

The block with color $i$ is called an $i$-block.

We denote by

{\small
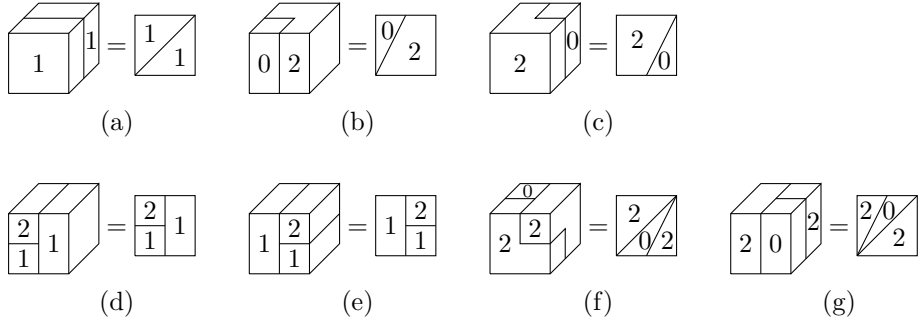
	\begin{center}
		\begin{tikzpicture}[scale=0.8]
			\draw (0,0)--(1,0)--(1,1)--(0,1)--(0,0);
			\draw (1,0)--(1.5,0.5)--(1.5,1.5)--(1,1);
			\draw (0,1)--(0.5,1.5)--(1.5,1.5);
			\draw (0.25,1.25)--(1.25,1.25)--(1.25,0.25);
			\node at (1.8,0.75) {=};
			\node at (1.8,-0.5) {(a)};
			\node at (0.5,0.5) {$1$};
			\node at (1.375,0.875) {$1$};
			\begin{scope}[shift={(2.1,0.3)}]		
				\draw (0,0)--(1,0)--(1,1)--(0,1)--(0,0);
				\draw (0,0)--(1,1);
				\node at (0.75,0.3) {$1$};
				\node at (0.25,0.7) {$1$};
			\end{scope}	
			\begin{scope}[shift={(4,0)}]
				\draw (0,0)--(1,0)--(1,1)--(0,1)--(0,0);
				\draw (1,0)--(1.5,0.5)--(1.5,1.5)--(1,1);
				\draw (0,1)--(0.5,1.5)--(1.5,1.5);
				\draw (0.25,1.25)--(0.75,1.25)--(0.5,1)--(0.5,0);
				\node at (0.25,0.5) {$0$};
				\node at (0.75,0.5) {$2$};
				\node at (1.8,0.75) {=};
				\node at (1.8,-0.5) {(b)};
				\begin{scope}[shift={(2.1,0.3)}]		
					\draw (0,0)--(1,0)--(1,1)--(0,1)--(0,0);
					\draw (0,0)--(0.5,1);
					\node at (0.2,0.75) {$0$};
					\node at (0.65,0.4) {$2$};
				\end{scope}	
			\end{scope}
			\begin{scope}[shift={(8,0)}]
				\draw (0,0)--(1,0)--(1,1)--(0,1)--(0,0);
				\draw (1,0)--(1.5,0.5)--(1.5,1.5)--(1,1);
				\draw (0,1)--(0.5,1.5)--(1.5,1.5);
				\draw (1,1.5)--(0.75,1.25)--(1.25,1.25)--(1.25,0.25);
				\node at (1.8,0.75) {=};
				\node at (1.8,-0.5) {(c)};
				\node at (0.5,0.5) {$2$};
				\node at (1.375,0.875) {$0$};
				\begin{scope}[shift={(2.1,0.3)}]		
					\draw (0,0)--(1,0)--(1,1)--(0,1)--(0,0);
					\draw (0.5,0)--(1,1);
					\node at (0.8,0.25) {$0$};
					\node at (0.35,0.6) {$2$};
				\end{scope}	
			\end{scope}
		
\begin{scope}[shift={(0,-3)}]	
\begin{scope}[shift={(0,0)}]
	\draw (0,0)--(1,0)--(1,1)--(0,1)--(0,0);
	\draw (1,0)--(1.5,0.5)--(1.5,1.5)--(1,1);
	\draw (0,1)--(0.5,1.5)--(1.5,1.5);
	\draw (0,0.5)--(0.5,0.5);
	\draw (0.5,0)--(0.5,1)--(1,1.5);
	\node at (1.8,0.75) {=};
	\node at (1.8,-0.5) {(d)};
	\node at (0.25,0.75) {$2$};
	\node at (0.25,0.25) {$1$};
	\node at (0.75,0.5) {$1$};
	\begin{scope}[shift={(2.1,0.3)}]		
		\draw (0,0)--(1,0)--(1,1)--(0,1)--(0,0);
		\draw (0,0.5)--(0.5,0.5);
		\draw (0.5,0)--(0.5,1);
		\node at (0.25,0.75) {$2$};
		\node at (0.25,0.25) {$1$};
		\node at (0.75,0.5) {$1$};
	\end{scope}	
\end{scope}	
\begin{scope}[shift={(4,0)}]
	\draw (0,0)--(1,0)--(1,1)--(0,1)--(0,0);
	\draw (1,0)--(1.5,0.5)--(1.5,1.5)--(1,1);
	\draw (0,1)--(0.5,1.5)--(1.5,1.5);
	\draw (0.5,0.5)--(1,0.5); \draw (1,0.5)--(1.5,1);
	\draw (0.5,0)--(0.5,1)--(1,1.5);
	\node at (1.8,0.75) {=};
	\node at (1.8,-0.5) {(e)};
	\node at (0.75,0.75) {$2$};
	\node at (0.75,0.25) {$1$};
	\node at (0.25,0.5) {$1$};
	\begin{scope}[shift={(2.1,0.3)}]		
		\draw (0,0)--(1,0)--(1,1)--(0,1)--(0,0);
		\draw (0.5,0.5)--(1,0.5);
		\draw (0.5,0)--(0.5,1);
		\node at (0.75,0.75) {$2$};
		\node at (0.75,0.25) {$1$};
		\node at (0.25,0.5) {$1$};
	\end{scope}	
\end{scope}
\begin{scope}[shift={(8,0)}]
	\draw (0,0)--(1,0)--(1,1)--(0,1)--(0,0);
	\draw (1,0)--(1.5,0.5)--(1.5,1.5)--(1,1);
	\draw (0,1)--(0.5,1.5)--(1.5,1.5);
	\draw (0.25,1.25)--(0.75,1.25);
	\draw (1,1.5)--(0.5,1)--(0.5,0.5)--(1,0.5)--(1.25,0.75)--(1.25,0.25);
	\node at (1.8,0.75) {=};
	\node at (0.25,0.5){$2$}; 	\node at (0.75,0.75){$2$};\node at (0.625,1.375){\tiny{$0$}};	\node at (1.8,-0.5) {(f)};
	\begin{scope}[shift={(2.1,0.3)}]		
		\draw (0,0)--(1,0)--(1,1)--(0,1)--(0,0);
		\draw (0.5,0)--(1,1);\draw (0,0)--(1,1);
		\node at (0.85,0.25) {$2$};	\node at (0.48,0.25) {$0$};
		\node at (0.3,0.7) {$2$};
	\end{scope}	
\end{scope}
\begin{scope}[shift={(12,0)}]
	\draw (0,0)--(1,0)--(1,1)--(0,1)--(0,0);
	\draw (1,0)--(1.5,0.5)--(1.5,1.5)--(1,1);
	\draw (0,1)--(0.5,1.5)--(1.5,1.5);
	\draw (0.5,0)--(0.5,1)--(1,1.5);
	\node at (1.8,0.75) {=};
	\draw (0.75,1.25)--(1.25,1.25);
	\draw (1.25,1.25)--(1.25,0.25);
	\node at (0.25,0.5) {$2$};
	\node at (0.75,0.5) {$0$};
	\node at (1.375,0.875) {$2$};
	\node at (1.8,-0.5) {(g)};
	\begin{scope}[shift={(2.1,0.3)}]		
		\draw (0,0)--(1,0)--(1,1)--(0,1)--(0,0);
		\draw (0,0)--(0.5,1);\draw (0,0)--(1,1);
		\node at (0.15,0.75) {$2$};	\node at (0.53,0.75) {$0$};
		\node at (0.7,0.3) {$2$};
	\end{scope}	
\end{scope}
\end{scope}		
		\end{tikzpicture}
\end{center}}

In the above picture, for (b) and (c), we can visualize that one of them can be obtained by rotating another one $180^{\circ}$ horizontally. Similarly, for (d) (resp. (f)) and (e) (resp. (g)), we can visualize that one of them can be obtained by rotating another one $180^{\circ}$ horizontally.

Now we introduce the {\it column pattern} for $U_q(D_4^{(3)})$ and $U_q(G_2^{(1)})$ as follows (see Figure~\ref{column pattern}).

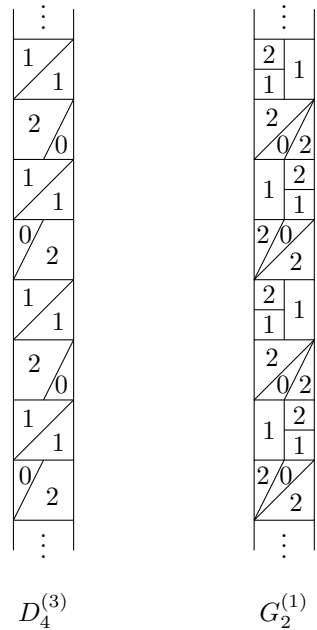
\begin{figure}[H]
{\small 
\begin{center}
	\begin{tikzpicture}[scale=0.8]
		
\begin{scope}[shift={(-2,0)}]
\node at (0.5,-1.5) {$D_4^{(3)}$};	
	\draw (0,0)--(0,-0.5);
\draw (1,0)--(1,-0.5);
\node at (0.5,-0.3) {$\vdots$};
\draw (0,0)--(1,0)--(1,1)--(0,1)--(0,0);
\draw (0,0)--(0.5,1);
\node at (0.2,0.75) {$0$};
\node at (0.65,0.4) {$2$};	
\draw (1,1)--(1,2)--(0,2)--(0,1);
\draw (0,1)--(1,2);
\node at (0.75,1.3) {$1$};
\node at (0.25,1.7) {$1$};
\draw (1,2)--(1,3)--(0,3)--(0,2);
\draw (0.5,2)--(1,3);
\node at (0.8,2.25) {$0$};
\node at (0.35,2.6) {$2$};
\draw (1,3)--(1,4)--(0,4)--(0,3);
\draw (0,3)--(1,4);
\node at (0.75,3.3) {$1$};
\node at (0.25,3.7) {$1$};
\begin{scope}[shift={(0,4)}]
	\draw (0,4)--(0,4.5);
	\draw (1,4)--(1,4.5);
	\node at (0.5,4.5) {$\vdots$};
	\draw (1,0)--(1,1)--(0,1)--(0,0);
	\draw (0,0)--(0.5,1);
	\node at (0.2,0.75) {$0$};
	\node at (0.65,0.4) {$2$};	
	\draw (1,1)--(1,2)--(0,2)--(0,1);
	\draw (0,1)--(1,2);
	\node at (0.75,1.3) {$1$};
	\node at (0.25,1.7) {$1$};
	\draw (1,2)--(1,3)--(0,3)--(0,2);
	\draw (0.5,2)--(1,3);
	\node at (0.8,2.25) {$0$};
	\node at (0.35,2.6) {$2$};
	\draw (1,3)--(1,4)--(0,4)--(0,3);
	\draw (0,3)--(1,4);
	\node at (0.75,3.3) {$1$};
	\node at (0.25,3.7) {$1$};		
\end{scope}
\end{scope}		

\begin{scope}[shift={(2,0)}]
	\node at (0.5,-1.5) {$G_2^{(1)}$};			
		\draw (0,0)--(0,-0.5);
		\draw (1,0)--(1,-0.5);
		\node at (0.5,-0.3) {$\vdots$};

	\draw (0,0)--(1,0)--(1,1)--(0,1)--(0,0);
	\draw (0,0)--(0.5,1);\draw (0,0)--(1,1);

	\draw (1,1)--(1,2)--(0,2)--(0,1);
	\draw (0.5,1.5)--(1,1.5);\draw (0.5,1)--(0.5,2);

		\draw (1,2)--(1,3)--(0,3)--(0,2);
	\draw (0.5,2)--(1,3);	\draw (0,2)--(1,3);

	\draw (1,3)--(1,4)--(0,4)--(0,3);
	\draw (0.5,3)--(0.5,4);\draw (0,3.5)--(0.5,3.5);
	\node at (0.15,0.75) {$2$};	
	\node at (0.53,0.75) {$0$};
\node at (0.7,0.3) {$2$};
\node at (0.75,1.75) {$2$};
\node at (0.75,1.25) {$1$};
\node at (0.25,1.5) {$1$};
	\node at (0.85,2.25) {$2$};	
	\node at (0.48,2.25) {$0$};
\node at (0.3,2.7) {$2$};
\node at (0.25,3.75) {$2$};
\node at (0.25,3.25) {$1$};
\node at (0.75,3.5) {$1$};
	\begin{scope}[shift={(0,4)}]
		\draw (0,4)--(0,4.5);
		\draw (1,4)--(1,4.5);
		\node at (0.5,4.5) {$\vdots$};
		\draw (1,0)--(1,1)--(0,1)--(0,0);
		\draw (0,0)--(0.5,1);\draw (0,0)--(1,1);
		
		\draw (1,1)--(1,2)--(0,2)--(0,1);
		\draw (0.5,1.5)--(1,1.5);\draw (0.5,1)--(0.5,2);
		
		\draw (1,2)--(1,3)--(0,3)--(0,2);
		\draw (0.5,2)--(1,3);	\draw (0,2)--(1,3);
		
		\draw (1,3)--(1,4)--(0,4)--(0,3);
		\draw (0.5,3)--(0.5,4);\draw (0,3.5)--(0.5,3.5);
			\node at (0.15,0.75) {$2$};	
		\node at (0.53,0.75) {$0$};
		\node at (0.7,0.3) {$2$};
		\node at (0.75,1.75) {$2$};
		\node at (0.75,1.25) {$1$};
		\node at (0.25,1.5) {$1$};
		\node at (0.85,2.25) {$2$};	
		\node at (0.48,2.25) {$0$};
		\node at (0.3,2.7) {$2$};
		\node at (0.25,3.75) {$2$};
		\node at (0.25,3.25) {$1$};
		\node at (0.75,3.5) {$1$};
		\end{scope}
	\end{scope}
\end{tikzpicture}
\end{center}}
\vskip -1pc
\caption{The column pattern}\label{column pattern}
\end{figure}

\begin{defn}
A {\it Young column} is a continuous part of the column  pattern satisfying the following conditions.
\begin{enumerate}
	\item[(1)] The continuous part can extend to infinity downwards, but cannot extend to infinity upward.
	\item[(2)] If there is only one 1-block at the top of this continuous part, then this 1-block must cover a 0-block.
	\item[(3)] There should be no free space below every block.
\end{enumerate}

\end{defn}

\begin{defn}
	Let $y$ and $y'$ be Young columns. We say that they are \textit{equivalent} if $y'$ can be obtained from $y$ by rotating the whole column ${180}^\circ$ horizontally.
\end{defn}

\begin{defn}
	\begin{enumerate}
		\item 	An $i$-block in a Young column is called \textit{removable} if we can remove the $i$-block and the resulting column is still a Young column.
		\item	A place in a Young column is called an  \textit{admissible} $i$-slot if we can add an $i$-block on this place and the resulting column is still a Young column.
		\item An $i$-block in a Young column is called \textit{second removable} if we need to remove another $i$-block before we can remove it. 
		\item A place in a Young column is called a \textit{second admissible} $i$-slot if we need to add an $i$-block at an admissible $i$-slot before we can add an $i$-block at this place.	
	\end{enumerate}	
\end{defn}

For a Young column $y$, we define the \textit{$i$-signature} of $y$ by $(\underbrace{-,\cdots,-}_r,\underbrace{+,\cdots,+}_a)$, where $r$ is the number of removable and second removable $i$-blocks in $y$ and $a$ is the number of admissible and second admissible $i$-slots in $y$. The $i$-signature of $y$ is denoted by $\text{sign}_i(y)$.

For $U_q(D_4^{(3)})$, we list all possible Young columns and the signature of Young columns as follows.

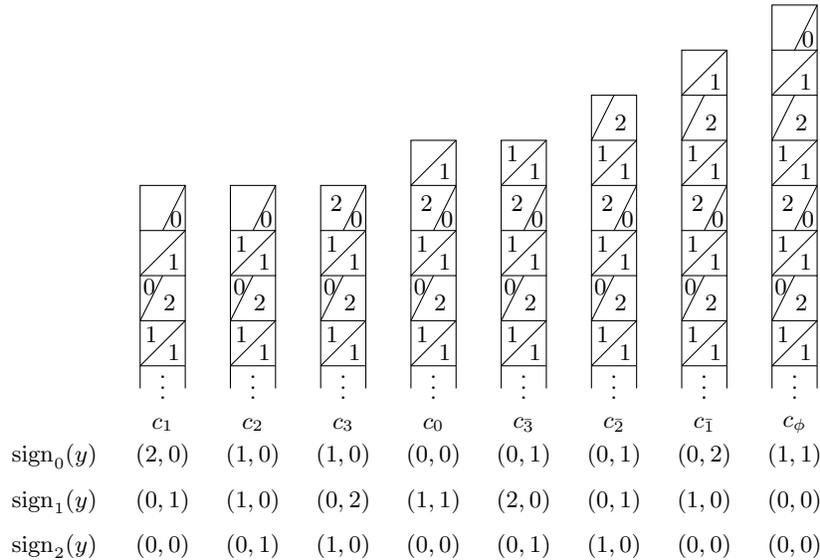
\begin{figure}[H]
	{\footnotesize
		\centering
		\begin{tikzpicture}[scale=0.6]
			\begin{scope}[shift={(-0.9,0)}]
				\node at (-1,-3) {$\text{sign}_0(y)$};
				\node at (-1,-4) {$\text{sign}_1(y)$};
				\node at (-1,-5) {$\text{sign}_2(y)$};
			\end{scope}			
			\draw (0,-1)--(0,-1.5);
			\draw (1,-1)--(1,-1.5);
			\node at (0.5,-1.3) {$\vdots$};
			\node at (0.5,-3) {$(2,0)$};
			\node at (0.5,-4) {$(0,1)$};
			\node at (0.5,-5) {$(0,0)$};
			\draw (0,0)--(0,-1);
			\draw (1,0)--(1,-1);
			\draw (0,-1)--(1,-1);
			\draw (0,-1)--(1,0);
			\node at (0.75,-0.7) {$1$};
			\node at (0.25,-0.3) {$1$};
			\draw (0,0)--(1,0)--(1,1)--(0,1)--(0,0);
			\draw (0,0)--(0.5,1);
			\node at (0.65,0.4) {$2$};	
			\draw (1,1)--(1,2)--(0,2)--(0,1);
			\draw (0,1)--(1,2);
			\node at (0.75,1.3) {$1$};
			\draw (1,2)--(1,3)--(0,3)--(0,2);
			\draw (0.5,2)--(1,3);
			\node at (0.8,2.25) {$0$};
			\node at (0.2,0.75) {$0$};
			\node at (0.5,-2.3) {$c_1$};
			\begin{scope}[shift={(2,0)}]
				\draw (0,-1)--(0,-1.5);
				\draw (1,-1)--(1,-1.5);
				\node at (0.5,-1.3) {$\vdots$};
				\node at (0.5,-3) {$(1,0)$};
				\node at (0.5,-4) {$(1,0)$};
				\node at (0.5,-5) {$(0,1)$};
				\draw (0,0)--(0,-1);
				\draw (1,0)--(1,-1);
				\draw (0,-1)--(1,-1);
				\draw (0,-1)--(1,0);
				\node at (0.75,-0.7) {$1$};
				\node at (0.25,-0.3) {$1$};
				\draw (0,0)--(1,0)--(1,1)--(0,1)--(0,0);
				\draw (0,0)--(0.5,1);
				\node at (0.65,0.4) {$2$};	
				\draw (1,1)--(1,2)--(0,2)--(0,1);
				\draw (0,1)--(1,2);
				\node at (0.75,1.3) {$1$};
				\node at (0.25,1.7) {$1$};
				\draw (1,2)--(1,3)--(0,3)--(0,2);
				\draw (0.5,2)--(1,3);
				\node at (0.8,2.25) {$0$};	
				\node at (0.2,0.75) {$0$};
				\node at (0.5,-2.3) {$c_2$};
			\end{scope}
			\begin{scope}[shift={(4,0)}]
				\draw (0,-1)--(0,-1.5);
				\draw (1,-1)--(1,-1.5);
				\node at (0.5,-1.3) {$\vdots$};
				\node at (0.5,-3) {$(1,0)$};
				\node at (0.5,-4) {$(0,2)$};
				\node at (0.5,-5) {$(1,0)$};
				\draw (0,0)--(0,-1);
				\draw (1,0)--(1,-1);
				\draw (0,-1)--(1,-1);
				\draw (0,-1)--(1,0);
				\node at (0.75,-0.7) {$1$};
				\node at (0.25,-0.3) {$1$};
				\draw (0,0)--(1,0)--(1,1)--(0,1)--(0,0);
				\draw (0,0)--(0.5,1);
				\node at (0.65,0.4) {$2$};	
				\draw (1,1)--(1,2)--(0,2)--(0,1);
				\draw (0,1)--(1,2);
				\node at (0.75,1.3) {$1$};
				\node at (0.25,1.7) {$1$};
				\draw (1,2)--(1,3)--(0,3)--(0,2);
				\draw (0.5,2)--(1,3);
				\node at (0.8,2.25) {$0$};	
				\node at (0.2,0.75) {$0$};
				\node at (0.35,2.6) {$2$};
				\node at (0.5,-2.3) {$c_3$};
			\end{scope}
			\begin{scope}[shift={(6,0)}]
				\draw (0,-1)--(0,-1.5);
				\draw (1,-1)--(1,-1.5);
				\node at (0.5,-1.3) {$\vdots$};
				\node at (0.5,-3) {$(0,0)$};
				\node at (0.5,-4) {$(1,1)$};
				\node at (0.5,-5) {$(0,0)$};
				\draw (0,0)--(0,-1);
				\draw (1,0)--(1,-1);
				\draw (0,-1)--(1,-1);
				\draw (0,-1)--(1,0);
				\node at (0.75,-0.7) {$1$};
				\node at (0.25,-0.3) {$1$};
				\draw (0,0)--(1,0)--(1,1)--(0,1)--(0,0);
				\draw (0,0)--(0.5,1);
				\node at (0.65,0.4) {$2$};	
				\draw (1,1)--(1,2)--(0,2)--(0,1);
				\draw (0,1)--(1,2);
				\node at (0.75,1.3) {$1$};
				\node at (0.25,1.7) {$1$};
				\draw (1,2)--(1,3)--(0,3)--(0,2);
				\draw (1,3)--(1,4)--(0,4)--(0,3);
				\draw (0,3)--(1,4);
				\draw (0.5,2)--(1,3);
				\node at (0.75,3.3) {$1$};
				
				\node at (0.8,2.25) {$0$};	
				\node at (0.2,0.75) {$0$};
				\node at (0.35,2.6) {$2$};
				\node at (0.5,-2.3) {$c_0$};
			\end{scope}
			\begin{scope}[shift={(8,0)}]
				\draw (0,-1)--(0,-1.5);
				\draw (1,-1)--(1,-1.5);
				\node at (0.5,-1.3) {$\vdots$};
				\node at (0.5,-3) {$(0,1)$};
				\node at (0.5,-4) {$(2,0)$};
				\node at (0.5,-5) {$(0,1)$};
				\draw (0,0)--(0,-1);
				\draw (1,0)--(1,-1);
				\draw (0,-1)--(1,-1);
				\draw (0,-1)--(1,0);
				\node at (0.75,-0.7) {$1$};
				\node at (0.25,-0.3) {$1$};
				\draw (0,0)--(1,0)--(1,1)--(0,1)--(0,0);
				\draw (0,0)--(0.5,1);
				\node at (0.65,0.4) {$2$};	
				\draw (1,1)--(1,2)--(0,2)--(0,1);
				\draw (0,1)--(1,2);
				\node at (0.75,1.3) {$1$};
				\node at (0.25,1.7) {$1$};
				\draw (1,2)--(1,3)--(0,3)--(0,2);
				\draw (1,3)--(1,4)--(0,4)--(0,3);
				\draw (0,3)--(1,4);
				\draw (0.5,2)--(1,3);
				\node at (0.75,3.3) {$1$};
				\node at (0.25,3.7) {$1$};
				\node at (0.8,2.25) {$0$};	
				\node at (0.2,0.75) {$0$};
				\node at (0.35,2.6) {$2$};
				\node at (0.5,-2.3) {$c_{\bar{3}}$};
				
			\end{scope}
			\begin{scope}[shift={(10,0)}]
				\draw (0,-1)--(0,-1.5);
				\draw (1,-1)--(1,-1.5);
				\node at (0.5,-1.3) {$\vdots$};
				\node at (0.5,-3) {$(0,1)$};
				\node at (0.5,-4) {$(0,1)$};
				\node at (0.5,-5) {$(1,0)$};
				\draw (0,0)--(0,-1);
				\draw (1,0)--(1,-1);
				\draw (0,-1)--(1,-1);
				\draw (0,-1)--(1,0);
				\node at (0.75,-0.7) {$1$};
				\node at (0.25,-0.3) {$1$};
				\draw (0,0)--(1,0)--(1,1)--(0,1)--(0,0);
				\draw (0,0)--(0.5,1);
				\node at (0.65,0.4) {$2$};	
				\draw (1,1)--(1,2)--(0,2)--(0,1);
				\draw (0,1)--(1,2);
				\node at (0.75,1.3) {$1$};
				\node at (0.25,1.7) {$1$};
				\draw (1,2)--(1,3)--(0,3)--(0,2);
				\draw (1,3)--(1,4)--(0,4)--(0,3);
				\draw (0,3)--(1,4);
				\draw (0.5,2)--(1,3);
				\node at (0.75,3.3) {$1$};
				\node at (0.25,3.7) {$1$};
				\node at (0.8,2.25) {$0$};	
				\node at (0.2,0.75) {$0$};
				\node at (0.35,2.6) {$2$};
				\draw (1,4)--(1,5)--(0,5)--(0,4);
				\draw (0,4)--(0.5,5);
				\node at (0.65,4.4) {$2$};
				\node at (0.5,-2.3) {$c_{\bar{2}}$};
				
			\end{scope}
			\begin{scope}[shift={(12,0)}]
				\draw (0,-1)--(0,-1.5);
				\draw (1,-1)--(1,-1.5);
				\node at (0.5,-1.3) {$\vdots$};
				\node at (0.5,-3) {$(0,2)$};
				\node at (0.5,-4) {$(1,0)$};
				\node at (0.5,-5) {$(0,0)$};
				\draw (0,0)--(0,-1);
				\draw (1,0)--(1,-1);
				\draw (0,-1)--(1,-1);
				\draw (0,-1)--(1,0);
				\node at (0.75,-0.7) {$1$};
				\node at (0.25,-0.3) {$1$};
				\draw (0,0)--(1,0)--(1,1)--(0,1)--(0,0);
				\draw (0,0)--(0.5,1);
				\node at (0.65,0.4) {$2$};	
				\draw (1,1)--(1,2)--(0,2)--(0,1);
				\draw (0,1)--(1,2);
				\node at (0.75,1.3) {$1$};
				\node at (0.25,1.7) {$1$};
				\draw (1,2)--(1,3)--(0,3)--(0,2);
				\draw (1,3)--(1,4)--(0,4)--(0,3);
				\draw (0,3)--(1,4);
				\draw (0.5,2)--(1,3);
				\node at (0.75,3.3) {$1$};
				\node at (0.25,3.7) {$1$};
				\node at (0.8,2.25) {$0$};	
				\node at (0.2,0.75) {$0$};
				\node at (0.35,2.6) {$2$};
				\draw (1,4)--(1,5)--(0,5)--(0,4);
				\draw (0,4)--(0.5,5);
				\node at (0.65,4.4) {$2$};
				\draw (1,5)--(1,6)--(0,6)--(0,5);
				\draw (0,5)--(1,6);
				\node at (0.75,5.3) {$1$};
				\node at (0.5,-2.3) {$c_{\bar{1}}$};
			\end{scope}
			
			\begin{scope}[shift={(14,0)}]
				\draw (0,-1)--(0,-1.5);
				\draw (1,-1)--(1,-1.5);
				\node at (0.5,-1.3) {$\vdots$};
				\node at (0.5,-3) {$(1,1)$};
				\node at (0.5,-4) {$(0,0)$};
				\node at (0.5,-5) {$(0,0)$};
				\draw (0,0)--(0,-1);
				\draw (1,0)--(1,-1);
				\draw (0,-1)--(1,-1);
				\draw (0,-1)--(1,0);
				\node at (0.75,-0.7) {$1$};
				\node at (0.25,-0.3) {$1$};
				\draw (0,0)--(1,0)--(1,1)--(0,1)--(0,0);
				\draw (0,0)--(0.5,1);
				\node at (0.65,0.4) {$2$};	
				\draw (1,1)--(1,2)--(0,2)--(0,1);
				\draw (0,1)--(1,2);
				\node at (0.75,1.3) {$1$};
				\node at (0.25,1.7) {$1$};
				\draw (1,2)--(1,3)--(0,3)--(0,2);
				\draw (1,3)--(1,4)--(0,4)--(0,3);
				\draw (0,3)--(1,4);
				\draw (0.5,2)--(1,3);
				\node at (0.75,3.3) {$1$};
				\node at (0.25,3.7) {$1$};
				\node at (0.8,2.25) {$0$};	
				\node at (0.2,0.75) {$0$};
				\node at (0.35,2.6) {$2$};
				\draw (1,4)--(1,5)--(0,5)--(0,4);
				\draw (0,4)--(0.5,5);
				\node at (0.65,4.4) {$2$};
				\draw (1,5)--(1,6)--(0,6)--(0,5);
				\draw (0,5)--(1,6);
				\node at (0.75,5.3) {$1$};
				\draw (1,6)--(1,7)--(0,7)--(0,6);
				\draw (0.5,6)--(1,7);
				\node at (0.8,6.25) {$0$};
				\node at (0.5,-2.3) {$c_{\phi}$};
			\end{scope}
		\end{tikzpicture}
	}
	\vskip -.5pc
	\caption{All possible Young columns and signature of Young columns for $U_q(D_4^{(3)})$} \label{all possible Young columns D43}
	\vskip -1pc
\end{figure}

For $U_q(G_2^{(1)})$, We list all possible Young columns and the signature of Young columns as follows.

\begin{figure}[H] 
	{\small
		\centering
	\begin{tikzpicture}[scale=0.6]
		\begin{scope}[shift={(-0.9,0)}]
			\node at (-0.7,-2) {$\text{sign}_0(y)$};
			\node at (-0.7,-3) {$\text{sign}_1(y)$};
			\node at (-0.7,-4) {$\text{sign}_2(y)$};
		\end{scope}	
		\begin{scope}[shift={(0,0)}]			
			\draw (0,0)--(0,-0.5);
			\draw (1,0)--(1,-0.5);
			\node at (0.5,-0.3) {$\vdots$};
			\node at (0.5,-2) {(2,0)};
			\node at (0.5,-3) {(0,1)};
			\node at (0.5,-4) {(0,0)};
			\draw (0,0)--(0,1)--(1,1)--(1,0)--(0,0); \draw (0,0)--(1,1);\draw (0,0)--(0.5,1);
			\draw (0,1)--(0,2)--(1,2)--(1,1);\draw (0.5,1)--(0.5,2);
			\draw (0,2)--(0,3)--(1,3)--(1,2);\draw (0,2)--(1,3);\draw (0.5,2)--(1,3);
			\node at (0.15,0.75) {$2$};\node at (0.53,0.75) {$0$};\node at (0.48,2.25) {$0$};
			\node at (0.7,0.3) {$2$};
			\node at (0.25,1.5) {$1$};
			\node at (0.5,-1.2) {$c'_1$};
		\end{scope}
		\begin{scope}[shift={(1.5,0)}]
			\draw (0,0)--(0,-0.5);
			\draw (1,0)--(1,-0.5);
			\node at (0.5,-0.3) {$\vdots$};
			\node at (0.5,-2) {(1,0)};
			\node at (0.5,-3) {(1,0)};
			\node at (0.5,-4) {(0,3)};
			\draw (0,0)--(0,1)--(1,1)--(1,0)--(0,0); \draw (0,0)--(1,1);\draw (0,0)--(0.5,1);
			\draw (0,1)--(0,2)--(1,2)--(1,1);\draw (0.5,1)--(0.5,2);
			\draw (0,2)--(0,3)--(1,3)--(1,2);\draw (0,2)--(1,3);\draw (0.5,2)--(1,3);
			\draw (0.5,1.5)--(1,1.5);\node at (0.75,1.25) {$1$};
			\node at (0.15,0.75) {$2$};\node at (0.53,0.75) {$0$};\node at (0.48,2.25) {$0$};
			\node at (0.7,0.3) {$2$};
			\node at (0.25,1.5) {$1$};
			\node at (0.5,-1.2) {$c'_2$};
		\end{scope}
		\begin{scope}[shift={(3,0)}]
			\draw (0,0)--(0,-0.5);
			\draw (1,0)--(1,-0.5);
			\node at (0.5,-0.3) {$\vdots$};
			\node at (0.5,-2) {(1,0)};
			\node at (0.5,-3) {(0,0)};
			\node at (0.5,-4) {(1,2)};
			\draw (0,0)--(0,1)--(1,1)--(1,0)--(0,0); \draw (0,0)--(1,1);\draw (0,0)--(0.5,1);
			\draw (0,1)--(0,2)--(1,2)--(1,1);\draw (0.5,1)--(0.5,2);
			\draw (0,2)--(0,3)--(1,3)--(1,2);\draw (0,2)--(1,3);\draw (0.5,2)--(1,3);
			\draw (0.5,1.5)--(1,1.5);\node at (0.75,1.25) {$1$};\node at (0.75,1.75) {$2$};
			\node at (0.15,0.75) {$2$};\node at (0.53,0.75) {$0$};\node at (0.48,2.25) {$0$};
			\node at (0.7,0.3) {$2$};
			\node at (0.25,1.5) {$1$};
			\node at (0.5,-1.2) {$c'_3$};
		\end{scope}
		\begin{scope}[shift={(4.5,0)}]
			\draw (0,0)--(0,-0.5);
			\draw (1,0)--(1,-0.5);
			\node at (0.5,-0.3) {$\vdots$};
			\node at (0.5,-2) {(1,0)};
			\node at (0.5,-3) {(0,1)};
			\node at (0.5,-4) {(2,1)};
			\draw (0,0)--(0,1)--(1,1)--(1,0)--(0,0); \draw (0,0)--(1,1);\draw (0,0)--(0.5,1);
			\draw (0,1)--(0,2)--(1,2)--(1,1);\draw (0.5,1)--(0.5,2);
			\draw (0,2)--(0,3)--(1,3)--(1,2);\draw (0,2)--(1,3);\draw (0.5,2)--(1,3);
			\draw (0.5,1.5)--(1,1.5);\node at (0.75,1.25) {$1$};\node at (0.75,1.75) {$2$};
			\node at (0.15,0.75) {$2$};\node at (0.53,0.75) {$0$};\node at (0.48,2.25) {$0$};
			\node at (0.7,0.3) {$2$};\node at (0.3,2.7) {$2$};
			\node at (0.25,1.5) {$1$};
			\node at (0.5,-1.2) {$c'_4$};
		\end{scope}
		\begin{scope}[shift={(6,0)}]
			\draw (0,0)--(0,-0.5);
			\draw (1,0)--(1,-0.5);
			\node at (0.5,-0.3) {$\vdots$};
			\node at (0.5,-2) {(0,0)};
			\node at (0.5,-3) {(1,0)};
			\node at (0.5,-4) {(0,2)};
			\draw (0,0)--(0,1)--(1,1)--(1,0)--(0,0); \draw (0,0)--(1,1);\draw (0,0)--(0.5,1);
			\draw (0,1)--(0,2)--(1,2)--(1,1);\draw (0.5,1)--(0.5,2);
			\draw (0,2)--(0,3)--(1,3)--(1,2);\draw (0,2)--(1,3);\draw (0.5,2)--(1,3);
			\draw (0.5,1.5)--(1,1.5);\node at (0.75,1.25) {$1$};\node at (0.75,1.75) {$2$};
			\node at (0.15,0.75) {$2$};\node at (0.53,0.75) {$0$};\node at (0.48,2.25) {$0$};
			\node at (0.7,0.3) {$2$};\node at (0.3,2.7) {$2$};
			\node at (0.25,1.5) {$1$};
			\draw (0,3)--(0,4)--(1,4)--(1,3);\draw (0.5,3)--(0.5,4);\draw (0,3.5)--(0.5,3.5);
			\node at (0.25,3.25) {$1$};
			\node at (0.5,-1.2) {$c'_5$};	
		\end{scope}
		\begin{scope}[shift={(7.5,0)}]
			\draw (0,0)--(0,-0.5);
			\draw (1,0)--(1,-0.5);
			\node at (0.5,-0.3) {$\vdots$};
			\node at (0.5,-2) {(1,0)};
			\node at (0.5,-3) {(0,2)};
			\node at (0.5,-4) {(3,0)};
			\draw (0,0)--(0,1)--(1,1)--(1,0)--(0,0); \draw (0,0)--(1,1);\draw (0,0)--(0.5,1);
			\draw (0,1)--(0,2)--(1,2)--(1,1);\draw (0.5,1)--(0.5,2);
			\draw (0,2)--(0,3)--(1,3)--(1,2);\draw (0,2)--(1,3);\draw (0.5,2)--(1,3);
			\draw (0.5,1.5)--(1,1.5);\node at (0.75,1.25) {$1$};\node at (0.75,1.75) {$2$};
			\node at (0.15,0.75) {$2$};\node at (0.53,0.75) {$0$};\node at (0.48,2.25) {$0$};
			\node at (0.7,0.3) {$2$};\node at (0.3,2.7) {$2$};
			\node at (0.25,1.5) {$1$};
			\node at (0.85,2.25) {$2$};
			\node at (0.5,-1.2) {$c'_6$};	
			
		\end{scope}
		\begin{scope}[shift={(9,0)}]
			\draw (0,0)--(0,-0.5);
			\draw (1,0)--(1,-0.5);
			\node at (0.5,-0.3) {$\vdots$};
			\node at (0.5,-2) {(0,0)};
			\node at (0.5,-3) {(1,1)};
			\node at (0.5,-4) {(0,0)};
			\draw (0,0)--(0,1)--(1,1)--(1,0)--(0,0); \draw (0,0)--(1,1);\draw (0,0)--(0.5,1);
			\draw (0,1)--(0,2)--(1,2)--(1,1);\draw (0.5,1)--(0.5,2);
			\draw (0,2)--(0,3)--(1,3)--(1,2);\draw (0,2)--(1,3);\draw (0.5,2)--(1,3);
			\draw (0.5,1.5)--(1,1.5);\node at (0.75,1.25) {$1$};\node at (0.75,1.75) {$2$};
			\node at (0.15,0.75) {$2$};\node at (0.53,0.75) {$0$};\node at (0.48,2.25) {$0$};
			\node at (0.7,0.3) {$2$};\node at (0.3,2.7) {$2$};
			\node at (0.25,1.5) {$1$};
			\draw (0,3)--(0,4)--(1,4)--(1,3);\draw (0.5,3)--(0.5,4);\draw (0,3.5)--(0.5,3.5);
			\node at (0.25,3.25) {$1$};
			\node at (0.85,2.25) {$2$};
			\node at (0.5,-1.2) {$c'_7$};	
		\end{scope}
		
		\begin{scope}[shift={(10.5,0)}]
			\draw (0,0)--(0,-0.5);
			\draw (1,0)--(1,-0.5);
			\node at (0.5,-0.3) {$\vdots$};
			\node at (0.5,-2) {(0,0)};
			\node at (0.5,-3) {(0,0)};
			\node at (0.5,-4) {(1,1)};
			\draw (0,0)--(0,1)--(1,1)--(1,0)--(0,0); \draw (0,0)--(1,1);\draw (0,0)--(0.5,1);
			\draw (0,1)--(0,2)--(1,2)--(1,1);\draw (0.5,1)--(0.5,2);
			\draw (0,2)--(0,3)--(1,3)--(1,2);\draw (0,2)--(1,3);\draw (0.5,2)--(1,3);
			\draw (0.5,1.5)--(1,1.5);\node at (0.75,1.25) {$1$};\node at (0.75,1.75) {$2$};
			\node at (0.15,0.75) {$2$};\node at (0.53,0.75) {$0$};\node at (0.48,2.25) {$0$};
			\node at (0.7,0.3) {$2$};\node at (0.3,2.7) {$2$};
			\node at (0.25,1.5) {$1$};
			\draw (0,3)--(0,4)--(1,4)--(1,3);\draw (0.5,3)--(0.5,4);\draw (0,3.5)--(0.5,3.5);
			\node at (0.25,3.25) {$1$};\node at (0.25,3.75) {$2$};
			\node at (0.5,-1.2) {$c'_{\bar{7}}$};	
		\end{scope}
		
		\begin{scope}[shift={(12,0)}]
			\draw (0,0)--(0,-0.5);
			\draw (1,0)--(1,-0.5);
			\node at (0.5,-0.3) {$\vdots$};
			\node at (0.5,-2) {(0,1)};
			\node at (0.5,-3) {(2,0)};
			\node at (0.5,-4) {(0,3)};
			\draw (0,0)--(0,1)--(1,1)--(1,0)--(0,0); \draw (0,0)--(1,1);\draw (0,0)--(0.5,1);
			\draw (0,1)--(0,2)--(1,2)--(1,1);\draw (0.5,1)--(0.5,2);
			\draw (0,2)--(0,3)--(1,3)--(1,2);\draw (0,2)--(1,3);\draw (0.5,2)--(1,3);
			\draw (0.5,1.5)--(1,1.5);\node at (0.75,1.25) {$1$};\node at (0.75,1.75) {$2$};
			\node at (0.15,0.75) {$2$};\node at (0.53,0.75) {$0$};\node at (0.48,2.25) {$0$};
			\node at (0.7,0.3) {$2$};\node at (0.3,2.7) {$2$};
			\node at (0.25,1.5) {$1$};
			\draw (0,3)--(0,4)--(1,4)--(1,3);\draw (0.5,3)--(0.5,4);\draw (0,3.5)--(0.5,3.5);
			\node at (0.25,3.25) {$1$};\node at (0.75,3.5) {$1$}; \node at (0.85,2.25) {$2$};
			\node at (0.5,-1.2) {$c'_{\bar{6}}$};	
		\end{scope}	
		
		\begin{scope}[shift={(13.5,0)}]
			\draw (0,0)--(0,-0.5);
			\draw (1,0)--(1,-0.5);
			\node at (0.5,-0.3) {$\vdots$};
			\node at (0.5,-2) {(0,0)};
			\node at (0.5,-3) {(0,1)};
			\node at (0.5,-4) {(2,0)};
			\draw (0,0)--(0,1)--(1,1)--(1,0)--(0,0); \draw (0,0)--(1,1);\draw (0,0)--(0.5,1);
			\draw (0,1)--(0,2)--(1,2)--(1,1);\draw (0.5,1)--(0.5,2);
			\draw (0,2)--(0,3)--(1,3)--(1,2);\draw (0,2)--(1,3);\draw (0.5,2)--(1,3);
			\draw (0.5,1.5)--(1,1.5);\node at (0.75,1.25) {$1$};\node at (0.75,1.75) {$2$};
			\node at (0.15,0.75) {$2$};\node at (0.53,0.75) {$0$};\node at (0.48,2.25) {$0$};
			\node at (0.7,0.3) {$2$};\node at (0.3,2.7) {$2$};
			\node at (0.25,1.5) {$1$};
			\draw (0,3)--(0,4)--(1,4)--(1,3);\draw (0.5,3)--(0.5,4);\draw (0,3.5)--(0.5,3.5);
			\node at (0.25,3.25) {$1$}; \node at (0.25,3.75) {$2$};\node at (0.85,2.25) {$2$};
			\node at (0.5,-1.2) {$c'_{\bar{5}}$};	
		\end{scope}	
		
		\begin{scope}[shift={(15,0)}]
			\draw (0,0)--(0,-0.5);
			\draw (1,0)--(1,-0.5);
			\node at (0.5,-0.3) {$\vdots$};
			\node at (0.5,-2) {(0,1)};
			\node at (0.5,-3) {(1,0)};
			\node at (0.5,-4) {(1,2)};
			\draw (0,0)--(0,1)--(1,1)--(1,0)--(0,0); \draw (0,0)--(1,1);\draw (0,0)--(0.5,1);
			\draw (0,1)--(0,2)--(1,2)--(1,1);\draw (0.5,1)--(0.5,2);
			\draw (0,2)--(0,3)--(1,3)--(1,2);\draw (0,2)--(1,3);\draw (0.5,2)--(1,3);
			\draw (0.5,1.5)--(1,1.5);\node at (0.75,1.25) {$1$};\node at (0.75,1.75) {$2$};
			\node at (0.15,0.75) {$2$};\node at (0.53,0.75) {$0$};\node at (0.48,2.25) {$0$};
			\node at (0.7,0.3) {$2$};\node at (0.3,2.7) {$2$};
			\node at (0.25,1.5) {$1$};
			\draw (0,3)--(0,4)--(1,4)--(1,3);\draw (0.5,3)--(0.5,4);\draw (0,3.5)--(0.5,3.5);
			\node at (0.25,3.25) {$1$}; \node at (0.25,3.75) {$2$};\node at (0.85,2.25) {$2$};
			\node at (0.5,-1.2) {$c'_{\bar{4}}$};\node at (0.75,3.5) {$1$};
		\end{scope}	
		
		\begin{scope}[shift={(16.5,0)}]
			\draw (0,0)--(0,-0.5);
			\draw (1,0)--(1,-0.5);
			\node at (0.5,-0.3) {$\vdots$};
			\node at (0.5,-2) {(0,1)};
			\node at (0.5,-3) {(0,0)};
			\node at (0.5,-4) {(2,1)};
			\draw (0,0)--(0,1)--(1,1)--(1,0)--(0,0); \draw (0,0)--(1,1);\draw (0,0)--(0.5,1);
			\draw (0,1)--(0,2)--(1,2)--(1,1);\draw (0.5,1)--(0.5,2);
			\draw (0,2)--(0,3)--(1,3)--(1,2);\draw (0,2)--(1,3);\draw (0.5,2)--(1,3);
			\draw (0.5,1.5)--(1,1.5);\node at (0.75,1.25) {$1$};\node at (0.75,1.75) {$2$};
			\node at (0.15,0.75) {$2$};\node at (0.53,0.75) {$0$};\node at (0.48,2.25) {$0$};
			\node at (0.7,0.3) {$2$};\node at (0.3,2.7) {$2$};
			\node at (0.25,1.5) {$1$};
			\draw (0,3)--(0,4)--(1,4)--(1,3);\draw (0.5,3)--(0.5,4);\draw (0,3.5)--(0.5,3.5);
			\node at (0.25,3.25) {$1$}; \node at (0.25,3.75) {$2$};\node at (0.85,2.25) {$2$};
			\node at (0.5,-1.2) {$c'_{\bar{3}}$};\node at (0.75,3.5) {$1$};
			\draw (0,4)--(0,5)--(1,5)--(1,4);\draw (0,4)--(1,5);\node at (0.7,4.3) {$2$};
		\end{scope}	
		
		\begin{scope}[shift={(18,0)}]
			\draw (0,0)--(0,-0.5);
			\draw (1,0)--(1,-0.5);
			\node at (0.5,-0.3) {$\vdots$};
			\node at (0.5,-2) {(0,1)};
			\node at (0.5,-3) {(0,1)};
			\node at (0.5,-4) {(3,0)};
			\draw (0,0)--(0,1)--(1,1)--(1,0)--(0,0); \draw (0,0)--(1,1);\draw (0,0)--(0.5,1);
			\draw (0,1)--(0,2)--(1,2)--(1,1);\draw (0.5,1)--(0.5,2);
			\draw (0,2)--(0,3)--(1,3)--(1,2);\draw (0,2)--(1,3);\draw (0.5,2)--(1,3);
			\draw (0.5,1.5)--(1,1.5);\node at (0.75,1.25) {$1$};\node at (0.75,1.75) {$2$};
			\node at (0.15,0.75) {$2$};\node at (0.53,0.75) {$0$};\node at (0.48,2.25) {$0$};
			\node at (0.7,0.3) {$2$};\node at (0.3,2.7) {$2$};
			\node at (0.25,1.5) {$1$};
			\draw (0,3)--(0,4)--(1,4)--(1,3);\draw (0.5,3)--(0.5,4);\draw (0,3.5)--(0.5,3.5);
			\node at (0.25,3.25) {$1$}; \node at (0.25,3.75) {$2$};\node at (0.85,2.25) {$2$};
			\node at (0.5,-1.2) {$c'_{\bar{2}}$};\node at (0.75,3.5) {$1$};
			\draw (0,4)--(0,5)--(1,5)--(1,4);\draw (0,4)--(1,5);\draw (0,4)--(0.5,5);\node at (0.7,4.3) {$2$};	\node at (0.15,4.75) {$2$};
		\end{scope}	
		
		\begin{scope}[shift={(19.5,0)}]
			\draw (0,0)--(0,-0.5);
			\draw (1,0)--(1,-0.5);
			\node at (0.5,-0.3) {$\vdots$};
			\node at (0.5,-2) {(0,2)};
			\node at (0.5,-3) {(1,0)};
			\node at (0.5,-4) {(0,0)};
			\draw (0,0)--(0,1)--(1,1)--(1,0)--(0,0); \draw (0,0)--(1,1);\draw (0,0)--(0.5,1);
			\draw (0,1)--(0,2)--(1,2)--(1,1);\draw (0.5,1)--(0.5,2);
			\draw (0,2)--(0,3)--(1,3)--(1,2);\draw (0,2)--(1,3);\draw (0.5,2)--(1,3);
			\draw (0.5,1.5)--(1,1.5);\node at (0.75,1.25) {$1$};\node at (0.75,1.75) {$2$};
			\node at (0.15,0.75) {$2$};\node at (0.53,0.75) {$0$};\node at (0.48,2.25) {$0$};
			\node at (0.7,0.3) {$2$};\node at (0.3,2.7) {$2$};
			\node at (0.25,1.5) {$1$};
			\draw (0,3)--(0,4)--(1,4)--(1,3);\draw (0.5,3)--(0.5,4);\draw (0,3.5)--(0.5,3.5);
			\node at (0.25,3.25) {$1$}; \node at (0.25,3.75) {$2$};\node at (0.85,2.25) {$2$};
			\node at (0.5,-1.2) {$c'_{\bar{1}}$};\node at (0.75,3.5) {$1$};
			\draw (0,4)--(0,5)--(1,5)--(1,4);\draw (0,4)--(1,5);\draw (0,4)--(0.5,5);\node at (0.7,4.3) {$2$};	\node at (0.15,4.75) {$2$};
			\draw (0,5)--(0,6)--(1,6)--(1,5);	\draw (0.5,5)--(0.5,6);	\draw (0.5,5.5)--(1,5.5); \node at (0.25,5.5) {$1$};
		\end{scope}	
		
		\begin{scope}[shift={(21,0)}]
			\draw (0,0)--(0,-0.5);
			\draw (1,0)--(1,-0.5);
			\node at (0.5,-0.3) {$\vdots$};
			\node at (0.5,-2) {(1,1)};
			\node at (0.5,-3) {(0,0)};
			\node at (0.5,-4) {(0,0)};
			\draw (0,0)--(0,1)--(1,1)--(1,0)--(0,0); \draw (0,0)--(1,1);\draw (0,0)--(0.5,1);
			\draw (0,1)--(0,2)--(1,2)--(1,1);\draw (0.5,1)--(0.5,2);
			\draw (0,2)--(0,3)--(1,3)--(1,2);\draw (0,2)--(1,3);\draw (0.5,2)--(1,3);
			\draw (0.5,1.5)--(1,1.5);\node at (0.75,1.25) {$1$};\node at (0.75,1.75) {$2$};
			\node at (0.15,0.75) {$2$};\node at (0.53,0.75) {$0$};\node at (0.48,2.25) {$0$};
			\node at (0.7,0.3) {$2$};\node at (0.3,2.7) {$2$};
			\node at (0.25,1.5) {$1$};
			\draw (0,3)--(0,4)--(1,4)--(1,3);\draw (0.5,3)--(0.5,4);\draw (0,3.5)--(0.5,3.5);
			\node at (0.25,3.25) {$1$}; \node at (0.25,3.75) {$2$};\node at (0.85,2.25) {$2$};
			\node at (0.5,-1.2) {$c'_0$};\node at (0.75,3.5) {$1$};
			\draw (0,4)--(0,5)--(1,5)--(1,4);\draw (0,4)--(1,5);\draw (0,4)--(0.5,5);\node at (0.7,4.3) {$2$};	\node at (0.15,4.75) {$2$};
			\draw (0,5)--(0,6)--(1,6)--(1,5);	\draw (0.5,5)--(0.5,6);	\draw (0.5,5.5)--(1,5.5); \node at (0.25,5.5) {$1$};
			\draw (0,6)--(0,7)--(1,7)--(1,6);\draw (0,6)--(1,7);\draw (0.5,6)--(1,7);
			\node at (0.48,6.25) {$0$};
		\end{scope}	
	\end{tikzpicture}
	}
\vskip -0.5pc
\caption{All possible Young columns and signature of Young columns for $U_q(G_2^{(1)})$}\label{all possible Young columns G21}
\end{figure}
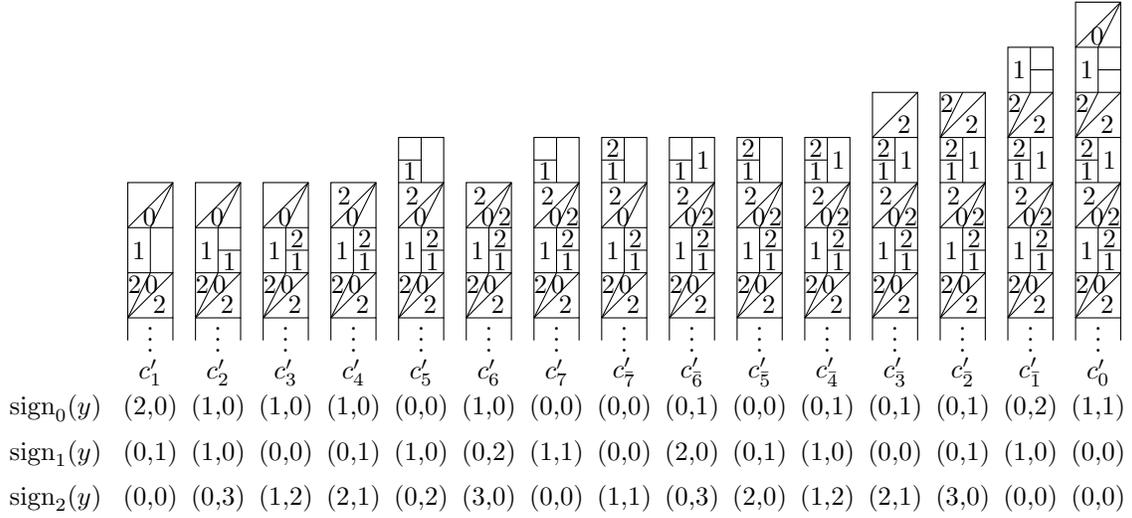

\begin{Rmk}
We need to make some special assumptions for Young column $c'_5$ and Young column $c'_7$.
\begin{enumerate}
	\item[(1)] For Young column $c'_5$, the upper 2-slot is the admissible slot and the lower 2-slot is the twice admissible slot.
	\item[(2)] For Young column $c'_7$, we cannot remove 2-block and cannot add 2-block. 
\end{enumerate} 
\end{Rmk}

\begin{Prop}\label{prop:realization of perfect crystal}
	Using the equivalence classes of Young columns, the perfect crystals $B_1$ and $B'_1$ given in Example \ref{ex:level 1 perfect crystal D43} and Example \ref{ex:level 1 perfect crystal G21} can be realized in terms of Young columns as follows (see Figure~\ref{FIG:realization of B1} and Figure~\ref{FIG:realization of B1'}).
\end{Prop}	
	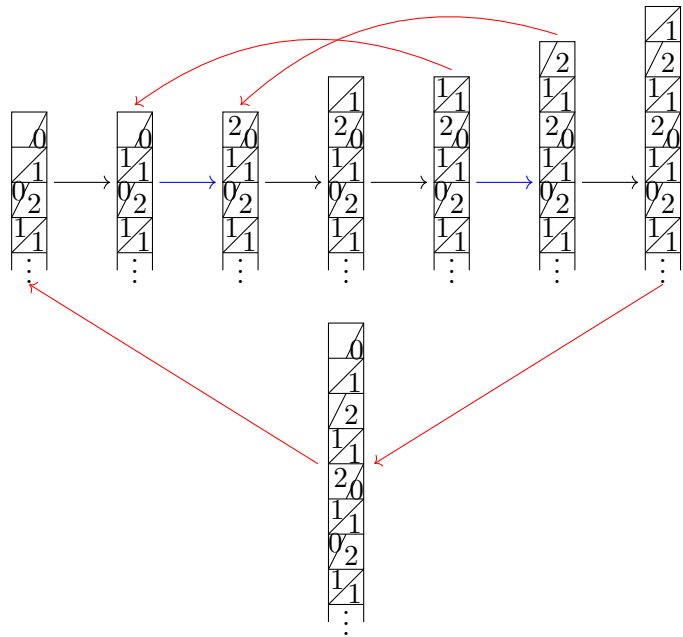
\begin{figure}[H]

			\centering 

			\begin{tikzpicture}[scale=0.468]
				\draw (0,-1)--(0,-1.5);
				\draw (1,-1)--(1,-1.5);
				\node at (0.5,-1.3) {$\vdots$};
				\draw (0,0)--(0,-1);
				\draw (1,0)--(1,-1);
				\draw (0,-1)--(1,-1);
				\draw (0,-1)--(1,0);
				\node at (0.75,-0.7) {$1$};
				\node at (0.25,-0.3) {$1$};
				\draw (0,0)--(1,0)--(1,1)--(0,1)--(0,0);
				\draw (0,0)--(0.5,1);
				\node at (0.65,0.4) {$2$};	
				\draw (1,1)--(1,2)--(0,2)--(0,1);
				\draw (0,1)--(1,2);
				\node at (0.75,1.3) {$1$};
				\draw (1,2)--(1,3)--(0,3)--(0,2);
				\draw (0.5,2)--(1,3);
				\node at (0.8,2.25) {$0$};
				\node at (0.2,0.75) {$0$};
				\draw [black,->] (1.2,1)--(2.8,1);
				\begin{scope}[shift={(3,0)}]
					\draw (0,-1)--(0,-1.5);
					\draw (1,-1)--(1,-1.5);
					\node at (0.5,-1.3) {$\vdots$};
					\draw (0,0)--(0,-1);
					\draw (1,0)--(1,-1);
					\draw (0,-1)--(1,-1);
					\draw (0,-1)--(1,0);
					\node at (0.75,-0.7) {$1$};
					\node at (0.25,-0.3) {$1$};
					\draw (0,0)--(1,0)--(1,1)--(0,1)--(0,0);
					\draw (0,0)--(0.5,1);
					\node at (0.65,0.4) {$2$};	
					\draw (1,1)--(1,2)--(0,2)--(0,1);
					\draw (0,1)--(1,2);
					\node at (0.75,1.3) {$1$};
					\node at (0.25,1.7) {$1$};
					\draw (1,2)--(1,3)--(0,3)--(0,2);
					\draw (0.5,2)--(1,3);
					\node at (0.8,2.25) {$0$};	
					\node at (0.2,0.75) {$0$};
					\draw [blue,->] (1.2,1)--(2.8,1);
				\end{scope}
				\begin{scope}[shift={(6,0)}]
					\draw (0,-1)--(0,-1.5);
					\draw (1,-1)--(1,-1.5);
					\node at (0.5,-1.3) {$\vdots$};
					\draw (0,0)--(0,-1);
					\draw (1,0)--(1,-1);
					\draw (0,-1)--(1,-1);
					\draw (0,-1)--(1,0);
					\node at (0.75,-0.7) {$1$};
					\node at (0.25,-0.3) {$1$};
					\draw (0,0)--(1,0)--(1,1)--(0,1)--(0,0);
					\draw (0,0)--(0.5,1);
					\node at (0.65,0.4) {$2$};	
					\draw (1,1)--(1,2)--(0,2)--(0,1);
					\draw (0,1)--(1,2);
					\node at (0.75,1.3) {$1$};
					\node at (0.25,1.7) {$1$};
					\draw (1,2)--(1,3)--(0,3)--(0,2);
					\draw (0.5,2)--(1,3);
					\node at (0.8,2.25) {$0$};	
					\node at (0.2,0.75) {$0$};
					\node at (0.35,2.6) {$2$};
					\draw [black,->] (1.2,1)--(2.8,1);
				\end{scope}
				\begin{scope}[shift={(9,0)}]
					\draw (0,-1)--(0,-1.5);
					\draw (1,-1)--(1,-1.5);
					\node at (0.5,-1.3) {$\vdots$};
					\draw (0,0)--(0,-1);
					\draw (1,0)--(1,-1);
					\draw (0,-1)--(1,-1);
					\draw (0,-1)--(1,0);
					\node at (0.75,-0.7) {$1$};
					\node at (0.25,-0.3) {$1$};
					\draw (0,0)--(1,0)--(1,1)--(0,1)--(0,0);
					\draw (0,0)--(0.5,1);
					\node at (0.65,0.4) {$2$};	
					\draw (1,1)--(1,2)--(0,2)--(0,1);
					\draw (0,1)--(1,2);
					\node at (0.75,1.3) {$1$};
					\node at (0.25,1.7) {$1$};
					\draw (1,2)--(1,3)--(0,3)--(0,2);
					\draw (1,3)--(1,4)--(0,4)--(0,3);
					\draw (0,3)--(1,4);
					\draw (0.5,2)--(1,3);
					\node at (0.75,3.3) {$1$};
					
					\node at (0.8,2.25) {$0$};	
					\node at (0.2,0.75) {$0$};
					\node at (0.35,2.6) {$2$};
					\draw [black,->] (1.2,1)--(2.8,1);	
				\end{scope}
				\begin{scope}[shift={(12,0)}]
					\draw (0,-1)--(0,-1.5);
					\draw (1,-1)--(1,-1.5);
					\node at (0.5,-1.3) {$\vdots$};
					\draw (0,0)--(0,-1);
					\draw (1,0)--(1,-1);
					\draw (0,-1)--(1,-1);
					\draw (0,-1)--(1,0);
					\node at (0.75,-0.7) {$1$};
					\node at (0.25,-0.3) {$1$};
					\draw (0,0)--(1,0)--(1,1)--(0,1)--(0,0);
					\draw (0,0)--(0.5,1);
					\node at (0.65,0.4) {$2$};	
					\draw (1,1)--(1,2)--(0,2)--(0,1);
					\draw (0,1)--(1,2);
					\node at (0.75,1.3) {$1$};
					\node at (0.25,1.7) {$1$};
					\draw (1,2)--(1,3)--(0,3)--(0,2);
					\draw (1,3)--(1,4)--(0,4)--(0,3);
					\draw (0,3)--(1,4);
					\draw (0.5,2)--(1,3);
					\node at (0.75,3.3) {$1$};
					\node at (0.25,3.7) {$1$};
					\node at (0.8,2.25) {$0$};	
					\node at (0.2,0.75) {$0$};
					\node at (0.35,2.6) {$2$};
					\draw [blue,->] (1.2,1)--(2.8,1);
					\draw[red,->,bend right] (0.5,4.2) to node [swap] {} (-8.5,3.2);	
				\end{scope}
				\begin{scope}[shift={(15,0)}]
					\draw (0,-1)--(0,-1.5);
					\draw (1,-1)--(1,-1.5);
					\node at (0.5,-1.3) {$\vdots$};
					\draw (0,0)--(0,-1);
					\draw (1,0)--(1,-1);
					\draw (0,-1)--(1,-1);
					\draw (0,-1)--(1,0);
					\node at (0.75,-0.7) {$1$};
					\node at (0.25,-0.3) {$1$};
					\draw (0,0)--(1,0)--(1,1)--(0,1)--(0,0);
					\draw (0,0)--(0.5,1);
					\node at (0.65,0.4) {$2$};	
					\draw (1,1)--(1,2)--(0,2)--(0,1);
					\draw (0,1)--(1,2);
					\node at (0.75,1.3) {$1$};
					\node at (0.25,1.7) {$1$};
					\draw (1,2)--(1,3)--(0,3)--(0,2);
					\draw (1,3)--(1,4)--(0,4)--(0,3);
					\draw (0,3)--(1,4);
					\draw (0.5,2)--(1,3);
					\node at (0.75,3.3) {$1$};
					\node at (0.25,3.7) {$1$};
					\node at (0.8,2.25) {$0$};	
					\node at (0.2,0.75) {$0$};
					\node at (0.35,2.6) {$2$};
					\draw (1,4)--(1,5)--(0,5)--(0,4);
					\draw (0,4)--(0.5,5);
					\node at (0.65,4.4) {$2$};
					\draw [black,->] (1.2,1)--(2.8,1);
					\draw[red,->,bend right] (0.5,5.2) to node [swap] {} (-8.5,3.2);
				\end{scope}
				\begin{scope}[shift={(18,0)}]
					\draw (0,-1)--(0,-1.5);
					\draw (1,-1)--(1,-1.5);
					\node at (0.5,-1.3) {$\vdots$};
					\draw (0,0)--(0,-1);
					\draw (1,0)--(1,-1);
					\draw (0,-1)--(1,-1);
					\draw (0,-1)--(1,0);
					\node at (0.75,-0.7) {$1$};
					\node at (0.25,-0.3) {$1$};
					\draw (0,0)--(1,0)--(1,1)--(0,1)--(0,0);
					\draw (0,0)--(0.5,1);
					\node at (0.65,0.4) {$2$};	
					\draw (1,1)--(1,2)--(0,2)--(0,1);
					\draw (0,1)--(1,2);
					\node at (0.75,1.3) {$1$};
					\node at (0.25,1.7) {$1$};
					\draw (1,2)--(1,3)--(0,3)--(0,2);
					\draw (1,3)--(1,4)--(0,4)--(0,3);
					\draw (0,3)--(1,4);
					\draw (0.5,2)--(1,3);
					\node at (0.75,3.3) {$1$};
					\node at (0.25,3.7) {$1$};
					\node at (0.8,2.25) {$0$};	
					\node at (0.2,0.75) {$0$};
					\node at (0.35,2.6) {$2$};
					\draw (1,4)--(1,5)--(0,5)--(0,4);
					\draw (0,4)--(0.5,5);
					\node at (0.65,4.4) {$2$};
					\draw (1,5)--(1,6)--(0,6)--(0,5);
					\draw (0,5)--(1,6);
					\node at (0.75,5.3) {$1$};
					\draw [red,->] (0.5,-1.9)--(-7.7,-7);						
				\end{scope}
				
				\begin{scope}[shift={(9,-10)}]
					\draw (0,-1)--(0,-1.5);
					\draw (1,-1)--(1,-1.5);
					\node at (0.5,-1.3) {$\vdots$};
					\draw (0,0)--(0,-1);
					\draw (1,0)--(1,-1);
					\draw (0,-1)--(1,-1);
					\draw (0,-1)--(1,0);
					\node at (0.75,-0.7) {$1$};
					\node at (0.25,-0.3) {$1$};
					\draw (0,0)--(1,0)--(1,1)--(0,1)--(0,0);
					\draw (0,0)--(0.5,1);
					\node at (0.65,0.4) {$2$};	
					\draw (1,1)--(1,2)--(0,2)--(0,1);
					\draw (0,1)--(1,2);
					\node at (0.75,1.3) {$1$};
					\node at (0.25,1.7) {$1$};
					\draw (1,2)--(1,3)--(0,3)--(0,2);
					\draw (1,3)--(1,4)--(0,4)--(0,3);
					\draw (0,3)--(1,4);
					\draw (0.5,2)--(1,3);
					\node at (0.75,3.3) {$1$};
					\node at (0.25,3.7) {$1$};
					\node at (0.8,2.25) {$0$};	
					\node at (0.2,0.75) {$0$};
					\node at (0.35,2.6) {$2$};
					\draw (1,4)--(1,5)--(0,5)--(0,4);
					\draw (0,4)--(0.5,5);
					\node at (0.65,4.4) {$2$};
					\draw (1,5)--(1,6)--(0,6)--(0,5);
					\draw (0,5)--(1,6);
					\node at (0.75,5.3) {$1$};
					\draw (1,6)--(1,7)--(0,7)--(0,6);
					\draw (0.5,6)--(1,7);
					\node at (0.8,6.25) {$0$};
					\draw [red,->] (-0.3,3)--(-8.5,8.1);
				\end{scope}
			\end{tikzpicture}
	\vskip -1pc	
		\caption{Young column realization of the perfect crystal $B_{1}$}\label{FIG:realization of B1} 
	\end{figure}

\newpage

	\begin{figure}[H]
	\centering
	\begin{tikzpicture}[scale=0.57]
		\draw[black,->] (-0.5,2)--(-2.5,0); 
		\draw[blue,->] (-2.5,-1)--(-0.5,-3); 
		\draw[blue,->] (1.5,-4)--(3.5,-6);   
		\draw[black,->] (3.5,-7)--(1.5,-9); 
		\draw[blue,->] (5.5,-7)--(7.5,-9);
		\draw[black,->] (7.5,-10)--(6.5,-11);
		\draw[black,->] (4.8,-12.7)--(1.5,-16);
		\draw[blue,->] (1.5,-10)--(2.5,-11);
		\draw[blue,->] (4.2,-12.7)--(7.5,-16);
		\draw[blue,->] (1.5,-17)--(3.5,-19);
		\draw[blue,->] (5.5,-20)--(7.5,-22);
		\draw[blue,->] (9.5,-23)--(11.5,-25);
		\draw[black,->] (7.5,-17)--(5.5,-19); 
		\draw[black,->] (11.5,-27)--(9.5,-29); 
		\draw[red,->] (9.5,-28)--(15.5,-8); 
		\draw[red,->] (15.5,-7)--(1.5,2); 
		\draw[red,->] (0.5,-13.5)--(-3.5,-3.5); 
		\draw[red,->] (4.5,-17.5)--(0.5,-6); 
		\draw[red,->] (8.5,-20.5)--(4.5,-9); 
		\draw[red,->] (12.5,-23.5)--(8.5,-12); 
		\begin{scope}[shift={(0,1)}]			
			\draw (0,0)--(0,-0.5);
			\draw (1,0)--(1,-0.5);
			\node at (0.5,-0.3) {$\vdots$};
			\draw (0,0)--(0,1)--(1,1)--(1,0)--(0,0); \draw (0,0)--(1,1);\draw (0,0)--(0.5,1);
			\draw (0,1)--(0,2)--(1,2)--(1,1);\draw (0.5,1)--(0.5,2);
			\draw (0,2)--(0,3)--(1,3)--(1,2);\draw (0,2)--(1,3);\draw (0.5,2)--(1,3);
			\node at (0.15,0.75) {$2$};\node at (0.53,0.75) {$0$};\node at (0.48,2.25) {$0$};
			\node at (0.7,0.3) {$2$};
			\node at (0.25,1.5) {$1$};
		\end{scope}
		\begin{scope}[shift={(-4,-2.5)}]			
			\draw (0,0)--(0,-0.5);
			\draw (1,0)--(1,-0.5);
			\node at (0.5,-0.3) {$\vdots$};
			\draw (0,0)--(0,1)--(1,1)--(1,0)--(0,0); \draw (0,0)--(1,1);\draw (0,0)--(0.5,1);
			\draw (0,1)--(0,2)--(1,2)--(1,1);\draw (0.5,1)--(0.5,2);
			\draw (0,2)--(0,3)--(1,3)--(1,2);\draw (0,2)--(1,3);\draw (0.5,2)--(1,3);
			\draw (0.5,1.5)--(1,1.5);\node at (0.75,1.25) {$1$};
			\node at (0.15,0.75) {$2$};\node at (0.53,0.75) {$0$};\node at (0.48,2.25) {$0$};
			\node at (0.7,0.3) {$2$};
			\node at (0.25,1.5) {$1$};
		\end{scope}
		\begin{scope}[shift={(0,-5)}]			
			\draw (0,0)--(0,-0.5);
			\draw (1,0)--(1,-0.5);
			\node at (0.5,-0.3) {$\vdots$};
			\draw (0,0)--(0,1)--(1,1)--(1,0)--(0,0); \draw (0,0)--(1,1);\draw (0,0)--(0.5,1);
			\draw (0,1)--(0,2)--(1,2)--(1,1);\draw (0.5,1)--(0.5,2);
			\draw (0,2)--(0,3)--(1,3)--(1,2);\draw (0,2)--(1,3);\draw (0.5,2)--(1,3);
			\draw (0.5,1.5)--(1,1.5);\node at (0.75,1.25) {$1$};\node at (0.75,1.75) {$2$};
			\node at (0.15,0.75) {$2$};\node at (0.53,0.75) {$0$};\node at (0.48,2.25) {$0$};
			\node at (0.7,0.3) {$2$};
			\node at (0.25,1.5) {$1$};
		\end{scope}
		\begin{scope}[shift={(4,-8)}]			
			\draw (0,0)--(0,-0.5);
			\draw (1,0)--(1,-0.5);
			\node at (0.5,-0.3) {$\vdots$};
			\draw (0,0)--(0,1)--(1,1)--(1,0)--(0,0); \draw (0,0)--(1,1);\draw (0,0)--(0.5,1);
			\draw (0,1)--(0,2)--(1,2)--(1,1);\draw (0.5,1)--(0.5,2);
			\draw (0,2)--(0,3)--(1,3)--(1,2);\draw (0,2)--(1,3);\draw (0.5,2)--(1,3);
			\draw (0.5,1.5)--(1,1.5);\node at (0.75,1.25) {$1$};\node at (0.75,1.75) {$2$};
			\node at (0.15,0.75) {$2$};\node at (0.53,0.75) {$0$};\node at (0.48,2.25) {$0$};
			\node at (0.7,0.3) {$2$};\node at (0.3,2.7) {$2$};
			\node at (0.25,1.5) {$1$};
		\end{scope}
		\begin{scope}[shift={(8,-11)}]			
			\draw (0,0)--(0,-0.5);
			\draw (1,0)--(1,-0.5);
			\node at (0.5,-0.3) {$\vdots$};
			\draw (0,0)--(0,1)--(1,1)--(1,0)--(0,0); \draw (0,0)--(1,1);\draw (0,0)--(0.5,1);
			\draw (0,1)--(0,2)--(1,2)--(1,1);\draw (0.5,1)--(0.5,2);
			\draw (0,2)--(0,3)--(1,3)--(1,2);\draw (0,2)--(1,3);\draw (0.5,2)--(1,3);
			\draw (0.5,1.5)--(1,1.5);\node at (0.75,1.25) {$1$};\node at (0.75,1.75) {$2$};
			\node at (0.15,0.75) {$2$};\node at (0.53,0.75) {$0$};\node at (0.48,2.25) {$0$};
			\node at (0.7,0.3) {$2$};\node at (0.3,2.7) {$2$};
			\node at (0.25,1.5) {$1$};
			\node at (0.85,2.25) {$2$};
		\end{scope}
		\begin{scope}[shift={(0,-11)}]			
			\draw (0,0)--(0,-0.5);
			\draw (1,0)--(1,-0.5);
			\node at (0.5,-0.3) {$\vdots$};
			\draw (0,0)--(0,1)--(1,1)--(1,0)--(0,0); \draw (0,0)--(1,1);\draw (0,0)--(0.5,1);
			\draw (0,1)--(0,2)--(1,2)--(1,1);\draw (0.5,1)--(0.5,2);
			\draw (0,2)--(0,3)--(1,3)--(1,2);\draw (0,2)--(1,3);\draw (0.5,2)--(1,3);
			\draw (0.5,1.5)--(1,1.5);\node at (0.75,1.25) {$1$};\node at (0.75,1.75) {$2$};
			\node at (0.15,0.75) {$2$};\node at (0.53,0.75) {$0$};\node at (0.48,2.25) {$0$};
			\node at (0.7,0.3) {$2$};\node at (0.3,2.7) {$2$};
			\node at (0.25,1.5) {$1$};
			\draw (0,3)--(0,4)--(1,4)--(1,3);\draw (0.5,3)--(0.5,4);\draw (0,3.5)--(0.5,3.5);
			\node at (0.25,3.25) {$1$};
		\end{scope}
		\begin{scope}[shift={(16,-11)}]			
			\draw (0,0)--(0,-0.5);
			\draw (1,0)--(1,-0.5);
			\node at (0.5,-0.3) {$\vdots$};
			\draw (0,0)--(0,1)--(1,1)--(1,0)--(0,0); \draw (0,0)--(1,1);\draw (0,0)--(0.5,1);
			\draw (0,1)--(0,2)--(1,2)--(1,1);\draw (0.5,1)--(0.5,2);
			\draw (0,2)--(0,3)--(1,3)--(1,2);\draw (0,2)--(1,3);\draw (0.5,2)--(1,3);
			\draw (0.5,1.5)--(1,1.5);\node at (0.75,1.25) {$1$};\node at (0.75,1.75) {$2$};
			\node at (0.15,0.75) {$2$};\node at (0.53,0.75) {$0$};\node at (0.48,2.25) {$0$};
			\node at (0.7,0.3) {$2$};\node at (0.3,2.7) {$2$};
			\node at (0.25,1.5) {$1$};
			\draw (0,3)--(0,4)--(1,4)--(1,3);\draw (0.5,3)--(0.5,4);\draw (0,3.5)--(0.5,3.5);
			\node at (0.25,3.25) {$1$}; \node at (0.25,3.75) {$2$};\node at (0.85,2.25) {$2$};
			\node at (0.75,3.5) {$1$};
			\draw (0,4)--(0,5)--(1,5)--(1,4);\draw (0,4)--(1,5);\draw (0,4)--(0.5,5);\node at (0.7,4.3) {$2$};	\node at (0.15,4.75) {$2$};
			\draw (0,5)--(0,6)--(1,6)--(1,5);	\draw (0.5,5)--(0.5,6);	\draw (0.5,5.5)--(1,5.5); \node at (0.25,5.5) {$1$};
			\draw (0,6)--(0,7)--(1,7)--(1,6);\draw (0,6)--(1,7);\draw (0.5,6)--(1,7);
			\node at (0.48,6.25) {$0$};
		\end{scope}
		\begin{scope}[shift={(3,-15)}]			
			\draw (0,0)--(0,-0.5);
			\draw (1,0)--(1,-0.5);
			\node at (0.5,-0.3) {$\vdots$};
			\draw (0,0)--(0,1)--(1,1)--(1,0)--(0,0); \draw (0,0)--(1,1);\draw (0,0)--(0.5,1);
			\draw (0,1)--(0,2)--(1,2)--(1,1);\draw (0.5,1)--(0.5,2);
			\draw (0,2)--(0,3)--(1,3)--(1,2);\draw (0,2)--(1,3);\draw (0.5,2)--(1,3);
			\draw (0.5,1.5)--(1,1.5);\node at (0.75,1.25) {$1$};\node at (0.75,1.75) {$2$};
			\node at (0.15,0.75) {$2$};\node at (0.53,0.75) {$0$};\node at (0.48,2.25) {$0$};
			\node at (0.7,0.3) {$2$};\node at (0.3,2.7) {$2$};
			\node at (0.25,1.5) {$1$};
			\draw (0,3)--(0,4)--(1,4)--(1,3);\draw (0.5,3)--(0.5,4);\draw (0,3.5)--(0.5,3.5);
			\node at (0.25,3.25) {$1$};\node at (0.25,3.75) {$2$};
		\end{scope}
		\begin{scope}[shift={(5,-15)}]			
			\draw (0,0)--(0,-0.5);
			\draw (1,0)--(1,-0.5);
			\node at (0.5,-0.3) {$\vdots$};
			\draw (0,0)--(0,1)--(1,1)--(1,0)--(0,0); \draw (0,0)--(1,1);\draw (0,0)--(0.5,1);
			\draw (0,1)--(0,2)--(1,2)--(1,1);\draw (0.5,1)--(0.5,2);
			\draw (0,2)--(0,3)--(1,3)--(1,2);\draw (0,2)--(1,3);\draw (0.5,2)--(1,3);
			\draw (0.5,1.5)--(1,1.5);\node at (0.75,1.25) {$1$};\node at (0.75,1.75) {$2$};
			\node at (0.15,0.75) {$2$};\node at (0.53,0.75) {$0$};\node at (0.48,2.25) {$0$};
			\node at (0.7,0.3) {$2$};\node at (0.3,2.7) {$2$};
			\node at (0.25,1.5) {$1$};
			\draw (0,3)--(0,4)--(1,4)--(1,3);\draw (0.5,3)--(0.5,4);\draw (0,3.5)--(0.5,3.5);
			\node at (0.25,3.25) {$1$};
			\node at (0.85,2.25) {$2$};
		\end{scope}
		\begin{scope}[shift={(0,-18)}]			
			\draw (0,0)--(0,-0.5);
			\draw (1,0)--(1,-0.5);
			\node at (0.5,-0.3) {$\vdots$};
			\draw (0,0)--(0,1)--(1,1)--(1,0)--(0,0); \draw (0,0)--(1,1);\draw (0,0)--(0.5,1);
			\draw (0,1)--(0,2)--(1,2)--(1,1);\draw (0.5,1)--(0.5,2);
			\draw (0,2)--(0,3)--(1,3)--(1,2);\draw (0,2)--(1,3);\draw (0.5,2)--(1,3);
			\draw (0.5,1.5)--(1,1.5);\node at (0.75,1.25) {$1$};\node at (0.75,1.75) {$2$};
			\node at (0.15,0.75) {$2$};\node at (0.53,0.75) {$0$};\node at (0.48,2.25) {$0$};
			\node at (0.7,0.3) {$2$};\node at (0.3,2.7) {$2$};
			\node at (0.25,1.5) {$1$};
			\draw (0,3)--(0,4)--(1,4)--(1,3);\draw (0.5,3)--(0.5,4);\draw (0,3.5)--(0.5,3.5);
			\node at (0.25,3.25) {$1$};\node at (0.75,3.5) {$1$}; \node at (0.85,2.25) {$2$};
		\end{scope}
		\begin{scope}[shift={(8,-18)}]			
			\draw (0,0)--(0,-0.5);
			\draw (1,0)--(1,-0.5);
			\node at (0.5,-0.3) {$\vdots$};
			\draw (0,0)--(0,1)--(1,1)--(1,0)--(0,0); \draw (0,0)--(1,1);\draw (0,0)--(0.5,1);
			\draw (0,1)--(0,2)--(1,2)--(1,1);\draw (0.5,1)--(0.5,2);
			\draw (0,2)--(0,3)--(1,3)--(1,2);\draw (0,2)--(1,3);\draw (0.5,2)--(1,3);
			\draw (0.5,1.5)--(1,1.5);\node at (0.75,1.25) {$1$};\node at (0.75,1.75) {$2$};
			\node at (0.15,0.75) {$2$};\node at (0.53,0.75) {$0$};\node at (0.48,2.25) {$0$};
			\node at (0.7,0.3) {$2$};\node at (0.3,2.7) {$2$};
			\node at (0.25,1.5) {$1$};
			\draw (0,3)--(0,4)--(1,4)--(1,3);\draw (0.5,3)--(0.5,4);\draw (0,3.5)--(0.5,3.5);
			\node at (0.25,3.25) {$1$}; \node at (0.25,3.75) {$2$};\node at (0.85,2.25) {$2$};
		\end{scope}
		\begin{scope}[shift={(4,-22)}]			
			\draw (0,0)--(0,-0.5);
			\draw (1,0)--(1,-0.5);
			\node at (0.5,-0.3) {$\vdots$};
			\draw (0,0)--(0,1)--(1,1)--(1,0)--(0,0); \draw (0,0)--(1,1);\draw (0,0)--(0.5,1);
			\draw (0,1)--(0,2)--(1,2)--(1,1);\draw (0.5,1)--(0.5,2);
			\draw (0,2)--(0,3)--(1,3)--(1,2);\draw (0,2)--(1,3);\draw (0.5,2)--(1,3);
			\draw (0.5,1.5)--(1,1.5);\node at (0.75,1.25) {$1$};\node at (0.75,1.75) {$2$};
			\node at (0.15,0.75) {$2$};\node at (0.53,0.75) {$0$};\node at (0.48,2.25) {$0$};
			\node at (0.7,0.3) {$2$};\node at (0.3,2.7) {$2$};
			\node at (0.25,1.5) {$1$};
			\draw (0,3)--(0,4)--(1,4)--(1,3);\draw (0.5,3)--(0.5,4);\draw (0,3.5)--(0.5,3.5);
			\node at (0.25,3.25) {$1$}; \node at (0.25,3.75) {$2$};\node at (0.85,2.25) {$2$};
			\node at (0.75,3.5) {$1$};
		\end{scope}
		\begin{scope}[shift={(8,-26)}]			
			\draw (0,0)--(0,-0.5);
			\draw (1,0)--(1,-0.5);
			\node at (0.5,-0.3) {$\vdots$};
			\draw (0,0)--(0,1)--(1,1)--(1,0)--(0,0); \draw (0,0)--(1,1);\draw (0,0)--(0.5,1);
			\draw (0,1)--(0,2)--(1,2)--(1,1);\draw (0.5,1)--(0.5,2);
			\draw (0,2)--(0,3)--(1,3)--(1,2);\draw (0,2)--(1,3);\draw (0.5,2)--(1,3);
			\draw (0.5,1.5)--(1,1.5);\node at (0.75,1.25) {$1$};\node at (0.75,1.75) {$2$};
			\node at (0.15,0.75) {$2$};\node at (0.53,0.75) {$0$};\node at (0.48,2.25) {$0$};
			\node at (0.7,0.3) {$2$};\node at (0.3,2.7) {$2$};
			\node at (0.25,1.5) {$1$};
			\draw (0,3)--(0,4)--(1,4)--(1,3);\draw (0.5,3)--(0.5,4);\draw (0,3.5)--(0.5,3.5);
			\node at (0.25,3.25) {$1$}; \node at (0.25,3.75) {$2$};\node at (0.85,2.25) {$2$};
			\node at (0.75,3.5) {$1$};
			\draw (0,4)--(0,5)--(1,5)--(1,4);\draw (0,4)--(1,5);\node at (0.7,4.3) {$2$};
		\end{scope}
		\begin{scope}[shift={(12,-29)}]			
			\draw (0,0)--(0,-0.5);
			\draw (1,0)--(1,-0.5);
			\node at (0.5,-0.3) {$\vdots$};
			\draw (0,0)--(0,1)--(1,1)--(1,0)--(0,0); \draw (0,0)--(1,1);\draw (0,0)--(0.5,1);
			\draw (0,1)--(0,2)--(1,2)--(1,1);\draw (0.5,1)--(0.5,2);
			\draw (0,2)--(0,3)--(1,3)--(1,2);\draw (0,2)--(1,3);\draw (0.5,2)--(1,3);
			\draw (0.5,1.5)--(1,1.5);\node at (0.75,1.25) {$1$};\node at (0.75,1.75) {$2$};
			\node at (0.15,0.75) {$2$};\node at (0.53,0.75) {$0$};\node at (0.48,2.25) {$0$};
			\node at (0.7,0.3) {$2$};\node at (0.3,2.7) {$2$};
			\node at (0.25,1.5) {$1$};
			\draw (0,3)--(0,4)--(1,4)--(1,3);\draw (0.5,3)--(0.5,4);\draw (0,3.5)--(0.5,3.5);
			\node at (0.25,3.25) {$1$}; \node at (0.25,3.75) {$2$};\node at (0.85,2.25) {$2$};
			\node at (0.75,3.5) {$1$};
			\draw (0,4)--(0,5)--(1,5)--(1,4);\draw (0,4)--(1,5);\draw (0,4)--(0.5,5);\node at (0.7,4.3) {$2$};	\node at (0.15,4.75) {$2$};
		\end{scope}
		\begin{scope}[shift={(8,-34)}]			
			\draw (0,0)--(0,-0.5);
			\draw (1,0)--(1,-0.5);
			\node at (0.5,-0.3) {$\vdots$};
			\draw (0,0)--(0,1)--(1,1)--(1,0)--(0,0); \draw (0,0)--(1,1);\draw (0,0)--(0.5,1);
			\draw (0,1)--(0,2)--(1,2)--(1,1);\draw (0.5,1)--(0.5,2);
			\draw (0,2)--(0,3)--(1,3)--(1,2);\draw (0,2)--(1,3);\draw (0.5,2)--(1,3);
			\draw (0.5,1.5)--(1,1.5);\node at (0.75,1.25) {$1$};\node at (0.75,1.75) {$2$};
			\node at (0.15,0.75) {$2$};\node at (0.53,0.75) {$0$};\node at (0.48,2.25) {$0$};
			\node at (0.7,0.3) {$2$};\node at (0.3,2.7) {$2$};
			\node at (0.25,1.5) {$1$};
			\draw (0,3)--(0,4)--(1,4)--(1,3);\draw (0.5,3)--(0.5,4);\draw (0,3.5)--(0.5,3.5);
			\node at (0.25,3.25) {$1$}; \node at (0.25,3.75) {$2$};\node at (0.85,2.25) {$2$};
			\node at (0.75,3.5) {$1$};
			\draw (0,4)--(0,5)--(1,5)--(1,4);\draw (0,4)--(1,5);\draw (0,4)--(0.5,5);\node at (0.7,4.3) {$2$};	\node at (0.15,4.75) {$2$};
			\draw (0,5)--(0,6)--(1,6)--(1,5);	\draw (0.5,5)--(0.5,6);	\draw (0.5,5.5)--(1,5.5); \node at (0.25,5.5) {$1$};
		\end{scope}
	\end{tikzpicture}
	\vskip -1pc
	\caption{Young column realization of the perfect crystal $B'_{1}$}\label{FIG:realization of B1'} 
\end{figure}
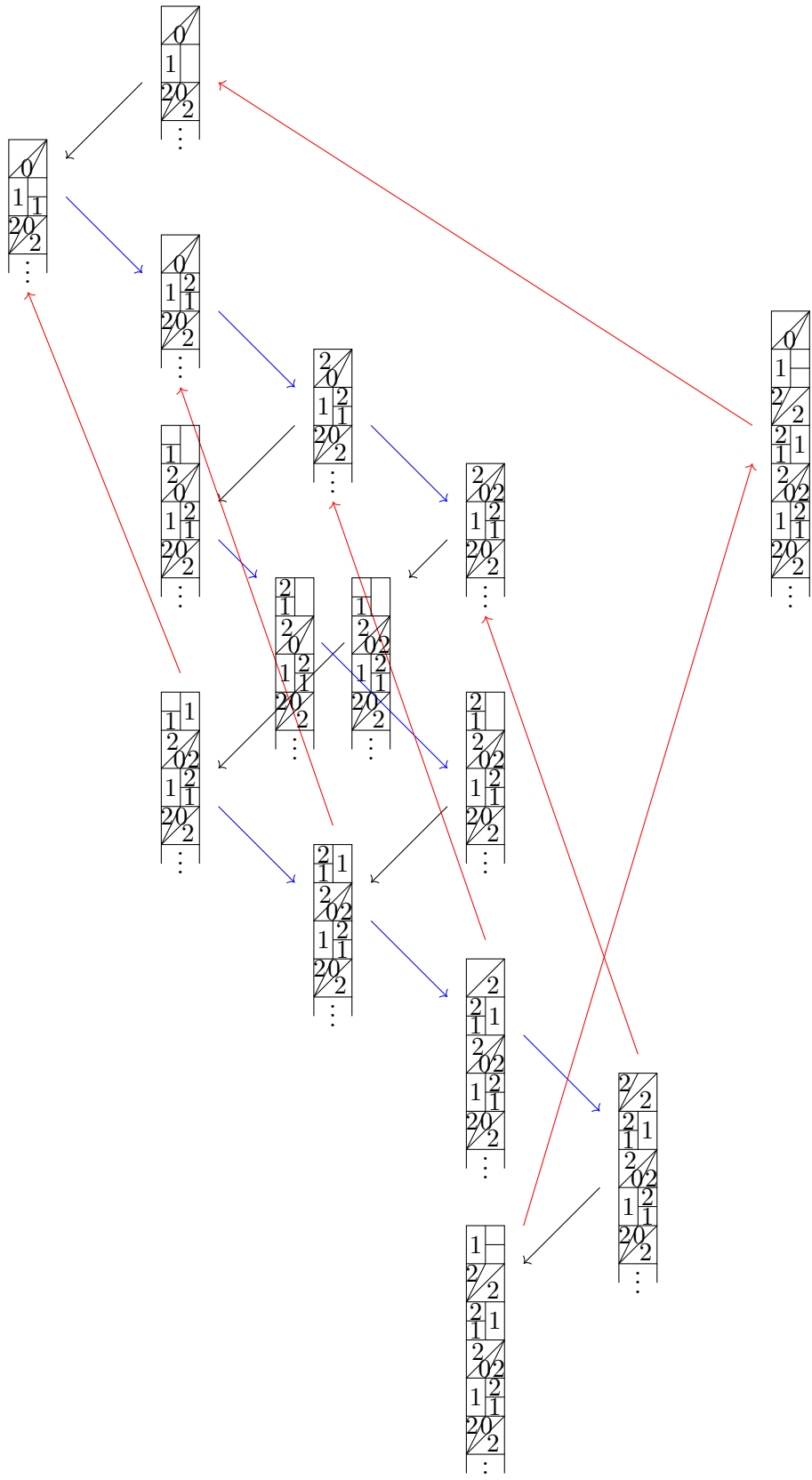

Let $C_1$ (resp. $C'_1$) denote the set of all equivalence classes of Young columns for $U_q(D_4^{(3)})$ (resp. $U_q(G_2^{(1)})$). 
We define
\begin{align*}
&\psi:C_1\to B_1,\quad c_i\mapsto u_i,\quad i\in\{1,2,3,0,\bar{3},\bar{2},\bar{1},\phi\},\\
&\psi':C'_1\to B'_1,\quad c'_i\mapsto v_i,\quad i\in\{1,2,3,4,5,6,7,0,\bar{7},\bar{6},\bar{5},\bar{4},\bar{3},\bar{2},\bar{1}\}.
\end{align*}

By the crystal graphs in Example \ref{ex:level 1 perfect crystal D43}, Example \ref{ex:level 1 perfect crystal G21} and the crystal graphs in Proposition \ref{prop:realization of perfect crystal}, we have the following proposition.

\begin{Prop} \label{perfect isomorphism}
The map $\psi:C_1\to B_1$ (resp. $\psi':C'_1\to B'_1$) is a crystal isomorphism.
\end{Prop}

\subsection{Young walls} Now we introduce the {\it wall pattern} and the {\it ground-state wall} as follows.

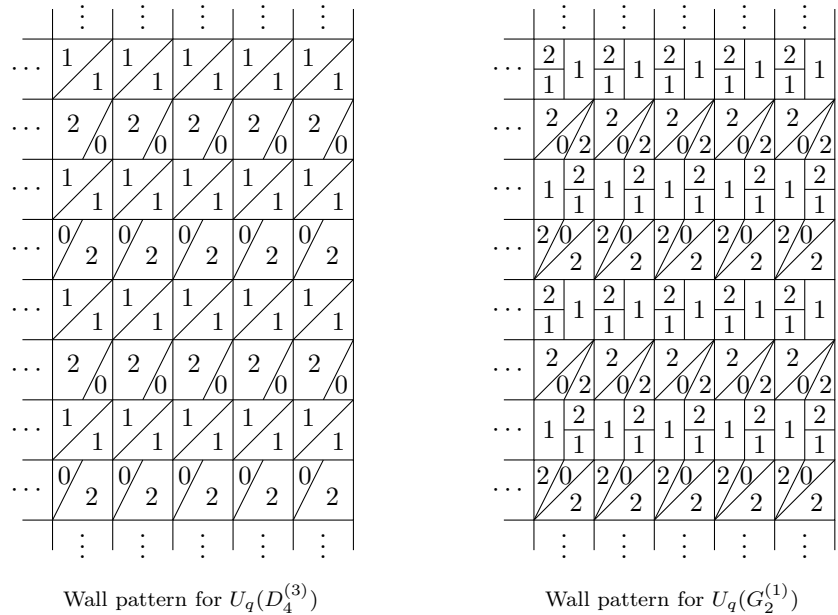
\begin{figure}[H] 
	{\small
		\centering
		\begin{tikzpicture}[scale=0.8]	
			\begin{scope}[shift={(-4,0)}]	
			\node at (-1.7,-1.3) {\scriptsize{Wall pattern for $U_q(D_4^{(3)})$}};		
			\draw (1,0)--(1,-0.5);
			\node at (0.5,-0.3) {$\vdots$};
			\draw (0,0)--(1,0)--(1,1)--(0,1);
			\draw (0,0)--(0.5,1);
			\node at (0.2,0.75) {$0$};
			\node at (0.65,0.4) {$2$};	
			\draw (1,1)--(1,2)--(0,2);
			\draw (0,1)--(1,2);
			\node at (0.75,1.3) {$1$};
			\node at (0.25,1.7) {$1$};
			\draw (1,2)--(1,3)--(0,3);
			\draw (0.5,2)--(1,3);
			\node at (0.8,2.25) {$0$};
			\node at (0.35,2.6) {$2$};
			\draw (1,3)--(1,4)--(0,4);
			\draw (0,3)--(1,4);
			\node at (0.75,3.3) {$1$};
			\node at (0.25,3.7) {$1$};
			\begin{scope}[shift={(0,4)}]
				
				\draw (1,4)--(1,4.5);
				\node at (0.5,4.5) {$\vdots$};
				\draw (1,0)--(1,1)--(0,1);
				\draw (0,0)--(0.5,1);
				\node at (0.2,0.75) {$0$};
				\node at (0.65,0.4) {$2$};	
				\draw (1,1)--(1,2)--(0,2);
				\draw (0,1)--(1,2);
				\node at (0.75,1.3) {$1$};
				\node at (0.25,1.7) {$1$};
				\draw (1,2)--(1,3)--(0,3);
				\draw (0.5,2)--(1,3);
				\node at (0.8,2.25) {$0$};
				\node at (0.35,2.6) {$2$};
				\draw (1,3)--(1,4)--(0,4);
				\draw (0,3)--(1,4);
				\node at (0.75,3.3) {$1$};
				\node at (0.25,3.7) {$1$};		
			\end{scope}
			\begin{scope}[shift={(-1,0)}]
				\draw (1,0)--(1,-0.5);
				\node at (0.5,-0.3) {$\vdots$};
				\draw (0,0)--(1,0)--(1,1)--(0,1);
				\draw (0,0)--(0.5,1);
				\node at (0.2,0.75) {$0$};
				\node at (0.65,0.4) {$2$};	
				\draw (1,1)--(1,2)--(0,2);
				\draw (0,1)--(1,2);
				\node at (0.75,1.3) {$1$};
				\node at (0.25,1.7) {$1$};
				\draw (1,2)--(1,3)--(0,3);
				\draw (0.5,2)--(1,3);
				\node at (0.8,2.25) {$0$};
				\node at (0.35,2.6) {$2$};
				\draw (1,3)--(1,4)--(0,4);
				\draw (0,3)--(1,4);
				\node at (0.75,3.3) {$1$};
				\node at (0.25,3.7) {$1$};
				\begin{scope}[shift={(0,4)}]
					
					\draw (1,4)--(1,4.5);
					\node at (0.5,4.5) {$\vdots$};
					\draw (1,0)--(1,1)--(0,1);
					\draw (0,0)--(0.5,1);
					\node at (0.2,0.75) {$0$};
					\node at (0.65,0.4) {$2$};	
					\draw (1,1)--(1,2)--(0,2);
					\draw (0,1)--(1,2);
					\node at (0.75,1.3) {$1$};
					\node at (0.25,1.7) {$1$};
					\draw (1,2)--(1,3)--(0,3);
					\draw (0.5,2)--(1,3);
					\node at (0.8,2.25) {$0$};
					\node at (0.35,2.6) {$2$};
					\draw (1,3)--(1,4)--(0,4);
					\draw (0,3)--(1,4);
					\node at (0.75,3.3) {$1$};
					\node at (0.25,3.7) {$1$};		
				\end{scope}	
			\end{scope}
			\begin{scope}[shift={(-2,0)}]
				\draw (1,0)--(1,-0.5);
				\node at (0.5,-0.3) {$\vdots$};
				\draw (0,0)--(1,0)--(1,1)--(0,1);
				\draw (0,0)--(0.5,1);
				\node at (0.2,0.75) {$0$};
				\node at (0.65,0.4) {$2$};	
				\draw (1,1)--(1,2)--(0,2);
				\draw (0,1)--(1,2);
				\node at (0.75,1.3) {$1$};
				\node at (0.25,1.7) {$1$};
				\draw (1,2)--(1,3)--(0,3);
				\draw (0.5,2)--(1,3);
				\node at (0.8,2.25) {$0$};
				\node at (0.35,2.6) {$2$};
				\draw (1,3)--(1,4)--(0,4);
				\draw (0,3)--(1,4);
				\node at (0.75,3.3) {$1$};
				\node at (0.25,3.7) {$1$};
				\begin{scope}[shift={(0,4)}]
					
					\draw (1,4)--(1,4.5);
					\node at (0.5,4.5) {$\vdots$};
					\draw (1,0)--(1,1)--(0,1);
					\draw (0,0)--(0.5,1);
					\node at (0.2,0.75) {$0$};
					\node at (0.65,0.4) {$2$};	
					\draw (1,1)--(1,2)--(0,2);
					\draw (0,1)--(1,2);
					\node at (0.75,1.3) {$1$};
					\node at (0.25,1.7) {$1$};
					\draw (1,2)--(1,3)--(0,3);
					\draw (0.5,2)--(1,3);
					\node at (0.8,2.25) {$0$};
					\node at (0.35,2.6) {$2$};
					\draw (1,3)--(1,4)--(0,4);
					\draw (0,3)--(1,4);
					\node at (0.75,3.3) {$1$};
					\node at (0.25,3.7) {$1$};		
				\end{scope}	
			\end{scope}
			\begin{scope}[shift={(-3,0)}]
				\draw (1,0)--(1,-0.5);
				\node at (0.5,-0.3) {$\vdots$};
				\draw (0,0)--(1,0)--(1,1)--(0,1);
				\draw (0,0)--(0.5,1);
				\node at (0.2,0.75) {$0$};
				\node at (0.65,0.4) {$2$};	
				\draw (1,1)--(1,2)--(0,2);
				\draw (0,1)--(1,2);
				\node at (0.75,1.3) {$1$};
				\node at (0.25,1.7) {$1$};
				\draw (1,2)--(1,3)--(0,3);
				\draw (0.5,2)--(1,3);
				\node at (0.8,2.25) {$0$};
				\node at (0.35,2.6) {$2$};
				\draw (1,3)--(1,4)--(0,4);
				\draw (0,3)--(1,4);
				\node at (0.75,3.3) {$1$};
				\node at (0.25,3.7) {$1$};
				\begin{scope}[shift={(0,4)}]
					
					\draw (1,4)--(1,4.5);
					\node at (0.5,4.5) {$\vdots$};
					\draw (1,0)--(1,1)--(0,1);
					\draw (0,0)--(0.5,1);
					\node at (0.2,0.75) {$0$};
					\node at (0.65,0.4) {$2$};	
					\draw (1,1)--(1,2)--(0,2);
					\draw (0,1)--(1,2);
					\node at (0.75,1.3) {$1$};
					\node at (0.25,1.7) {$1$};
					\draw (1,2)--(1,3)--(0,3);
					\draw (0.5,2)--(1,3);
					\node at (0.8,2.25) {$0$};
					\node at (0.35,2.6) {$2$};
					\draw (1,3)--(1,4)--(0,4);
					\draw (0,3)--(1,4);
					\node at (0.75,3.3) {$1$};
					\node at (0.25,3.7) {$1$};		
				\end{scope}	
			\end{scope}
			\begin{scope}[shift={(-4,0)}]
				\draw (0,0)--(0,-0.5);
				\draw (1,0)--(1,-0.5);
				\node at (0.5,-0.3) {$\vdots$};
				\draw (0,0)--(1,0)--(1,1)--(0,1)--(0,0);
				\draw (0,0)--(-0.5,0);
				\draw (0,1)--(-0.5,1);
				\draw (0,2)--(-0.5,2);
				\draw (0,3)--(-0.5,3);
				\node at (-0.4,0.5) {$\cdots$};
				\node at (-0.4,1.5) {$\cdots$};
				\node at (-0.4,2.5) {$\cdots$};
				\node at (-0.4,3.5) {$\cdots$};
				\draw (0,0)--(0.5,1);
				\node at (0.2,0.75) {$0$};
				\node at (0.65,0.4) {$2$};	
				\draw (1,1)--(1,2)--(0,2)--(0,1);
				\draw (0,1)--(1,2);
				\node at (0.75,1.3) {$1$};
				\node at (0.25,1.7) {$1$};
				\draw (1,2)--(1,3)--(0,3)--(0,2);
				\draw (0.5,2)--(1,3);
				\node at (0.8,2.25) {$0$};
				\node at (0.35,2.6) {$2$};
				\draw (1,3)--(1,4)--(0,4)--(0,3);
				\draw (0,3)--(1,4);
				\node at (0.75,3.3) {$1$};
				\node at (0.25,3.7) {$1$};
				\begin{scope}[shift={(0,4)}]
					\draw (0,4)--(0,4.5);
					\draw (1,4)--(1,4.5);
					\node at (0.5,4.5) {$\vdots$};
					\draw (1,0)--(1,1)--(0,1)--(0,0);
					\draw (0,0)--(-0.5,0);
					\draw (0,1)--(-0.5,1);
					\draw (0,2)--(-0.5,2);
					\draw (0,3)--(-0.5,3);
					\draw (0,4)--(-0.5,4);
					\node at (-0.4,0.5) {$\cdots$};
					\node at (-0.4,1.5) {$\cdots$};
					\node at (-0.4,2.5) {$\cdots$};
					\node at (-0.4,3.5) {$\cdots$};
					\draw (0,0)--(0.5,1);
					\node at (0.2,0.75) {$0$};
					\node at (0.65,0.4) {$2$};	
					\draw (1,1)--(1,2)--(0,2)--(0,1);
					\draw (0,1)--(1,2);
					\node at (0.75,1.3) {$1$};
					\node at (0.25,1.7) {$1$};
					\draw (1,2)--(1,3)--(0,3)--(0,2);
					\draw (0.5,2)--(1,3);
					\node at (0.8,2.25) {$0$};
					\node at (0.35,2.6) {$2$};
					\draw (1,3)--(1,4)--(0,4)--(0,3);
					\draw (0,3)--(1,4);
					\node at (0.75,3.3) {$1$};
					\node at (0.25,3.7) {$1$};	
					
				\end{scope}
			\end{scope}
			\end{scope}
		
	\begin{scope}[shift={(4,0)}]
		\node at (-1.7,-1.3) {\scriptsize{Wall pattern for $U_q(G_2^{(1)})$}};			
	\begin{scope}[shift={(0,0)}]
		\draw (1,0)--(1,-0.5);
		\node at (0.5,-0.3) {$\vdots$};
		
		\draw (0,0)--(1,0)--(1,1)--(0,1)--(0,0);
		\draw (0,0)--(0.5,1);\draw (0,0)--(1,1);
		
		\draw (1,1)--(1,2)--(0,2)--(0,1);
		\draw (0.5,1.5)--(1,1.5);\draw (0.5,1)--(0.5,2);
		
		\draw (1,2)--(1,3)--(0,3)--(0,2);
		\draw (0.5,2)--(1,3);	\draw (0,2)--(1,3);
		
		\draw (1,3)--(1,4)--(0,4)--(0,3);
		\draw (0.5,3)--(0.5,4);\draw (0,3.5)--(0.5,3.5);
		\node at (0.15,0.75) {$2$};	
		\node at (0.53,0.75) {$0$};
		\node at (0.7,0.3) {$2$};
		\node at (0.75,1.75) {$2$};
		\node at (0.75,1.25) {$1$};
		\node at (0.25,1.5) {$1$};
		\node at (0.85,2.25) {$2$};	
		\node at (0.48,2.25) {$0$};
		\node at (0.3,2.7) {$2$};
		\node at (0.25,3.75) {$2$};
		\node at (0.25,3.25) {$1$};
		\node at (0.75,3.5) {$1$};
		\begin{scope}[shift={(0,4)}]
			\draw (1,4)--(1,4.5);
			\node at (0.5,4.5) {$\vdots$};
			\draw (1,0)--(1,1)--(0,1)--(0,0);
			\draw (0,0)--(0.5,1);\draw (0,0)--(1,1);
			
			\draw (1,1)--(1,2)--(0,2)--(0,1);
			\draw (0.5,1.5)--(1,1.5);\draw (0.5,1)--(0.5,2);
			
			\draw (1,2)--(1,3)--(0,3)--(0,2);
			\draw (0.5,2)--(1,3);	\draw (0,2)--(1,3);
			
			\draw (1,3)--(1,4)--(0,4)--(0,3);
			\draw (0.5,3)--(0.5,4);\draw (0,3.5)--(0.5,3.5);
			\node at (0.15,0.75) {$2$};	
			\node at (0.53,0.75) {$0$};
			\node at (0.7,0.3) {$2$};
			\node at (0.75,1.75) {$2$};
			\node at (0.75,1.25) {$1$};
			\node at (0.25,1.5) {$1$};
			\node at (0.85,2.25) {$2$};	
			\node at (0.48,2.25) {$0$};
			\node at (0.3,2.7) {$2$};
			\node at (0.25,3.75) {$2$};
			\node at (0.25,3.25) {$1$};
			\node at (0.75,3.5) {$1$};
		\end{scope}
	\end{scope}
	
	\begin{scope}[shift={(-1,0)}]
		\draw (1,0)--(1,-0.5);
		\node at (0.5,-0.3) {$\vdots$};
		
		\draw (1,0)--(0,0)--(0,1)--(1,1);
		\draw (0,0)--(0.5,1);\draw (0,0)--(1,1);
		
		\draw (0,1)--(0,2)--(1,2);
		\draw (0.5,1.5)--(1,1.5);\draw (0.5,1)--(0.5,2);
		
		\draw (0,2)--(0,3)--(1,3);
		\draw (0.5,2)--(1,3);	\draw (0,2)--(1,3);
		
		\draw (0,3)--(0,4)--(1,4);
		\draw (0.5,3)--(0.5,4);\draw (0,3.5)--(0.5,3.5);
		\node at (0.15,0.75) {$2$};	
		\node at (0.53,0.75) {$0$};
		\node at (0.7,0.3) {$2$};
		\node at (0.75,1.75) {$2$};
		\node at (0.75,1.25) {$1$};
		\node at (0.25,1.5) {$1$};
		\node at (0.85,2.25) {$2$};	
		\node at (0.48,2.25) {$0$};
		\node at (0.3,2.7) {$2$};
		\node at (0.25,3.75) {$2$};
		\node at (0.25,3.25) {$1$};
		\node at (0.75,3.5) {$1$};
		\begin{scope}[shift={(0,4)}]
			\draw (1,4)--(1,4.5);
			\node at (0.5,4.5) {$\vdots$};
			\draw (0,0)--(0,1)--(1,1);
			\draw (0,0)--(0.5,1);\draw (0,0)--(1,1);
			
			\draw (0,1)--(0,2)--(1,2);
			\draw (0.5,1.5)--(1,1.5);\draw (0.5,1)--(0.5,2);
			
			\draw (0,2)--(0,3)--(1,3);
			\draw (0.5,2)--(1,3);	\draw (0,2)--(1,3);
			
			\draw (0,3)--(0,4)--(1,4);
			\draw (0.5,3)--(0.5,4);\draw (0,3.5)--(0.5,3.5);
			\node at (0.15,0.75) {$2$};	
			\node at (0.53,0.75) {$0$};
			\node at (0.7,0.3) {$2$};
			\node at (0.75,1.75) {$2$};
			\node at (0.75,1.25) {$1$};
			\node at (0.25,1.5) {$1$};
			\node at (0.85,2.25) {$2$};	
			\node at (0.48,2.25) {$0$};
			\node at (0.3,2.7) {$2$};
			\node at (0.25,3.75) {$2$};
			\node at (0.25,3.25) {$1$};
			\node at (0.75,3.5) {$1$};
		\end{scope}
	\end{scope}
	\begin{scope}[shift={(-2,0)}]
		\draw (1,0)--(1,-0.5);
		\node at (0.5,-0.3) {$\vdots$};
		
		\draw (1,0)--(0,0)--(0,1)--(1,1);
		\draw (0,0)--(0.5,1);\draw (0,0)--(1,1);
		
		\draw (0,1)--(0,2)--(1,2);
		\draw (0.5,1.5)--(1,1.5);\draw (0.5,1)--(0.5,2);
		
		\draw (0,2)--(0,3)--(1,3);
		\draw (0.5,2)--(1,3);	\draw (0,2)--(1,3);
		
		\draw (0,3)--(0,4)--(1,4);
		\draw (0.5,3)--(0.5,4);\draw (0,3.5)--(0.5,3.5);
		\node at (0.15,0.75) {$2$};	
		\node at (0.53,0.75) {$0$};
		\node at (0.7,0.3) {$2$};
		\node at (0.75,1.75) {$2$};
		\node at (0.75,1.25) {$1$};
		\node at (0.25,1.5) {$1$};
		\node at (0.85,2.25) {$2$};	
		\node at (0.48,2.25) {$0$};
		\node at (0.3,2.7) {$2$};
		\node at (0.25,3.75) {$2$};
		\node at (0.25,3.25) {$1$};
		\node at (0.75,3.5) {$1$};
		\begin{scope}[shift={(0,4)}]
			\draw (1,4)--(1,4.5);
			\node at (0.5,4.5) {$\vdots$};
			\draw (0,0)--(0,1)--(1,1);
			\draw (0,0)--(0.5,1);\draw (0,0)--(1,1);
			
			\draw (0,1)--(0,2)--(1,2);
			\draw (0.5,1.5)--(1,1.5);\draw (0.5,1)--(0.5,2);
			
			\draw (0,2)--(0,3)--(1,3);
			\draw (0.5,2)--(1,3);	\draw (0,2)--(1,3);
			
			\draw (0,3)--(0,4)--(1,4);
			\draw (0.5,3)--(0.5,4);\draw (0,3.5)--(0.5,3.5);
			\node at (0.15,0.75) {$2$};	
			\node at (0.53,0.75) {$0$};
			\node at (0.7,0.3) {$2$};
			\node at (0.75,1.75) {$2$};
			\node at (0.75,1.25) {$1$};
			\node at (0.25,1.5) {$1$};
			\node at (0.85,2.25) {$2$};	
			\node at (0.48,2.25) {$0$};
			\node at (0.3,2.7) {$2$};
			\node at (0.25,3.75) {$2$};
			\node at (0.25,3.25) {$1$};
			\node at (0.75,3.5) {$1$};
		\end{scope}
	\end{scope}
	\begin{scope}[shift={(-3,0)}]
		\draw (1,0)--(1,-0.5);
		\node at (0.5,-0.3) {$\vdots$};
		
		\draw (1,0)--(0,0)--(0,1)--(1,1);
		\draw (0,0)--(0.5,1);\draw (0,0)--(1,1);
		
		\draw (0,1)--(0,2)--(1,2);
		\draw (0.5,1.5)--(1,1.5);\draw (0.5,1)--(0.5,2);
		
		\draw (0,2)--(0,3)--(1,3);
		\draw (0.5,2)--(1,3);	\draw (0,2)--(1,3);
		
		\draw (0,3)--(0,4)--(1,4);
		\draw (0.5,3)--(0.5,4);\draw (0,3.5)--(0.5,3.5);
		\node at (0.15,0.75) {$2$};	
		\node at (0.53,0.75) {$0$};
		\node at (0.7,0.3) {$2$};
		\node at (0.75,1.75) {$2$};
		\node at (0.75,1.25) {$1$};
		\node at (0.25,1.5) {$1$};
		\node at (0.85,2.25) {$2$};	
		\node at (0.48,2.25) {$0$};
		\node at (0.3,2.7) {$2$};
		\node at (0.25,3.75) {$2$};
		\node at (0.25,3.25) {$1$};
		\node at (0.75,3.5) {$1$};
		\begin{scope}[shift={(0,4)}]
			\draw (1,4)--(1,4.5);
			\node at (0.5,4.5) {$\vdots$};
			\draw (0,0)--(0,1)--(1,1);
			\draw (0,0)--(0.5,1);\draw (0,0)--(1,1);
			
			\draw (0,1)--(0,2)--(1,2);
			\draw (0.5,1.5)--(1,1.5);\draw (0.5,1)--(0.5,2);
			
			\draw (0,2)--(0,3)--(1,3);
			\draw (0.5,2)--(1,3);	\draw (0,2)--(1,3);
			
			\draw (0,3)--(0,4)--(1,4);
			\draw (0.5,3)--(0.5,4);\draw (0,3.5)--(0.5,3.5);
			\node at (0.15,0.75) {$2$};	
			\node at (0.53,0.75) {$0$};
			\node at (0.7,0.3) {$2$};
			\node at (0.75,1.75) {$2$};
			\node at (0.75,1.25) {$1$};
			\node at (0.25,1.5) {$1$};
			\node at (0.85,2.25) {$2$};	
			\node at (0.48,2.25) {$0$};
			\node at (0.3,2.7) {$2$};
			\node at (0.25,3.75) {$2$};
			\node at (0.25,3.25) {$1$};
			\node at (0.75,3.5) {$1$};
		\end{scope}
	\end{scope}
	\begin{scope}[shift={(-4,0)}]
		\draw (0,0)--(-0.5,0);
		\draw (0,1)--(-0.5,1);
		\draw (0,2)--(-0.5,2);
		\draw (0,3)--(-0.5,3);
		\node at (-0.4,0.5) {$\cdots$};
		\node at (-0.4,1.5) {$\cdots$};
		\node at (-0.4,2.5) {$\cdots$};
		\node at (-0.4,3.5) {$\cdots$};
		\draw (0,0)--(0,-0.5);
		\draw (1,0)--(1,-0.5);
		\node at (0.5,-0.3) {$\vdots$};
		
		\draw (1,0)--(0,0)--(0,1)--(1,1);
		\draw (0,0)--(0.5,1);\draw (0,0)--(1,1);
		
		\draw (0,1)--(0,2)--(1,2);
		\draw (0.5,1.5)--(1,1.5);\draw (0.5,1)--(0.5,2);
		
		\draw (0,2)--(0,3)--(1,3);
		\draw (0.5,2)--(1,3);	\draw (0,2)--(1,3);
		
		\draw (0,3)--(0,4)--(1,4);
		\draw (0.5,3)--(0.5,4);\draw (0,3.5)--(0.5,3.5);
		\node at (0.15,0.75) {$2$};	
		\node at (0.53,0.75) {$0$};
		\node at (0.7,0.3) {$2$};
		\node at (0.75,1.75) {$2$};
		\node at (0.75,1.25) {$1$};
		\node at (0.25,1.5) {$1$};
		\node at (0.85,2.25) {$2$};	
		\node at (0.48,2.25) {$0$};
		\node at (0.3,2.7) {$2$};
		\node at (0.25,3.75) {$2$};
		\node at (0.25,3.25) {$1$};
		\node at (0.75,3.5) {$1$};
		\begin{scope}[shift={(0,4)}]
			
			\draw (0,0)--(-0.5,0);
			\draw (0,1)--(-0.5,1);
			\draw (0,2)--(-0.5,2);
			\draw (0,3)--(-0.5,3);
			\draw (0,4)--(-0.5,4);
			\node at (-0.4,0.5) {$\cdots$};
			\node at (-0.4,1.5) {$\cdots$};
			\node at (-0.4,2.5) {$\cdots$};
			\node at (-0.4,3.5) {$\cdots$};
			\draw (0,4)--(0,4.5);
			\draw (1,4)--(1,4.5);
			\node at (0.5,4.5) {$\vdots$};
			\draw (0,0)--(0,1)--(1,1);
			\draw (0,0)--(0.5,1);\draw (0,0)--(1,1);
			
			\draw (0,1)--(0,2)--(1,2);
			\draw (0.5,1.5)--(1,1.5);\draw (0.5,1)--(0.5,2);
			
			\draw (0,2)--(0,3)--(1,3);
			\draw (0.5,2)--(1,3);	\draw (0,2)--(1,3);
			
			\draw (0,3)--(0,4)--(1,4);
			\draw (0.5,3)--(0.5,4);\draw (0,3.5)--(0.5,3.5);
			\node at (0.15,0.75) {$2$};	
			\node at (0.53,0.75) {$0$};
			\node at (0.7,0.3) {$2$};
			\node at (0.75,1.75) {$2$};
			\node at (0.75,1.25) {$1$};
			\node at (0.25,1.5) {$1$};
			\node at (0.85,2.25) {$2$};	
			\node at (0.48,2.25) {$0$};
			\node at (0.3,2.7) {$2$};
			\node at (0.25,3.75) {$2$};
			\node at (0.25,3.25) {$1$};
			\node at (0.75,3.5) {$1$};
		\end{scope}
	\end{scope}
\end{scope}		
		\end{tikzpicture}
	}
	\vskip -0.5pc
	\caption{The wall pattern}\label{wall pattern}
\end{figure}


\begin{figure}[H]
	{\small
		\centering
		\begin{tikzpicture}[scale=0.8]
			\begin{scope}[shift={(-7,0)}]	
				\node at (-1.6,-2.3) {\scriptsize{Ground-state wall $Y_{\Lambda_0}$ for $U_q(D_4^{(3)})$}};		
			\draw (1,-1)--(1,-1.5);
			\node at (0.5,-1.3) {$\vdots$};
			\draw (1,0)--(1,-1);
			\draw (0,-1)--(1,-1);
			\draw (0,-1)--(1,0);
			\node at (0.75,-0.7) {$1$};
			\node at (0.25,-0.3) {$1$};
			\draw (0,0)--(1,0)--(1,1)--(0,1);
			\draw (0,0)--(0.5,1);
			\node at (0.65,0.4) {$2$};	
			\draw (1,1)--(1,2)--(0,2);
			\draw (0,1)--(1,2);
			\node at (0.75,1.3) {$1$};
			\draw (1,2)--(1,3)--(0,3);
			\draw (0.5,2)--(1,3);
			\node at (0.8,2.25) {$0$};
			
			\begin{scope}[shift={(-1,0)}]
				\draw (1,-1)--(1,-1.5);
				\node at (0.5,-1.3) {$\vdots$};
				\draw (1,0)--(1,-1);
				\draw (0,-1)--(1,-1);
				\draw (0,-1)--(1,0);
				\node at (0.75,-0.7) {$1$};
				\node at (0.25,-0.3) {$1$};
				\draw (0,0)--(1,0)--(1,1)--(0,1);
				\draw (0,0)--(0.5,1);
				\node at (0.65,0.4) {$2$};	
				\draw (1,1)--(1,2)--(0,2);
				\draw (0,1)--(1,2);
				\node at (0.75,1.3) {$1$};
				\draw (1,2)--(1,3)--(0,3);
				\draw (0.5,2)--(1,3);
				\node at (0.8,2.25) {$0$};
				
			\end{scope}
			\begin{scope}[shift={(-2,0)}]
				\draw (1,-1)--(1,-1.5);
				\node at (0.5,-1.3) {$\vdots$};
				\draw (1,0)--(1,-1);
				\draw (0,-1)--(1,-1);
				\draw (0,-1)--(1,0);
				\node at (0.75,-0.7) {$1$};
				\node at (0.25,-0.3) {$1$};
				\draw (0,0)--(1,0)--(1,1)--(0,1);
				\draw (0,0)--(0.5,1);
				\node at (0.65,0.4) {$2$};	
				\draw (1,1)--(1,2)--(0,2);
				\draw (0,1)--(1,2);
				\node at (0.75,1.3) {$1$};
				\draw (1,2)--(1,3)--(0,3);
				\draw (0.5,2)--(1,3);
				\node at (0.8,2.25) {$0$};
				
			\end{scope}
			\begin{scope}[shift={(-3,0)}]
				\draw (1,-1)--(1,-1.5);
				\node at (0.5,-1.3) {$\vdots$};
				\draw (1,0)--(1,-1);
				\draw (0,-1)--(1,-1);
				\draw (0,-1)--(1,0);
				\node at (0.75,-0.7) {$1$};
				\node at (0.25,-0.3) {$1$};
				\draw (0,0)--(1,0)--(1,1)--(0,1);
				\draw (0,0)--(0.5,1);
				\node at (0.65,0.4) {$2$};	
				\draw (1,1)--(1,2)--(0,2);
				\draw (0,1)--(1,2);
				\node at (0.75,1.3) {$1$};
				\draw (1,2)--(1,3)--(0,3);
				\draw (0.5,2)--(1,3);
				\node at (0.8,2.25) {$0$};
				
			\end{scope}
			\begin{scope}[shift={(-4,0)}]
				\draw (0,-1)--(0,-1.5);
				\draw (1,-1)--(1,-1.5);
				\node at (0.5,-1.3) {$\vdots$};
				\draw (0,0)--(0,-1);
				\draw (1,0)--(1,-1);
				\draw (0,-1)--(1,-1);
				\draw (0,-1)--(1,0);
				\node at (0.75,-0.7) {$1$};
				\node at (0.25,-0.3) {$1$};
				\draw (0,0)--(1,0)--(1,1)--(0,1)--(0,0);
				\draw (0,0)--(0.5,1);
				\node at (0.65,0.4) {$2$};	
				\draw (1,1)--(1,2)--(0,2)--(0,1);
				\draw (0,1)--(1,2);
				\node at (0.75,1.3) {$1$};
				\draw (1,2)--(1,3)--(0,3)--(0,2);
				\draw (0.5,2)--(1,3);
				\node at (0.8,2.25) {$0$};
				
				\draw (0,-1)--(-0.5,-1);
				\draw (0,0)--(-0.5,0);
				\draw (0,1)--(-0.5,1);
				\draw (0,2)--(-0.5,2);
				\draw (0,3)--(-0.5,3);
				\node at (-0.4,-0.5) {$\cdots$};
				\node at (-0.4,0.5) {$\cdots$};
				\node at (-0.4,1.5) {$\cdots$};
				\node at (-0.4,2.5) {$\cdots$};
			\end{scope}
			\end{scope}
			\begin{scope}[shift={(0,0)}]
	\node at (-1.6,-2.3) {\scriptsize{Ground-state wall $Y_{\Lambda_0}$ for $U_q(G_2^{(1)})$}};					
	\draw (0,0)--(0,-1.5);
\draw (1,0)--(1,-1.5);
\node at (0.5,-1.3) {$\vdots$};
\draw (0,-1)--(1,-1);\draw (0.5,0)--(0.5,-1); \draw (0,-0.5)--(0.5,-0.5);
\node at (0.25,-0.25) {$2$};\node at (0.25,-0.75) {$1$};\node at (0.75,-0.5) {$1$};	
\draw (0,0)--(1,0)--(1,1)--(0,1)--(0,0);
\draw (0,0)--(0.5,1);\draw (0,0)--(1,1);

\draw (1,1)--(1,2)--(0,2)--(0,1);
\draw (0.5,1.5)--(1,1.5);\draw (0.5,1)--(0.5,2);

\draw (1,2)--(1,3)--(0,3)--(0,2);
\draw (0.5,2)--(1,3);	\draw (0,2)--(1,3);

\node at (0.15,0.75) {$2$};	
\node at (0.7,0.3) {$2$};
\node at (0.25,1.5) {$1$};
\node at (0.48,2.25) {$0$};
\begin{scope}[shift={(-1,0)}]
	\draw (0,0)--(0,-1.5);
\draw (1,0)--(1,-1.5);
\node at (0.5,-1.3) {$\vdots$};
\draw (0,-1)--(1,-1);\draw (0.5,0)--(0.5,-1); \draw (0,-0.5)--(0.5,-0.5);
\node at (0.25,-0.25) {$2$};\node at (0.25,-0.75) {$1$};\node at (0.75,-0.5) {$1$};	
	
	\draw (1,0)--(0,0)--(0,1)--(1,1);
	\draw (0,0)--(0.5,1);\draw (0,0)--(1,1);
	
	\draw (0,1)--(0,2)--(1,2);
	\draw (0.5,1.5)--(1,1.5);\draw (0.5,1)--(0.5,2);
	
	\draw (0,2)--(0,3)--(1,3);
	\draw (0.5,2)--(1,3);	\draw (0,2)--(1,3);
	
	\node at (0.15,0.75) {$2$};	
	\node at (0.7,0.3) {$2$};
	\node at (0.25,1.5) {$1$};
	\node at (0.48,2.25) {$0$};
\end{scope}
\begin{scope}[shift={(-2,0)}]
	\draw (0,0)--(0,-1.5);
\draw (1,0)--(1,-1.5);
\node at (0.5,-1.3) {$\vdots$};
\draw (0,-1)--(1,-1);\draw (0.5,0)--(0.5,-1); \draw (0,-0.5)--(0.5,-0.5);
\node at (0.25,-0.25) {$2$};\node at (0.25,-0.75) {$1$};\node at (0.75,-0.5) {$1$};	
	
	\draw (1,0)--(0,0)--(0,1)--(1,1);
	\draw (0,0)--(0.5,1);\draw (0,0)--(1,1);
	
	\draw (0,1)--(0,2)--(1,2);
	\draw (0.5,1.5)--(1,1.5);\draw (0.5,1)--(0.5,2);
	
	\draw (0,2)--(0,3)--(1,3);
	\draw (0.5,2)--(1,3);	\draw (0,2)--(1,3);
	
	\node at (0.15,0.75) {$2$};	
	\node at (0.7,0.3) {$2$};
	\node at (0.25,1.5) {$1$};
	\node at (0.48,2.25) {$0$};
\end{scope}
\begin{scope}[shift={(-3,0)}]
	\draw (0,0)--(0,-1.5);
\draw (1,0)--(1,-1.5);
\node at (0.5,-1.3) {$\vdots$};
\draw (0,-1)--(1,-1);\draw (0.5,0)--(0.5,-1); \draw (0,-0.5)--(0.5,-0.5);
\node at (0.25,-0.25) {$2$};\node at (0.25,-0.75) {$1$};\node at (0.75,-0.5) {$1$};	
	
	\draw (1,0)--(0,0)--(0,1)--(1,1);
	\draw (0,0)--(0.5,1);\draw (0,0)--(1,1);
	
	\draw (0,1)--(0,2)--(1,2);
	\draw (0.5,1.5)--(1,1.5);\draw (0.5,1)--(0.5,2);
	
	\draw (0,2)--(0,3)--(1,3);
	\draw (0.5,2)--(1,3);	\draw (0,2)--(1,3);
	
	\node at (0.15,0.75) {$2$};	
	\node at (0.7,0.3) {$2$};
	\node at (0.25,1.5) {$1$};
	\node at (0.48,2.25) {$0$};
\end{scope}
\begin{scope}[shift={(-4,0)}]
	\draw (0,0)--(0,-1.5);
\draw (1,0)--(1,-1.5);
\node at (0.5,-1.3) {$\vdots$};
\draw (0,-1)--(1,-1);\draw (0.5,0)--(0.5,-1); \draw (0,-0.5)--(0.5,-0.5);
\node at (0.25,-0.25) {$2$};\node at (0.25,-0.75) {$1$};\node at (0.75,-0.5) {$1$};	
	\draw (0,-1)--(-0.5,-1);
	\draw (0,0)--(-0.5,0);
	\draw (0,1)--(-0.5,1);
	\draw (0,2)--(-0.5,2);
	\draw (0,3)--(-0.5,3);
	\node at (-0.4,-0.5){$\cdots$};
	\node at (-0.4,0.5){$\cdots$};
	\node at (-0.4,1.5){$\cdots$};
	\node at (-0.4,2.5){$\cdots$};

	\draw (1,0)--(0,0)--(0,1)--(1,1);
	\draw (0,0)--(0.5,1);\draw (0,0)--(1,1);
	
	\draw (0,1)--(0,2)--(1,2);
	\draw (0.5,1.5)--(1,1.5);\draw (0.5,1)--(0.5,2);
	
	\draw (0,2)--(0,3)--(1,3);
	\draw (0.5,2)--(1,3);	\draw (0,2)--(1,3);
	
	\node at (0.15,0.75) {$2$};	
	\node at (0.7,0.3) {$2$};
	\node at (0.25,1.5) {$1$};
	\node at (0.48,2.25) {$0$};
\end{scope}				
				\end{scope}
\begin{scope}[shift={(7,-1)}]
	\node at (-1.6,-1.3) {\scriptsize{Ground-state wall $Y_{\Lambda_2}$ for $U_q(G_2^{(1)})$}};		
	\draw (0,0)--(0,-0.5);
\draw (1,0)--(1,-0.5);
\node at (0.5,-0.3) {$\vdots$};
\draw (0,0)--(0,1)--(1,1)--(1,0)--(0,0); \draw (0,0)--(1,1);\draw (0,0)--(0.5,1);
\draw (0,1)--(0,2)--(1,2)--(1,1);\draw (0.5,1)--(0.5,2);
\draw (0,2)--(0,3)--(1,3)--(1,2);\draw (0,2)--(1,3);\draw (0.5,2)--(1,3);
\draw (0.5,1.5)--(1,1.5);\node at (0.75,1.25) {$1$};\node at (0.75,1.75) {$2$};
\node at (0.15,0.75) {$2$};\node at (0.53,0.75) {$0$};\node at (0.48,2.25) {$0$};
\node at (0.7,0.3) {$2$};\node at (0.3,2.7) {$2$};
\node at (0.25,1.5) {$1$};
\draw (0,3)--(0,4)--(1,4)--(1,3);\draw (0.5,3)--(0.5,4);\draw (0,3.5)--(0.5,3.5);
\node at (0.25,3.25) {$1$};\node at (0.25,3.75) {$2$};
\begin{scope}[shift={(-1,0)}]
	\draw (0,0)--(0,-0.5);
	\node at (0.5,-0.3) {$\vdots$};
	\draw (1,0)--(0,0)--(0,1)--(1,1); \draw (0,0)--(1,1);\draw (0,0)--(0.5,1);
	\draw (0,1)--(0,2)--(1,2);\draw (0.5,1)--(0.5,2);
	\draw (0,2)--(0,3)--(1,3);\draw (0,2)--(1,3);\draw (0.5,2)--(1,3);
	\draw (0.5,1.5)--(1,1.5);\node at (0.75,1.25) {$1$};\node at (0.75,1.75) {$2$};
	\node at (0.15,0.75) {$2$};\node at (0.53,0.75) {$0$};\node at (0.48,2.25) {$0$};
	\node at (0.7,0.3) {$2$};\node at (0.3,2.7) {$2$};
	\node at (0.25,1.5) {$1$};
	\draw (0,3)--(0,4)--(1,4);\draw (0.5,3)--(0.5,4);\draw (0,3.5)--(0.5,3.5);
	\node at (0.25,3.25) {$1$};\node at (0.25,3.75) {$2$};
\end{scope}
\begin{scope}[shift={(-2,0)}]
	\draw (0,0)--(0,-0.5);
	\node at (0.5,-0.3) {$\vdots$};
	\draw (1,0)--(0,0)--(0,1)--(1,1); \draw (0,0)--(1,1);\draw (0,0)--(0.5,1);
	\draw (0,1)--(0,2)--(1,2);\draw (0.5,1)--(0.5,2);
	\draw (0,2)--(0,3)--(1,3);\draw (0,2)--(1,3);\draw (0.5,2)--(1,3);
	\draw (0.5,1.5)--(1,1.5);\node at (0.75,1.25) {$1$};\node at (0.75,1.75) {$2$};
	\node at (0.15,0.75) {$2$};\node at (0.53,0.75) {$0$};\node at (0.48,2.25) {$0$};
	\node at (0.7,0.3) {$2$};\node at (0.3,2.7) {$2$};
	\node at (0.25,1.5) {$1$};
	\draw (0,3)--(0,4)--(1,4);\draw (0.5,3)--(0.5,4);\draw (0,3.5)--(0.5,3.5);
	\node at (0.25,3.25) {$1$};\node at (0.25,3.75) {$2$};
\end{scope}
\begin{scope}[shift={(-3,0)}]
	\draw (0,0)--(0,-0.5);
	\node at (0.5,-0.3) {$\vdots$};
	\draw (1,0)--(0,0)--(0,1)--(1,1); \draw (0,0)--(1,1);\draw (0,0)--(0.5,1);
	\draw (0,1)--(0,2)--(1,2);\draw (0.5,1)--(0.5,2);
	\draw (0,2)--(0,3)--(1,3);\draw (0,2)--(1,3);\draw (0.5,2)--(1,3);
	\draw (0.5,1.5)--(1,1.5);\node at (0.75,1.25) {$1$};\node at (0.75,1.75) {$2$};
	\node at (0.15,0.75) {$2$};\node at (0.53,0.75) {$0$};\node at (0.48,2.25) {$0$};
	\node at (0.7,0.3) {$2$};\node at (0.3,2.7) {$2$};
	\node at (0.25,1.5) {$1$};
	\draw (0,3)--(0,4)--(1,4);\draw (0.5,3)--(0.5,4);\draw (0,3.5)--(0.5,3.5);
	\node at (0.25,3.25) {$1$};\node at (0.25,3.75) {$2$};
\end{scope}
\begin{scope}[shift={(-4,0)}]
	\draw (0,0)--(0,-0.5);
	\node at (0.5,-0.3) {$\vdots$};
	\draw (1,0)--(0,0)--(0,1)--(1,1); \draw (0,0)--(1,1);\draw (0,0)--(0.5,1);
	\draw (0,1)--(0,2)--(1,2);\draw (0.5,1)--(0.5,2);
	\draw (0,2)--(0,3)--(1,3);\draw (0,2)--(1,3);\draw (0.5,2)--(1,3);
	\draw (0.5,1.5)--(1,1.5);\node at (0.75,1.25) {$1$};\node at (0.75,1.75) {$2$};
	\node at (0.15,0.75) {$2$};\node at (0.53,0.75) {$0$};\node at (0.48,2.25) {$0$};
	\node at (0.7,0.3) {$2$};\node at (0.3,2.7) {$2$};
	\node at (0.25,1.5) {$1$};
	\draw (0,3)--(0,4)--(1,4);\draw (0.5,3)--(0.5,4);\draw (0,3.5)--(0.5,3.5);
	\node at (0.25,3.25) {$1$};\node at (0.25,3.75) {$2$};				
	\draw (0,0)--(-0.5,0);
	\draw (0,1)--(-0.5,1);
	\draw (0,2)--(-0.5,2);
	\draw (0,3)--(-0.5,3);
	\draw (0,4)--(-0.5,4);
	\node at (-0.4,0.5){$\cdots$};
	\node at (-0.4,1.5){$\cdots$};
	\node at (-0.4,2.5){$\cdots$};
	\node at (-0.4,3.5){$\cdots$};
\end{scope}
\end{scope}			
		\end{tikzpicture}
	}
	\vskip -1pc
	\caption{The ground-state wall}\label{Fig:ground-state wall}
\end{figure}
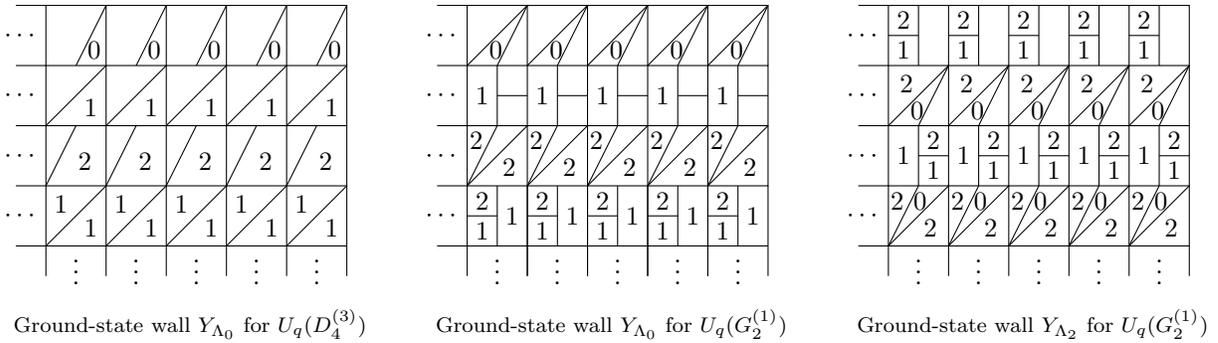 

The Young columns in the ground-state walls are called {\it ground-state columns}.

\begin{defn}\label{def:Young wall}
A wall built on the ground-state wall following the rules listed below is called a {\it Young wall }.
\begin{enumerate}
	\item The walls must be built on the top of the ground-state wall.
	\item The way of stacking blocks satisfies the wall pattern.
	\item Every column in the wall must be a Young column.
	\item In adjacent columns, for each block in the left column, there must be a block
	lying in the same position in the right column.
\end{enumerate}
	\end{defn}

\subsection{Energy function}

\begin{defn}
An \textit{energy function} on a crystal $B$ is a $\mathbb Z$-valued function $H:B\otimes B\to\mathbb Z$ satisfying the following conditions.
\begin{equation*}
	H(\tilde{f}_i(b_1\otimes b_2))=
	\begin{cases}
	H(b_1\otimes b_2), &\text{if } i\neq 0,\\
	H(b_1\otimes b_2)-1, &\text{if } i=0,\ \varphi_0(b_1)>\varepsilon_0(b_2),\\
	H(b_1\otimes b_2)+1, &\text{if } i=0,\ \varphi_0(b_1)\leq\varepsilon_0(b_2),
	\end{cases}
\end{equation*}	
for all $i\in I$, $b_1\otimes b_2\in B\otimes B$ with $\tilde{f}_i(b_1\otimes b_2)\in B\otimes B$.
	\end{defn}

Let $B^{\mathrm{aff}}:= \{b(n) \mid b \in B, n \in \Z \}$ be the {\it affinization} of $B$.
The crystal structure on $B^{\mathrm{aff}}$ is given as follows.
\begin{equation*}
\begin{aligned}
& \tilde{e}_{i} (b(n)) = (\tilde{e}_{i} b) (n), \ \ 
\tilde{f}_{i}(b(n)) = (\tilde{f}_{i} b)(n) \ \ \text{for} \ \ i \neq 0, \\ 
& \tilde{e}_{0} (b(n)) = (\tilde{e}_{i} b)(n-1), \ \  \tilde{f}_{0} (b(n)) = (\tilde{f}_{0} b)(n+1).
\end{aligned}
\end{equation*} 

\begin{defn}\label{affine energy function}
	The affine energy function $H^{\mathrm{aff}}:B^{\mathrm{aff}}\otimes B^{\mathrm{aff}}\to \mathbb Z$ is given as follows.
	\begin{equation*}
		H^{\mathrm{aff}}(x(m)\otimes y(n))=H(x\otimes y)+m-n,
	\end{equation*} 
	where $x(m),y(n)\in B^{\mathrm{aff}}$ and $m,n\in\mathbb Z$.
\end{defn}

\begin{defn}
	A \textit{combinatorical $R$-matrix} on $B^{\mathrm{aff}}$ is an endomorphism of affine crystals $R:B^{\mathrm{aff}}\otimes B^{\mathrm{aff}}\to B^{\mathrm{aff}}\otimes B^{\mathrm{aff}}$ such that
	\begin{equation}\label{rmatrix}
		\begin{aligned}
			& (T\otimes\mathrm{id})\circ R=R\circ(\mathrm{id}\otimes T),\\
			& (\mathrm{id}\otimes T)\circ R=R\circ(T\otimes \mathrm{id}),
		\end{aligned}
	\end{equation}
	where $T:B^{\mathrm{aff}}\to B^{\mathrm{aff}}$ is the shift operator given by
	\begin{equation*}
		T(x(m))=x(m+1).
	\end{equation*}
\end{defn}

\begin{Lem}\label{rmatrix R}
	We define a map $R:B^{\mathrm{aff}}\to B^{\mathrm{aff}}$ by	
	\begin{equation*}	
		R(x(m)\otimes y(n))=x(n-H(x\otimes y))\otimes y(m+H(x\otimes y))
	\end{equation*}
	for $x,y\in B$, $m,n\in\mathbb Z$. Then $R$ is a combinatorial $R$-matrix.
\end{Lem}

\begin{proof}
	We first verify the relations in \eqref{rmatrix}.
	\begin{align*}
		&(T\otimes\mathrm{id})\circ R(x(m)\otimes y(n))\\
		=&(T\otimes\mathrm{id})(x(n-H(x\otimes y))\otimes y(m+H(x\otimes y)))\\
		=&T(x(n-H(x\otimes y)))\otimes y(m+H(x\otimes y))\\
		=&x(n+1-H(x\otimes y))\otimes y(m+H(x\otimes y)\\
		=&R(x(m)\otimes y(n+1))\\
		=&R(x(m)\otimes T(y(n)))\\
		=&R\circ(\mathrm{id}\otimes T)(x(m)\otimes y(n)).
	\end{align*}
	
	Similarly, we have
	\begin{align*}
		&(\mathrm{id}\otimes T)\circ R(x(m)\otimes y(n))\\
		=&(\mathrm{id}\otimes T)(x(n-H(x\otimes y))\otimes y(m+H(x\otimes y)))\\
		=&x(n-H(x\otimes y))\otimes T(y(m+H(x\otimes y)))\\
		=&x(n-H(x\otimes y))\otimes y(m+1+H(x\otimes y)\\
		=&R(x(m+1)\otimes y(n))\\
		=&R(T(x(m)\otimes y(n))\\
		=&R\circ(T\otimes \mathrm{id})(x(m)\otimes y(n)).
	\end{align*}
	
	Next, we shall show that $R$ commutes with the action of Kashiwara operators. 
	\begin{enumerate}
		\item[(1)] If $\varphi_0(x)>\varepsilon_0(y)$, we have
		\begin{align*}
			&R(\tilde{f}_0(x(m)\otimes y(n)))\\
			=&R((\tilde{f}_0x(m))\otimes y(n))\\
			=&R((\tilde{f}_0x)(m+1)\otimes y(n))\\
			=&(\tilde{f}_0x)(n-H(\tilde{f}_0x\otimes y))\otimes y(m+1+H(\tilde{f}_0x\otimes y))\\
			=&(\tilde{f}_0x)(n+1-H(x\otimes y))\otimes y(m+H(x\otimes y))\\
			=&\tilde{f}_0(x(n-H(x\otimes y)))\otimes y(m+H(x\otimes y))\\
			=&\tilde{f}_0(R(x(m)\otimes y(n))).
		\end{align*}
		
		\item[(2)] If $\varphi_0(x)\leq\varepsilon_0(y)$, we have
		\begin{align*}
			&R(\tilde{f}_0(x(m)\otimes y(n)))\\
			=&R(x(m)\otimes \tilde{f}_0y(n))\\
			=&R(x(m)\otimes (\tilde{f}_0y)(n+1))\\
			=&x(n+1-H(x\otimes\tilde{f}_0 y))\otimes (\tilde{f}_0y)(m+H(x\otimes\tilde{f}_0 y))\\
			=&x(n-H(x\otimes y))\otimes (\tilde{f}_0y)(m+1+H(x\otimes y))\\
			=&x(n-H(x\otimes y))\otimes\tilde{f}_0(y(m+H(x\otimes y)))\\
			=&\tilde{f}_0(R(x(m)\otimes y(n))).
		\end{align*}	
	\end{enumerate}
	
	Similarly, one can prove
	\begin{equation*}
		R(\tilde{f}_i(x(m)\otimes y(n)))=\tilde{f}_i(R(x(m)\otimes y(n)))
	\end{equation*}
	for $i\neq 0$, which completes the proof.
\end{proof}	

\begin{Lem}\label{constant}
	The affine energy function $H^{\mathrm{aff}}$ is constant on each connected component
	of $B^{\mathrm{aff}}\otimes B^{\mathrm{aff}}$.
\end{Lem}

\begin{proof}
By the definition of affine energy function, combinatorial $R$-matrix and Lemma \ref{rmatrix R}, 
we have
	\begin{align*}
		&R(x(m)\otimes y(n))\\
		=&x(n-H(x\otimes y))\otimes y(m+H(x\otimes y))\\
		=&x(m-H^{\mathrm{aff}}(x(m)\otimes y(n)))\otimes y(n+H^{\mathrm{aff}}(x(m)\otimes y(n)))\\
		=&T^{-H^{\mathrm{aff}}(x(m)\otimes y(n))} x(m)\otimes T^{H^{\mathrm{aff}}(x(m)\otimes y(n))}y(n)\\
		=&(T^{-H^{\mathrm{aff}}(x(m)\otimes y(n))}\otimes  T^{H^{\mathrm{aff}}(x(m)\otimes y(n))})(x(m)\otimes y(n)).
	\end{align*}

	We replace $x(m)\otimes y(n)$ by $\tilde{f}_i(x(m)\otimes y(n))$, and then we get
	\begin{align*}
		R(\tilde{f}_i(x(m)\otimes y(n)))=(T^{-H^{\mathrm{aff}}(\tilde{f}_i(x(m)\otimes y(n)))}\otimes  T^{H^{\mathrm{aff}}(\tilde{f}_i(x(m)\otimes y(n)))})(\tilde{f}_i(x(m)\otimes y(n))).
	\end{align*}
	
	We also have
	\begin{align*}
		\tilde{f}_i(R(x(m)\otimes y(n)))=(T^{-H^{\mathrm{aff}}(x(m)\otimes y(n))}\otimes  T^{H^{\mathrm{aff}}(x(m)\otimes y(n))})(\tilde{f}_i(x(m)\otimes y(n))).
	\end{align*}
	
	Since $R\circ\tilde{f}_i=\tilde{f}_i\circ R$, we have
	\begin{align*}
		H^{\mathrm{aff}}(\tilde{f}_i(x(m)\otimes y(n)))=H^{\mathrm{aff}}(x(m)\otimes y(n)).
	\end{align*}
	
	Similarly, one can show that
	\begin{align*}
		H^{\mathrm{aff}}(\tilde{e}_i(x(m)\otimes y(n)))=H^{\mathrm{aff}}(x(m)\otimes y(n)).
	\end{align*}
	
	This completes the proof.
\end{proof}

We now consider the energy functions $H:C_1\otimes C_1\to\mathbb Z$ and $H':C'_1\otimes C'_1\to\mathbb Z$.

For the case of $U_q(D_4^{(3)})$, we set
\begin{align*}
&A_0=\{c_0,c_1,c_2,c_3\},\quad \bar{A}_0=\{c_0,c_{\bar{1}},c_{\bar{2}},c_{\bar{3}}\},\quad A_1=C_1\setminus\{c_\phi\},\\
&A_2=\{c_2,c_3,c_0,c_{\bar{3}},c_{\bar{2}}\},\quad A_3=\{c_3,c_0,c_{\bar{3}}\},
\end{align*}
and
\begin{align*}
&\mathbf A=\{(c_\phi,c_\phi), (c_2,c_{\bar{2}})\}\cup A_0\times \{c_{\bar{1}}\}\cup \{c_1\}\times \bar{A}_0,\\
&\mathbf B=\cup_{i=1}^3(A_i\times \{c_i\}\cup \{c_{\bar{i}}\}\times A_i).
\end{align*}

For any $(x,y)\in C_1\times C_1$, we have 
\begin{equation} \label{value of energy function D43}
	H(x\otimes y)=\begin{cases}
		0,& \text{if}\ (x,y)\in \mathbf A,\\
		2,& \text{if}\ (x,y)\in \mathbf B,\\
		1,& \text{otherwise}.
	\end{cases}
\end{equation}

We list the values of energy function $H$ as follows.

\begin{equation}\label{table D43}
	\begin{tabular}{|c|c|c|c|c|c|c|c|c|}
		\hline \multicolumn{9}{|c|}{$H(c_i\otimes c_j)$}\\	
		\hline
		\diagbox[width=3em]{$i$}{$j$} & $\phi$ & 1 & 2 & 3 & 0 & $\bar{3}$ & $\bar{2}$ & $\bar{1}$ \\
		\hline
		$\phi$ &  0 &  1 &  1 &  1 & 1 & 1 & 1 & 1\\
		\hline	1 &  1 &  2 &  1 &  1 & 0 & 0 & 0 & 0\\
		\hline	2 &  1 &  2 &  2 &  1 & 1 & 1 & 0 & 0\\
		\hline	3 &  1 &  2 &  2 &  2 & 1 & 1 & 1 & 0\\
		\hline	0 &  1 &  2 &  2 &  2 & 1 & 1 & 1 & 0\\
		\hline	$\bar{3}$ &  1 &  2 &  2 &  2 & 2 & 2 & 1 & 1\\
		\hline	$\bar{2}$ &  1 &  2 &  2 &  2 & 2 & 2 & 2 & 1\\
		\hline	$\bar{1}$ &  1 &  2 &  2 &  2 & 2 & 2 & 2 & 2\\
		\hline
	\end{tabular}
\end{equation}

For the case of $U_q(G_2^{(1)})$, we set
\begin{align*}
D_0&=\{c'_1,c'_2,c'_3,c'_4,c'_5,c'_{\bar{7}}\},  &\bar{D}_0&=\{c'_{\bar{1}},c'_{\bar{2}},c'_{\bar{3}},c'_{\bar{4}},c'_{\bar{5}},c'_{\bar{7}}\},
&D'_0&=\{c'_5,c'_7,c'_{\bar{3}}\},  &\bar{D}'_0&=\{c'_3,c'_7,c'_{\bar{5}}\},\\
D_1&=C'_1\setminus \{c'_0\},  &\bar{D}_1&=D_1,
&D_2&=C'_1\setminus \{c'_0,c'_1,c'_{\bar{1}}\},  &\bar{D}_2&=D_2,\\
D_3&=\{c'_{\bar{5}},c'_6,c'_7,c'_{\bar{6}},c'_{\bar{4}},c'_{\bar{3}}\},  &\bar{D}_3&=\{c'_{\bar{6}},c'_5,c'_7,c'_6,c'_4,c'_3\},
&D_4&=D_3\setminus \{c'_{\bar{3}}\},  &\bar{D}_4&=\bar{D}_3\setminus \{c'_3\},\\
D_5&=\{c'_{\bar{6}}\},  &\bar{D}_5&=\{c'_6\},
&D_6&=\{c'_6,c'_7,c'_{\bar{6}}\},  &\bar{D}_6&=D_6,
\end{align*}
and
\begin{align*}
&\mathbf X=\{(c'_0,c'_0)\}\cup D_0\times \bar{D}_0\cup \{c'_1\}\times D'_0\cup \bar{D}'_0\times \{c'_{\bar{1}}\},\\
&\mathbf Y=\cup_{i=1}^6(D_i\times \{c'_i\}\cup \{c'_{\bar{i}}\}\times \bar{D}_i).
\end{align*}

For any $(x,y)\in C'_1\times C'_1$, we have 
\begin{equation} \label{value of energy function G21}
H'(x\otimes y)=\begin{cases}
0,& \text{if}\ (x,y)\in \mathbf X,\\
2,& \text{if}\ (x,y)\in \mathbf Y,\\
1,& \text{otherwise}.
\end{cases}
\end{equation}

We list the values of the energy function $H'$ as follows. 

\vskip -2mm

\begin{equation}\label{table G21}
\begin{tabular}{|c|c|c|c|c|c|c|c|c|c|c|c|c|c|c|c|}
\hline \multicolumn{16}{|c|}{$H'(c'_i\otimes c'_j)$}\\	
 \hline
\diagbox[width=3em]{$i$}{$j$} & 0 & 1 & 2 & 3 & 4 & 5 & 6 & 7 & $\bar{7}$ & $\bar{6}$ & $\bar{5}$ & $\bar{4}$ & $\bar{3}$ & $\bar{2}$ & $\bar{1}$ \\
\hline	$0$ &  0 &  1 &  1 &  1 & 1 & 1 & 1 & 1 &  1 &  1 & 1 & 1 & 1 & 1 & 1\\
\hline	1 &  1 &  2 &  1 &  1 & 1 & 0 & 1 & 0 &  0 &  0 & 0 & 0 & 0 & 0 & 0\\
\hline	2 &  1 &  2 &  2 &  1 & 1 & 1 & 1 & 1 &  0 &  1 & 0 & 0 & 0 & 0 & 0\\
\hline	3 &  1 &  2 &  2 &  1 & 1 & 1 & 1 & 1 &  0 &  1 & 0 & 0 & 0 & 0 & 0\\
\hline	4 &  1 &  2 &  2 &  1 & 1 & 1 & 1 & 1 &  0 &  1 & 0 & 0 & 0 & 0 & 0\\
\hline	5 &  1 &  2 &  2 &  1 & 1 & 1 & 1 & 1 &  0 &  1 & 0 & 0 & 0 & 0 & 0\\
\hline	6 &  1 &  2 &  2 &  2 & 2 & 1 & 2 & 1 &  1 &  1 & 1 & 1 & 1 & 1 & 0\\
\hline	7 &  1 &  2 &  2 &  2 & 2 & 1 & 2 & 1 &  1 &  1 & 1 & 1 & 1 & 1 & 0\\
\hline	$\bar{7}$ & 1 &  2 &  2 &  1 & 1 & 1 & 1 & 1 &  0 &  1 & 0 & 0 & 0 & 0 & 0\\
\hline	$\bar{6}$ &  1 &  2 &  2 &  2 & 2 & 2 & 2 & 2 &  1 &  2 & 1 & 1 & 1 & 1 & 1\\
\hline	$\bar{5}$ &  1 &  2 &  2 &  2 & 2 & 1 & 2 & 1 &  1 &  1 & 1 & 1 & 1 & 1 & 0\\
\hline	$\bar{4}$ &  1 &  2 &  2 &  2 & 2 & 2 & 2 & 2 &  1 &  2 & 1 & 1 & 1 & 1 & 1\\
\hline	$\bar{3}$ &  1 &  2 &  2 &  2 & 2 & 2 & 2 & 2 &  1 &  2 & 1 & 1 & 1 & 1 & 1\\
\hline	$\bar{2}$ &  1 &  2 &  2 &  2 & 2 & 2 & 2 & 2 &  2 &  2 & 2 & 2 & 2 & 2 & 1\\
\hline	$\bar{1}$ &  1 &  2 &  2 &  2 & 2 & 2 & 2 & 2 &  2 &  2 & 2 & 2 & 2 & 2 & 2\\
\hline
\end{tabular}
\end{equation}

\subsection{Reduced Young walls}

We usually use the sequence $(\cdots,y_{i+1},y_i,\cdots,y_1,y_0)$ to represent a Young wall, where $y_i$ $(i\geq 0)$ is a Young column. Let $|y_i|$ denote  the number of blocks in $y_i$ that have been added to
the ground-state column. Let ${|y_i|}_0$ denote the number of $0$-blocks in $y_i$ that have been added to
the ground-state column. 

Let $C_1^{\mathrm{aff}}$ (resp. ${C'_1}^{\mathrm{aff}}$) be the affinization of $C_1$ (resp. $C'_1$). Let $\mathrm{adj}(Y)$ denote the set of all adjacent columns in a Young wall $Y$.

\begin{defn}\label{def:reduced}
	\begin{enumerate}
		\item We say that a Young wall $Y$ is \textit{reduced} if $$H^{\mathrm{aff}}(y_{i+1}({|y_{i+1}|}_0)\otimes y_i({|y_i|}_0
		))=0$$ for any adjacent columns $(y_{i+1},y_i)\in\mathrm{adj}(Y)$ and $y_i({|y_i|}_0),\ y_{i+1}({|y_{i+1}|}_0)\in C_1^{\mathrm{aff}}\ \text{or}\  {C'_1}^{\mathrm{aff}}$.
		\item The adjacent columns in a reduced Young wall are called \textit{reduced adjacent columns}. 
	\end{enumerate}
\end{defn}

\begin{Prop}\label{adjacent column}
\begin{enumerate}
	\item[{\rm (1)}] Let $(y_{i+1},y_i)$ be reduced adjacent columns. Then we have	
\begin{equation*}
{|y_i|}_0
=\begin{cases}
		{|y_{i+1}|}_0,& \text{if}\ (y_{i+1},y_i)\in \mathbf A\ (resp.\ \mathbf X),\\
		2+{|y_{i+1}|}_0,& \text{if}\ (y_{i+1},y_i)\in \mathbf B\ (resp.\ \mathbf Y),\\
		1+{|y_{i+1}|}_0,& \text{otherwise}.
	\end{cases}
\end{equation*}	
	\item[{\rm (2)}] If the adjacent Young columns $(y_{i+1},y_i)$ are reduced adjacent columns, then the number $|y_{i}|-|y_{i+1}|$ is a fixed non-negative integer.
\end{enumerate}	
\end{Prop}

\begin{proof}
By the values of energy function in \eqref{value of energy function D43}, \eqref{value of energy function G21} and the definition of reduced Young wall in Definition \ref{def:reduced}, the conclusion (1) holds naturally. 

For (2), Let $z_i$ (resp. $z_{i+1}$) represent the same Young column as $y_i$ (resp. $y_{i+1}$). We assume that $(z_{i+1},z_i)$ are another reduced adjacent columns. We have
\begin{equation*}
	H^{\mathrm{aff}}(y_{i+1}({|y_{i+1}|}_0)\otimes y_{i}({|y_i|}_0))=H^{\mathrm{aff}}(z_{i+1}({|z_{i+1}|}_0)\otimes z_{i}({|z_i|}_0))=0.
\end{equation*}

Thus we have 
\begin{equation*}
	H(y_{i+1}\otimes y_{i})+{|y_{i+1}|}_0-{|y_i|}_0=H(z_{i+1}\otimes z_{i})+{|z_{i+1}|}_0-{|z_i|}_0=0.
\end{equation*}

Since $H(y_{i+1}\otimes y_{i})=H(z_{i+1}\otimes z_{i})$, we have
\begin{equation}\label{difference of zero blocks}
	{|y_i|}_0-{|y_{i+1}|}_0={|z_i|}_0-{|z_{i+1}|}_0.
\end{equation}

By the observation of column patterns in Figure \ref{column pattern}, we can see that there are always two $0$-blocks in a cycle. Combining the wall pattern in Figure \ref{wall pattern}, the description of all possible Young columns in Figure \ref{all possible Young columns D43}, Figure \ref{all possible Young columns G21} and the formula \eqref{difference of zero blocks}, we have $$|y_{i}|-|y_{i+1}|=|z_{i}|-|z_{i+1}|\geq 0.$$
Hence, our assertion holds.
\end{proof}

\begin{Rmk}
By Proposition \ref{adjacent column} and all possible Young columns listed in Figure \ref{all possible Young columns D43} and Figure \ref{all possible Young columns G21}, it is easy to see that there are 64 cases of reduced adjacent columns for $U_q(D_4^{(3)})$ and 225 cases of reduced adjacent columns for $U_q(G_2^{(1)})$. By the values in table \eqref{table D43} (resp. \eqref{table G21}), we can specifically write any reduced adjacent columns for $U_q(D_4^{(3)})$ (resp.  $U_q(G_2^{(1)})$).
\end{Rmk}

\begin{Ex}
For the case of $U_q(D_4^{(3)})$, in the picture below, the Young walls  (a) and (b) are reduced, but the Young walls (c) and (d) are not reduced.

		\begin{center}
			\begin{tikzpicture}[scale=0.7] 	
				\begin{scope}[shift={(0,0)}]
					\draw (0,0)--(0,-0.5);
				\draw (1,0)--(1,-0.5);
				\node at (0.5,-0.3) {$\vdots$};	
					\draw (0,0)--(1,0)--(1,1)--(0,1)--(0,0);
					\draw (0,0)--(0.5,1);
					\node at (0.2,0.75) {$0$};
					
					\draw (1,1)--(1,2)--(0,2)--(0,1);
					\draw (0,1)--(1,2);
					\node at (0.75,1.3) {$1$};
					\node at (0.25,1.7) {$1$};
					\draw (1,2)--(1,3)--(0,3)--(0,2);
					\draw (0.5,2)--(1,3);
					\node at (0.8,2.25) {$0$};
					\node at (0.35,2.6) {$2$};
					\draw (1,3)--(1,4);
					\draw (0,4)--(0,3);
					\draw (0,3)--(1,4);
					\node at (0.75,3.3) {$1$};
					\node at (0.25,3.7) {$1$};
					
					\node at (0.65,0.4) {$2$};
				\end{scope}
				\begin{scope}[shift={(0,4)}]
					\draw (0,0)--(1,0)--(1,1)--(0,1)--(0,0);
					\draw (0,0)--(0.5,1);
					\node at (0.2,0.75) {$0$};
					
					\draw (1,1)--(1,2)--(0,2)--(0,1);
					\draw (0,1)--(1,2);
					\node at (0.75,1.3) {$1$};
					\node at (0.25,1.7) {$1$};
					\draw (1,2)--(1,3)--(0,3)--(0,2);
					\draw (0.5,2)--(1,3);
					\node at (0.8,2.25) {$0$};
					\node at (0.35,2.6) {$2$};
					\draw (1,3)--(1,4);
					\draw (0,4)--(0,3);
					\draw (0,3)--(1,4);
					\node at (0.75,3.3) {$1$};
					\node at (0.25,3.7) {$1$};
					
					\node at (0.65,0.4) {$2$};
				\end{scope}
				\begin{scope}[shift={(0,8)}]
					\draw (0,0)--(1,0)--(1,1)--(0,1)--(0,0);
					\draw (0,0)--(0.5,1);
					\node at (0.2,0.75) {$0$};
					
					\draw (1,1)--(1,2)--(0,2)--(0,1);
					\draw (0,1)--(1,2);
					\node at (0.75,1.3) {$1$};
					\node at (0.25,1.7) {$1$};
					\draw (1,2)--(1,3)--(0,3)--(0,2);
					\draw (0.5,2)--(1,3);
					\node at (0.8,2.25) {$0$};
					\node at (0.35,2.6) {$2$};
					\draw (1,3)--(1,4)--(0,4)--(0,3);
					\draw (0,3)--(1,4);
					\node at (0.75,3.3) {$1$};
					\node at (0.25,3.7) {$1$};
					
					\node at (0.65,0.4) {$2$};
				\end{scope}
				\begin{scope}[shift={(-1,0)}]
						\draw (0,0)--(0,-0.5);
					\draw (1,0)--(1,-0.5);
					\node at (0.5,-0.3) {$\vdots$};
					\draw (0,0)--(1,0)--(1,1)--(0,1);
					\draw (0,0)--(0.5,1);
					\node at (0.2,0.75) {$0$};
					
					\draw (1,1)--(1,2)--(0,2);
					\draw (0,1)--(1,2);
					\node at (0.75,1.3) {$1$};
					\node at (0.25,1.7) {$1$};
					\draw (1,2)--(1,3)--(0,3);
					\draw (0.5,2)--(1,3);
					\node at (0.8,2.25) {$0$};
					\node at (0.35,2.6) {$2$};
					\draw (1,3)--(1,4);
					
					\draw (0,3)--(1,4);
					\node at (0.75,3.3) {$1$};
					\node at (0.25,3.7) {$1$};
					
					\node at (0.65,0.4) {$2$};
				\end{scope}
				\begin{scope}[shift={(-1,4)}]
					\draw (0,0)--(1,0)--(1,1)--(0,1)--(0,0);
					\draw (0,0)--(0.5,1);
					\node at (0.2,0.75) {$0$};
					
					\draw (1,1)--(1,2)--(0,2)--(0,1);
					\draw (0,1)--(1,2);
					\node at (0.75,1.3) {$1$};
					\node at (0.25,1.7) {$1$};
					\draw (1,2)--(1,3)--(0,3)--(0,2);
					\draw (0.5,2)--(1,3);
					\node at (0.8,2.25) {$0$};
					\node at (0.35,2.6) {$2$};
					\draw (1,3)--(1,4);
					\draw (0,4)--(0,3);
					\draw (0,4)--(1,4);
					\draw (0,3)--(1,4);
					\node at (0.75,3.3) {$1$};
					\node at (0.25,3.7) {$1$};
					
					\node at (0.65,0.4) {$2$};
				\end{scope}
				\begin{scope}[shift={(-2,0)}]
						\draw (0,0)--(0,-0.5);
					\draw (1,0)--(1,-0.5);
					\node at (0.5,-0.3) {$\vdots$};
					\draw (0,0)--(1,0)--(1,1)--(0,1);
					\draw (0,0)--(0.5,1);
					\node at (0.2,0.75) {$0$};
					
					\draw (1,1)--(1,2)--(0,2);
					\draw (0,1)--(1,2);
					\node at (0.75,1.3) {$1$};
					\node at (0.25,1.7) {$1$};
					\draw (1,2)--(1,3)--(0,3);
					\draw (0.5,2)--(1,3);
					\node at (0.8,2.25) {$0$};
					\node at (0.35,2.6) {$2$};
					\draw (1,3)--(1,4);
					\draw (0,4)--(1,4);
					\draw (0,3)--(1,4);
					\draw (0,4)--(0,5)--(1,5);
					\draw (0,4)--(0.5,5);
					\node at (0.2,4.75) {$0$};
					\node at (0.65,4.4) {$2$};
					\draw (0,5)--(0,6)--(1,6);
					\draw (0,5)--(1,6);
					\node at (0.25,5.7) {$1$};	
					\node at (0.75,3.3) {$1$};
					\node at (0.25,3.7) {$1$};
					
					\node at (0.65,0.4) {$2$};
				\end{scope}	
				\begin{scope}[shift={(-3,0)}]
						\draw (0,0)--(0,-0.5);
					\draw (1,0)--(1,-0.5);
					\node at (0.5,-0.3) {$\vdots$};
					\draw (0,0)--(1,0)--(1,1)--(0,1);
					\draw (0,0)--(0.5,1);
					\node at (0.2,0.75) {$0$};
					
					\draw (1,1)--(1,2)--(0,2);
					\draw (0,1)--(1,2);
					\node at (0.75,1.3) {$1$};
					\node at (0.25,1.7) {$1$};
					\draw (1,2)--(1,3)--(0,3);
					\draw (0.5,2)--(1,3);
					\node at (0.8,2.25) {$0$};
					\node at (0.35,2.6) {$2$};
					\draw (1,3)--(1,4);
					
					\draw (0,3)--(1,4);
					\node at (0.75,3.3) {$1$};
					\draw (0,4)--(1,4);
					
					\node at (0.65,0.4) {$2$};
				\end{scope}	
				\begin{scope}[shift={(-4,0)}]
					\draw (0,0)--(-0.5,0);
				\draw (0,1)--(-0.5,1);
				\draw (0,2)--(-0.5,2);
				\draw (0,3)--(-0.5,3);
				
				\node at (-0.4,0.5) {$\cdots$};
				\node at (-0.4,1.5) {$\cdots$};
				\node at (-0.4,2.5) {$\cdots$};	
						\draw (0,0)--(0,-0.5);
					\draw (1,0)--(1,-0.5);
					\node at (0.5,-0.3) {$\vdots$};
					\draw (0,0)--(1,0)--(1,1)--(0,1)--(0,0);
					\draw (0,0)--(0.5,1);
					\node at (0.2,0.75) {$0$};
					
					\draw (1,1)--(1,2)--(0,2)--(0,1);
					\draw (0,1)--(1,2);
					\node at (0.75,1.3) {$1$};
					
					\draw (1,2)--(1,3)--(0,3)--(0,2);
					\draw (0.5,2)--(1,3);
					\node at (0.8,2.25) {$0$};
					
					\draw (1,3)--(1,4);

					\node at (0.65,0.4) {$2$};
				\end{scope}	
				\node at (-1.5,-1.5) {(a)};
				\begin{scope}[shift={(10,0)}]
					\begin{scope}[shift={(0,0)}]
							\draw (0,0)--(0,-0.5);
						\draw (1,0)--(1,-0.5);
						\node at (0.5,-0.3) {$\vdots$};
						\draw (0,0)--(1,0)--(1,1)--(0,1)--(0,0);
						\draw (0,0)--(0.5,1);
						\node at (0.2,0.75) {$0$};
						
						\draw (1,1)--(1,2)--(0,2)--(0,1);
						\draw (0,1)--(1,2);
						\node at (0.75,1.3) {$1$};
						\node at (0.25,1.7) {$1$};
						\draw (1,2)--(1,3)--(0,3)--(0,2);
						\draw (0.5,2)--(1,3);
						\node at (0.8,2.25) {$0$};
						\node at (0.35,2.6) {$2$};
						\draw (1,3)--(1,4);
						\draw (0,4)--(0,3);
						\draw (0,3)--(1,4);
						\node at (0.75,3.3) {$1$};
						\node at (0.25,3.7) {$1$};
						
						\node at (0.65,0.4) {$2$};
					\end{scope}
					\begin{scope}[shift={(0,4)}]
						\draw (0,0)--(1,0)--(1,1)--(0,1)--(0,0);
						\draw (0,0)--(0.5,1);
						\node at (0.2,0.75) {$0$};
						
						\draw (1,1)--(1,2)--(0,2)--(0,1);
						\draw (0,1)--(1,2);
						\node at (0.75,1.3) {$1$};
						\node at (0.25,1.7) {$1$};
						\draw (1,2)--(1,3)--(0,3)--(0,2);
						\draw (0.5,2)--(1,3);
						\node at (0.8,2.25) {$0$};
						\node at (0.35,2.6) {$2$};
						\draw (1,3)--(1,4);
						\draw (0,4)--(0,3);
						\draw (0,3)--(1,4);
						\node at (0.75,3.3) {$1$};
						\node at (0.25,3.7) {$1$};
						
						\node at (0.65,0.4) {$2$};
					\end{scope}
					\begin{scope}[shift={(0,8)}]
						\draw (0,0)--(1,0)--(1,1)--(0,1)--(0,0);
						\draw (0,0)--(0.5,1);
						\node at (0.2,0.75) {$0$};
						
						\draw (1,1)--(1,2)--(0,2)--(0,1);
						\draw (0,1)--(1,2);
						\node at (0.75,1.3) {$1$};
						
						\draw (1,2)--(1,3)--(0,3)--(0,2);
						\draw (0.5,2)--(1,3);
						\node at (0.8,2.25) {$0$};

						\node at (0.65,0.4) {$2$};
					\end{scope}
					\begin{scope}[shift={(-1,0)}]
							\draw (0,0)--(0,-0.5);
						\draw (1,0)--(1,-0.5);
						\node at (0.5,-0.3) {$\vdots$};
						\draw (0,0)--(1,0)--(1,1)--(0,1);
						\draw (0,0)--(0.5,1);
						\node at (0.2,0.75) {$0$};
						
						\draw (1,1)--(1,2)--(0,2);
						\draw (0,1)--(1,2);
						\node at (0.75,1.3) {$1$};
						\node at (0.25,1.7) {$1$};
						\draw (1,2)--(1,3)--(0,3);
						\draw (0.5,2)--(1,3);
						\node at (0.8,2.25) {$0$};
						\node at (0.35,2.6) {$2$};
						\draw (1,3)--(1,4);
						
						\draw (0,3)--(1,4);
						\node at (0.75,3.3) {$1$};
						\node at (0.25,3.7) {$1$};
						
						\node at (0.65,0.4) {$2$};
					\end{scope}
					\begin{scope}[shift={(-1,4)}]
						\draw (0,0)--(1,0)--(1,1)--(0,1)--(0,0);
						\draw (0,0)--(0.5,1);
						\node at (0.2,0.75) {$0$};
						
						\draw (1,1)--(1,2)--(0,2)--(0,1);
						\draw (0,1)--(1,2);
						\node at (0.75,1.3) {$1$};
						
						\draw (1,2)--(1,3)--(0,3)--(0,2);
						\draw (0.5,2)--(1,3);
						\node at (0.8,2.25) {$0$};
						
						\draw (1,3)--(1,4);

						\node at (0.65,0.4) {$2$};
					\end{scope}
					\begin{scope}[shift={(-2,0)}]
							\draw (0,0)--(0,-0.5);
						\draw (1,0)--(1,-0.5);
						\node at (0.5,-0.3) {$\vdots$};
						\draw (0,0)--(1,0)--(1,1)--(0,1);
						\draw (0,0)--(0.5,1);
						\node at (0.2,0.75) {$0$};
						
						\draw (1,1)--(1,2)--(0,2);
						\draw (0,1)--(1,2);
						\node at (0.75,1.3) {$1$};
						\node at (0.25,1.7) {$1$};
						\draw (1,2)--(1,3)--(0,3);
						\draw (0.5,2)--(1,3);
						\node at (0.8,2.25) {$0$};
						\node at (0.35,2.6) {$2$};
						\draw (1,3)--(1,4);
						\draw (0,4)--(1,4);
						\draw (0,3)--(1,4);
						\draw (0,4)--(0,5)--(1,5);
						\draw (0,4)--(0.5,5);
						\node at (0.8,6.25) {$0$};
						\node at (0.65,4.4) {$2$};
						\draw (0,5)--(0,6)--(1,6);
						\draw (0,5)--(1,6);
						\draw (0,6)--(0,7)--(1,7);
						\draw (0.5,6)--(1,7);
						
						\node at (0.75,5.3) {$1$};
						\node at (0.75,3.3) {$1$};
						\node at (0.25,3.7) {$1$};
						
						\node at (0.65,0.4) {$2$};
					\end{scope}	
					\begin{scope}[shift={(-3,0)}]
							\draw (0,0)--(0,-0.5);
						\draw (1,0)--(1,-0.5);
						\node at (0.5,-0.3) {$\vdots$};
						\draw (0,0)--(1,0)--(1,1)--(0,1);
						\draw (0,0)--(0.5,1);
						\node at (0.2,0.75) {$0$};
						
						\draw (1,1)--(1,2)--(0,2);
						\draw (0,1)--(1,2);
						\node at (0.75,1.3) {$1$};
						\node at (0.25,1.7) {$1$};
						\draw (1,2)--(1,3)--(0,3);
						\draw (0.5,2)--(1,3);
						\node at (0.8,2.25) {$0$};
						\node at (0.35,2.6) {$2$};
						\draw (1,3)--(1,4);
						
						\draw (0,3)--(1,4);
						\node at (0.75,3.3) {$1$};
						\draw (0,4)--(1,4);
						
						\node at (0.65,0.4) {$2$};
					\end{scope}	
					\begin{scope}[shift={(-4,0)}]
						\draw (0,0)--(0,-0.5);
					\draw (1,0)--(1,-0.5);
					\node at (0.5,-0.3) {$\vdots$};	
						
						\draw (0,0)--(-0.5,0);
					\draw (0,1)--(-0.5,1);
					\draw (0,2)--(-0.5,2);
					\draw (0,3)--(-0.5,3);
					
					\node at (-0.4,0.5) {$\cdots$};
					\node at (-0.4,1.5) {$\cdots$};
					\node at (-0.4,2.5) {$\cdots$};		
						\draw (0,0)--(1,0)--(1,1)--(0,1)--(0,0);
						\draw (0,0)--(0.5,1);
						\node at (0.2,0.75) {$0$};
						
						\draw (1,1)--(1,2)--(0,2)--(0,1);
						\draw (0,1)--(1,2);
						\node at (0.75,1.3) {$1$};
						
						\draw (1,2)--(1,3)--(0,3)--(0,2);
						\draw (0.5,2)--(1,3);
						\node at (0.8,2.25) {$0$};
						
						\draw (1,3)--(1,4);

						\node at (0.65,0.4) {$2$};
					\end{scope}	
					\node at (-1.5,-1.5) {(b)};	
				\end{scope}
\end{tikzpicture}			
\end{center}

\begin{center}
\begin{tikzpicture}	[scale=0.7] 				
				\begin{scope}[shift={(0,0)}]
					\begin{scope}[shift={(0,0)}]
							\draw (0,0)--(0,-0.5);
						\draw (1,0)--(1,-0.5);
						\node at (0.5,-0.3) {$\vdots$};
						
						\draw (0,0)--(1,0)--(1,1)--(0,1)--(0,0);
						\draw (0,0)--(0.5,1);
						\node at (0.2,0.75) {$0$};
						
						\draw (1,1)--(1,2)--(0,2)--(0,1);
						\draw (0,1)--(1,2);
						\node at (0.75,1.3) {$1$};
						\node at (0.25,1.7) {$1$};
						\draw (1,2)--(1,3)--(0,3)--(0,2);
						\draw (0.5,2)--(1,3);
						\node at (0.8,2.25) {$0$};
						\node at (0.35,2.6) {$2$};
						\draw (1,3)--(1,4);
						\draw (0,4)--(0,3);
						\draw (0,3)--(1,4);
						\node at (0.75,3.3) {$1$};
						\node at (0.25,3.7) {$1$};
						
						\node at (0.65,0.4) {$2$};
					\end{scope}
					\begin{scope}[shift={(0,4)}]
						\draw (0,0)--(1,0)--(1,1)--(0,1)--(0,0);
						\draw (0,0)--(0.5,1);
						\node at (0.2,0.75) {$0$};
						
						\draw (1,1)--(1,2)--(0,2)--(0,1);
						\draw (0,1)--(1,2);
						\node at (0.75,1.3) {$1$};
						\node at (0.25,1.7) {$1$};
						\draw (1,2)--(1,3)--(0,3)--(0,2);
						\draw (0.5,2)--(1,3);
						\node at (0.8,2.25) {$0$};
						\node at (0.35,2.6) {$2$};
						\draw (1,3)--(1,4);
						\draw (0,4)--(0,3);
						\draw (0,3)--(1,4);
						\node at (0.75,3.3) {$1$};
						\node at (0.25,3.7) {$1$};
						
						\node at (0.65,0.4) {$2$};
					\end{scope}
					\begin{scope}[shift={(0,8)}]
						\draw (0,0)--(1,0)--(1,1)--(0,1)--(0,0);
						\draw (0,0)--(0.5,1);
						\node at (0.2,0.75) {$0$};
						\node at (0.8,2.25) {$0$};
						\draw (1,1)--(1,2)--(0,2)--(0,1);
						\draw (0,1)--(1,2);
						\node at (0.25,1.7) {$1$};
						\node at (0.75,1.3) {$1$};
						\draw (1,2)--(1,3)--(0,3)--(0,2);
						\draw (0.5,2)--(1,3);
						\node at (0.65,0.4) {$2$};
					\end{scope}
					\begin{scope}[shift={(-1,0)}]
							\draw (0,0)--(0,-0.5);
						\draw (1,0)--(1,-0.5);
						\node at (0.5,-0.3) {$\vdots$};
						
						\draw (0,0)--(1,0)--(1,1)--(0,1);
						\draw (0,0)--(0.5,1);
						\node at (0.2,0.75) {$0$};
						
						\draw (1,1)--(1,2)--(0,2);
						\draw (0,1)--(1,2);
						\node at (0.75,1.3) {$1$};
						\node at (0.25,1.7) {$1$};
						\draw (1,2)--(1,3)--(0,3);
						\draw (0.5,2)--(1,3);
						\node at (0.8,2.25) {$0$};
						\node at (0.35,2.6) {$2$};
						\draw (1,3)--(1,4);
						
						\draw (0,3)--(1,4);
						\node at (0.75,3.3) {$1$};
						\node at (0.25,3.7) {$1$};
						
						\node at (0.65,0.4) {$2$};
					\end{scope}
					\begin{scope}[shift={(-1,4)}]
						\draw (0,0)--(1,0)--(1,1)--(0,1)--(0,0);
						\draw (0,0)--(0.5,1);
						\node at (0.2,0.75) {$0$};
						
						\draw (1,1)--(1,2)--(0,2)--(0,1);
						\draw (0,1)--(1,2);
						\node at (0.75,1.3) {$1$};
						
						\draw (1,3)--(0,3)--(0,2);
						\draw (0.5,2)--(1,3);

						\draw (1,3)--(1,4);
						\node at (0.25,1.7) {$1$};
						\node at (0.35,2.6) {$2$};

						\node at (0.65,0.4) {$2$};
						\draw (1,4)--(0,4)--(0,3);
						\draw (0,3)--(1,4);
						\node at (0.25,3.7) {$1$};
						\draw (1,5)--(0,5)--(0,4);
						\draw (0,4)--(0.5,5);
						\node at (0.2,4.75) {$0$};	
					\end{scope}
					\begin{scope}[shift={(-2,0)}]
							\draw (0,0)--(0,-0.5);
						\draw (1,0)--(1,-0.5);
						\node at (0.5,-0.3) {$\vdots$};
						
						\draw (0,0)--(1,0)--(1,1)--(0,1);
						\draw (0,0)--(0.5,1);
						\node at (0.2,0.75) {$0$};
						
						\draw (1,1)--(1,2)--(0,2);
						\draw (0,1)--(1,2);
						\node at (0.75,1.3) {$1$};
						\node at (0.25,1.7) {$1$};
						\draw (1,2)--(1,3)--(0,3);
						\draw (0.5,2)--(1,3);
						\node at (0.8,2.25) {$0$};
						\node at (0.35,2.6) {$2$};
						\draw (1,3)--(1,4);
						\draw (0,4)--(1,4);
						\draw (0,3)--(1,4);
						\draw (0,4)--(0,5)--(1,5);
						\draw (0,4)--(0.5,5);
						
						\node at (0.65,4.4) {$2$};
						\draw (0,5)--(0,6)--(1,6);
						\draw (0,5)--(1,6);
						\draw (0,6)--(0,7)--(1,7);
						\draw (0.5,6)--(1,7);
						
						\node at (0.75,5.3) {$1$};
						\node at (0.75,3.3) {$1$};
						\node at (0.25,3.7) {$1$};
						
						\node at (0.65,0.4) {$2$};
						\draw (0,7)--(0,8)--(1,8);
						\draw (0,7)--(1,8);
						\draw (0,8)--(0,9)--(1,9);	
						\draw (0,8)--(0.5,9);
						\node at (0.2,4.75) {$0$};
						\node at (0.25,5.7) {$1$};
						
						\node at (0.35,6.6) {$2$};
						\node at (0.25,7.7) {$1$};
						
						\node at (0.2,8.75) {$0$};	
					\end{scope}	
					\begin{scope}[shift={(-3,0)}]
							\draw (0,0)--(0,-0.5);
						\draw (1,0)--(1,-0.5);
						\node at (0.5,-0.3) {$\vdots$};
						
						\draw (0,0)--(1,0)--(1,1)--(0,1);
						\draw (0,0)--(0.5,1);
						\node at (0.2,0.75) {$0$};
						
						\draw (1,1)--(1,2)--(0,2);
						\draw (0,1)--(1,2);
						\node at (0.75,1.3) {$1$};
						\node at (0.25,1.7) {$1$};
						\draw (1,2)--(1,3)--(0,3);
						\draw (0.5,2)--(1,3);
						\node at (0.8,2.25) {$0$};
						\node at (0.35,2.6) {$2$};
						\draw (1,3)--(1,4);
						
						\draw (0,3)--(1,4);
						\node at (0.75,3.3) {$1$};
						\node at (0.25,3.7) {$1$};
						\draw (0,4)--(1,4);
						\node at (0.65,0.4) {$2$};
						\draw (0,4)--(0,5)--(1,5);
						\draw (0,4)--(0.5,5);
						\draw (0,5)--(0,6)--(1,6);
						\draw (0,5)--(1,6);
						\draw (0,6)--(0,7)--(1,7);
						\draw (0.5,6)--(1,7);
						\node at (0.2,4.75) {$0$};
						\node at (0.65,4.4) {$2$};
						\node at (0.75,5.3) {$1$};
						\node at (0.25,5.7) {$1$};
						\node at (0.35,6.6) {$2$};
					\end{scope}	
					\begin{scope}[shift={(-4,0)}]
							\draw (0,0)--(0,-0.5);
						\draw (1,0)--(1,-0.5);
						\node at (0.5,-0.3) {$\vdots$};
						
						\draw (0,0)--(1,0)--(1,1)--(0,1)--(0,0);
						\draw (0,0)--(0.5,1);
						\node at (0.2,0.75) {$0$};
						
						\draw (1,1)--(1,2)--(0,2)--(0,1);
						\draw (0,1)--(1,2);
						\node at (0.75,1.3) {$1$};
						\node at (0.25,1.7) {$1$};
						\draw (1,2)--(1,3)--(0,3)--(0,2);
						\draw (0.5,2)--(1,3);
						\node at (0.35,2.6) {$2$};
						\node at (0.8,2.25) {$0$};
						
						\draw (0,3)--(0,4)--(1,4)--(1,3);
						\draw (0,3)--(1,4);	
						\node at (0.75,3.3) {$1$};
						\node at (0.25,3.7) {$1$};
						
						\node at (0.65,0.4) {$2$};
					\end{scope}	
					\begin{scope}[shift={(-5,0)}]
							\draw (0,0)--(0,-0.5);
						\draw (1,0)--(1,-0.5);
						\node at (0.5,-0.3) {$\vdots$};
						
							\draw (0,0)--(-0.5,0);
						\draw (0,1)--(-0.5,1);
						\draw (0,2)--(-0.5,2);
						\draw (0,3)--(-0.5,3);
						
						\node at (-0.4,0.5) {$\cdots$};
						\node at (-0.4,1.5) {$\cdots$};
						\node at (-0.4,2.5) {$\cdots$};		
						
						\draw (0,0)--(1,0)--(1,1)--(0,1)--(0,0);
						\draw (0,0)--(0.5,1);
						\node at (0.2,0.75) {$0$};
						
						\draw (1,1)--(1,2)--(0,2)--(0,1);
						\draw (0,1)--(1,2);
						\node at (0.75,1.3) {$1$};
						
						\draw (1,2)--(1,3)--(0,3)--(0,2);
						\draw (0.5,2)--(1,3);
						\node at (0.8,2.25) {$0$};
						\node at (0.65,0.4) {$2$};
					\end{scope}	
					\node at (-2,-1.5) {(c)};	
				\end{scope}
				\begin{scope}[shift={(9,0)}]
					\begin{scope}[shift={(0,0)}]
							\draw (0,0)--(0,-0.5);
						\draw (1,0)--(1,-0.5);
						\node at (0.5,-0.3) {$\vdots$};
						
						\draw (0,0)--(1,0)--(1,1)--(0,1)--(0,0);
						\draw (0,0)--(0.5,1);
						\node at (0.2,0.75) {$0$};
						
						\draw (1,1)--(1,2)--(0,2)--(0,1);
						\draw (0,1)--(1,2);
						\node at (0.75,1.3) {$1$};
						\node at (0.25,1.7) {$1$};
						\draw (1,2)--(1,3)--(0,3)--(0,2);
						\draw (0.5,2)--(1,3);
						\node at (0.8,2.25) {$0$};
						\node at (0.35,2.6) {$2$};
						\draw (1,3)--(1,4);
						\draw (0,4)--(0,3);
						\draw (0,3)--(1,4);
						\node at (0.75,3.3) {$1$};
						\node at (0.25,3.7) {$1$};
						
						\node at (0.65,0.4) {$2$};
					\end{scope}
					\begin{scope}[shift={(0,4)}]
						\draw (0,0)--(1,0)--(1,1)--(0,1)--(0,0);
						\draw (0,0)--(0.5,1);
						\node at (0.2,0.75) {$0$};
						
						\draw (1,1)--(1,2)--(0,2)--(0,1);
						\draw (0,1)--(1,2);
						\node at (0.75,1.3) {$1$};
						\node at (0.25,1.7) {$1$};
						\draw (1,2)--(1,3)--(0,3)--(0,2);
						\draw (0.5,2)--(1,3);
						\node at (0.8,2.25) {$0$};
						\node at (0.35,2.6) {$2$};
						\draw (1,3)--(1,4);
						\draw (0,4)--(0,3);
						\draw (0,3)--(1,4);
						\node at (0.75,3.3) {$1$};
						\node at (0.25,3.7) {$1$};
						
						\node at (0.65,0.4) {$2$};
					\end{scope}
					\begin{scope}[shift={(0,8)}]
						\draw (0,0)--(1,0)--(1,1)--(0,1)--(0,0);
						\draw (0,0)--(0.5,1);
						\node at (0.2,0.75) {$0$};
						
						\draw (1,1)--(1,2)--(0,2)--(0,1);
						\draw (0,1)--(1,2);
						\node at (0.75,1.3) {$1$};
						\node at (0.25,1.7) {$1$};
						\draw (1,2)--(1,3)--(0,3)--(0,2);
						\draw (0.5,2)--(1,3);
						\node at (0.8,2.25) {$0$};
						\node at (0.35,2.6) {$2$};
						\draw (1,3)--(1,4)--(0,4)--(0,3);
						\draw (0,3)--(1,4);
						\node at (0.75,3.3) {$1$};

						\node at (0.65,0.4) {$2$};
					\end{scope}
					\begin{scope}[shift={(-1,0)}]
							\draw (0,0)--(0,-0.5);
						\draw (1,0)--(1,-0.5);
						\node at (0.5,-0.3) {$\vdots$};
						\draw (0,0)--(1,0)--(1,1)--(0,1);
						\draw (0,0)--(0.5,1);
						\node at (0.2,0.75) {$0$};
						
						\draw (1,1)--(1,2)--(0,2);
						\draw (0,1)--(1,2);
						\node at (0.75,1.3) {$1$};
						\node at (0.25,1.7) {$1$};
						\draw (1,2)--(1,3)--(0,3);
						\draw (0.5,2)--(1,3);
						\node at (0.8,2.25) {$0$};
						\node at (0.35,2.6) {$2$};
						\draw (1,3)--(1,4);
						
						\draw (0,3)--(1,4);
						\node at (0.75,3.3) {$1$};
						\node at (0.25,3.7) {$1$};
						
						\node at (0.65,0.4) {$2$};
					\end{scope}
					\begin{scope}[shift={(-1,4)}]
						\draw (0,0)--(1,0)--(1,1)--(0,1)--(0,0);
						\draw (0,0)--(0.5,1);
						\node at (0.2,0.75) {$0$};
						
						\draw (1,1)--(1,2)--(0,2)--(0,1);
						\draw (0,1)--(1,2);
						\node at (0.75,1.3) {$1$};
						\node at (0.25,1.7) {$1$};
						\draw (1,2)--(1,3)--(0,3)--(0,2);
						\draw (0.5,2)--(1,3);
						\node at (0.8,2.25) {$0$};
						\node at (0.35,2.6) {$2$};
						
						\draw (0,4)--(0,3);
						\draw (0,4)--(1,4);
						\draw (0,3)--(1,4);
						\node at (0.75,3.3) {$1$};
						\node at (0.25,3.7) {$1$};
						
						\node at (0.65,0.4) {$2$};
						\draw (0,4)--(0,5)--(1,5);
						\draw (0,4)--(0.5,5);
						\draw (0,5)--(0,6)--(1,6);
						\draw (0,5)--(1,6);
						\node at (0.2,4.75) {$0$};	
						\node at (0.65,4.4) {$2$};
						\node at (0.25,5.7) {$1$};		
					\end{scope}
					\begin{scope}[shift={(-2,0)}]
							\draw (0,0)--(0,-0.5);
						\draw (1,0)--(1,-0.5);
						\node at (0.5,-0.3) {$\vdots$};
						\draw (0,0)--(1,0)--(1,1)--(0,1);
						\draw (0,0)--(0.5,1);
						\node at (0.2,0.75) {$0$};
						
						\draw (1,1)--(1,2)--(0,2);
						\draw (0,1)--(1,2);
						\node at (0.75,1.3) {$1$};
						\node at (0.25,1.7) {$1$};
						\draw (1,2)--(1,3)--(0,3);
						\draw (0.5,2)--(1,3);
						\node at (0.8,2.25) {$0$};
						\node at (0.35,2.6) {$2$};
						\draw (1,3)--(1,4);
						\draw (0,4)--(1,4);
						\draw (0,3)--(1,4);
						\draw (0,4)--(0,5)--(1,5);
						\draw (0,4)--(0.5,5);
						\node at (0.2,4.75) {$0$};
						\node at (0.65,4.4) {$2$};
						\draw (0,5)--(0,6)--(1,6);
						\draw (0,5)--(1,6);
						\node at (0.25,5.7) {$1$};
						\node at (0.75,5.3) {$1$};	
						\node at (0.75,3.3) {$1$};
						\node at (0.25,3.7) {$1$};
						
						\node at (0.65,0.4) {$2$};
						\draw (0,6)--(0,7)--(1,7);
						\draw (0.5,6)--(1,7);
						\draw (0,7)--(0,8)--(1,8);
						\draw (0,7)--(1,8);
						\node at (0.8,6.25) {$0$};
						\node at (0.35,6.6) {$2$};	
						\node at (0.75,7.3) {$1$};
						
					\end{scope}	
					\begin{scope}[shift={(-3,0)}]
							\draw (0,0)--(0,-0.5);
						\draw (1,0)--(1,-0.5);
						\node at (0.5,-0.3) {$\vdots$};
						\draw (0,0)--(1,0)--(1,1)--(0,1);
						\draw (0,0)--(0.5,1);
						\node at (0.2,0.75) {$0$};
						
						\draw (1,1)--(1,2)--(0,2);
						\draw (0,1)--(1,2);
						\node at (0.75,1.3) {$1$};
						
						\draw (1,2)--(1,3)--(0,3);
						\draw (0.5,2)--(1,3);
						\node at (0.8,2.25) {$0$};
						
						\draw (1,3)--(1,4);
						
						\node at (0.65,0.4) {$2$};
					\end{scope}	
					\begin{scope}[shift={(-4,0)}]
						\draw (0,0)--(0,-0.5);
					\draw (1,0)--(1,-0.5);
					\node at (0.5,-0.3) {$\vdots$};	
						
							\draw (0,0)--(-0.5,0);
						\draw (0,1)--(-0.5,1);
						\draw (0,2)--(-0.5,2);
						\draw (0,3)--(-0.5,3);
						
						\node at (-0.4,0.5) {$\cdots$};
						\node at (-0.4,1.5) {$\cdots$};
						\node at (-0.4,2.5) {$\cdots$};		
						
						\draw (0,0)--(1,0)--(1,1)--(0,1)--(0,0);
						\draw (0,0)--(0.5,1);
						
						\draw (1,1)--(1,2)--(0,2)--(0,1);
						\draw (0,1)--(1,2);
						\node at (0.75,1.3) {$1$};
						
						\draw (1,2)--(1,3)--(0,3)--(0,2);
						\draw (0.5,2)--(1,3);
						\node at (0.8,2.25) {$0$};
						\node at (0.65,0.4) {$2$};
					\end{scope}	
					\node at (-1.5,-1.5) {(d)};	
				\end{scope}

			\end{tikzpicture}	
		\end{center}	

\end{Ex}

\begin{Ex}
For the case of $U_q(G_2^{(1)})$, in the picture below, the Young walls  (a) and (b) are reduced, but the Young walls (c) and (d) are not reduced.

	\begin{center}
			\begin{tikzpicture}[scale=0.7]
\begin{scope}	[shift={(0,0)}]		
				\draw (0,0)--(0,-0.5);
				\draw (1,0)--(1,-0.5);
				\node at (0.5,-0.3) {$\vdots$};
				
				\draw (0,0)--(1,0)--(1,1)--(0,1)--(0,0);
				\draw (0,0)--(0.5,1);\draw (0,0)--(1,1);
				
				\draw (1,1)--(1,2)--(0,2)--(0,1);
				\draw (0.5,1.5)--(1,1.5);\draw (0.5,1)--(0.5,2);
				
				\draw (1,2)--(1,3)--(0,3)--(0,2);
				\draw (0.5,2)--(1,3);	\draw (0,2)--(1,3);
				
				\draw (1,3)--(1,4)--(0,4)--(0,3);
				\draw (0.5,3)--(0.5,4);\draw (0,3.5)--(0.5,3.5);
				\node at (0.15,0.75) {$2$};	
				\node at (0.53,0.75) {$0$};
				\node at (0.7,0.3) {$2$};
				\node at (0.75,1.75) {$2$};
				\node at (0.75,1.25) {$1$};
				\node at (0.25,1.5) {$1$};
				\node at (0.85,2.25) {$2$};	
				\node at (0.48,2.25) {$0$};
				\node at (0.3,2.7) {$2$};
				\node at (0.25,3.75) {$2$};
				\node at (0.25,3.25) {$1$};
				\node at (0.75,3.5) {$1$};
				\begin{scope}[shift={(0,4)}]
				
				\draw (1,0)--(1,1)--(0,1)--(0,0);
				\draw (0,0)--(0.5,1);\draw (0,0)--(1,1);
				
				\draw (1,1)--(1,2)--(0,2)--(0,1);
				\draw (0.5,1.5)--(1,1.5);\draw (0.5,1)--(0.5,2);
				
				\draw (1,2)--(1,3)--(0,3)--(0,2);
				\draw (0.5,2)--(1,3);	\draw (0,2)--(1,3);
				
				\draw (1,3)--(1,4)--(0,4)--(0,3);
				\draw (0.5,3)--(0.5,4);\draw (0,3.5)--(0.5,3.5);
				\node at (0.15,0.75) {$2$};	
				\node at (0.53,0.75) {$0$};
				\node at (0.7,0.3) {$2$};
				\node at (0.75,1.75) {$2$};
				\node at (0.75,1.25) {$1$};
				\node at (0.25,1.5) {$1$};
				\node at (0.85,2.25) {$2$};	
				\node at (0.48,2.25) {$0$};
				\node at (0.3,2.7) {$2$};
				\node at (0.25,3.75) {$2$};
				\node at (0.25,3.25) {$1$};
				\node at (0.75,3.5) {$1$};
			\end{scope}
					\begin{scope}[shift={(0,8)}]
			
			\draw (1,0)--(1,1)--(0,1)--(0,0);
			\draw (0,0)--(0.5,1);\draw (0,0)--(1,1);
			
			\draw (1,1)--(1,2)--(0,2)--(0,1);
			\draw (0.5,1.5)--(1,1.5);\draw (0.5,1)--(0.5,2);
			
			\draw (1,2)--(1,3)--(0,3)--(0,2);
			\draw (0.5,2)--(1,3);	\draw (0,2)--(1,3);
			
			\draw (1,3)--(1,4)--(0,4)--(0,3);
			\draw (0.5,3)--(0.5,4);\draw (0,3.5)--(0.5,3.5);
				\draw (1,4)--(1,5)--(0,5)--(0,4);\draw (0,4)--(1,5);
			\node at (0.7,4.3) {$2$};
			\node at (0.15,0.75) {$2$};	
			\node at (0.53,0.75) {$0$};
			\node at (0.7,0.3) {$2$};
			\node at (0.75,1.75) {$2$};
			\node at (0.75,1.25) {$1$};
			\node at (0.25,1.5) {$1$};
			\node at (0.85,2.25) {$2$};	
			\node at (0.48,2.25) {$0$};
			\node at (0.3,2.7) {$2$};
			\node at (0.25,3.75) {$2$};
			\node at (0.25,3.25) {$1$};
			\node at (0.75,3.5) {$1$};
		\end{scope}
	\end{scope}	

\begin{scope}	[shift={(-1,0)}]		
	\draw (0,0)--(0,-0.5);
	\draw (1,0)--(1,-0.5);
	\node at (0.5,-0.3) {$\vdots$};
	
	\draw (0,0)--(1,0)--(1,1)--(0,1)--(0,0);
	\draw (0,0)--(0.5,1);\draw (0,0)--(1,1);
	
	\draw (1,1)--(1,2)--(0,2)--(0,1);
	\draw (0.5,1.5)--(1,1.5);\draw (0.5,1)--(0.5,2);
	
	\draw (1,2)--(1,3)--(0,3)--(0,2);
	\draw (0.5,2)--(1,3);	\draw (0,2)--(1,3);
	
	\draw (1,3)--(1,4)--(0,4)--(0,3);
	\draw (0.5,3)--(0.5,4);\draw (0,3.5)--(0.5,3.5);
	\node at (0.15,0.75) {$2$};	
	\node at (0.53,0.75) {$0$};
	\node at (0.7,0.3) {$2$};
	\node at (0.75,1.75) {$2$};
	\node at (0.75,1.25) {$1$};
	\node at (0.25,1.5) {$1$};
	\node at (0.85,2.25) {$2$};	
	\node at (0.48,2.25) {$0$};
	\node at (0.3,2.7) {$2$};
	\node at (0.25,3.75) {$2$};
	\node at (0.25,3.25) {$1$};
	\node at (0.75,3.5) {$1$};
	\begin{scope}[shift={(0,4)}]
		
		\draw (1,0)--(1,1)--(0,1)--(0,0);
		\draw (0,0)--(0.5,1);\draw (0,0)--(1,1);
		
		\draw (1,1)--(1,2)--(0,2)--(0,1);
		\draw (0.5,1.5)--(1,1.5);\draw (0.5,1)--(0.5,2);
		
		\draw (1,2)--(1,3)--(0,3)--(0,2);
		\draw (0.5,2)--(1,3);	\draw (0,2)--(1,3);
		
		\draw (1,3)--(1,4)--(0,4)--(0,3);
		\draw (0.5,3)--(0.5,4);\draw (0,3.5)--(0.5,3.5);
		\node at (0.15,0.75) {$2$};	
		\node at (0.53,0.75) {$0$};
		\node at (0.7,0.3) {$2$};
		\node at (0.75,1.75) {$2$};
		\node at (0.75,1.25) {$1$};
		\node at (0.25,1.5) {$1$};
		\node at (0.85,2.25) {$2$};	
		\node at (0.48,2.25) {$0$};
		\node at (0.3,2.7) {$2$};
		\node at (0.25,3.75) {$2$};
		\node at (0.25,3.25) {$1$};
		\node at (0.75,3.5) {$1$};
	\end{scope}
	\begin{scope}[shift={(0,8)}]
		
		\draw (1,0)--(1,1)--(0,1)--(0,0);
		\draw (0,0)--(0.5,1);\draw (0,0)--(1,1);
		
		\draw (1,1)--(1,2)--(0,2)--(0,1);
		\draw (0.5,1.5)--(1,1.5);\draw (0.5,1)--(0.5,2);
		
		\draw (1,2)--(1,3)--(0,3)--(0,2);
		\draw (0.5,2)--(1,3);	\draw (0,2)--(1,3);

		\node at (0.15,0.75) {$2$};	
		\node at (0.53,0.75) {$0$};
		\node at (0.7,0.3) {$2$};

		\node at (0.25,1.5) {$1$};

		\node at (0.48,2.25) {$0$};

	\end{scope}
\end{scope}	

\begin{scope}	[shift={(-2,0)}]		
	\draw (0,0)--(0,-0.5);
	\draw (1,0)--(1,-0.5);
	\node at (0.5,-0.3) {$\vdots$};
	
	\draw (0,0)--(1,0)--(1,1)--(0,1)--(0,0);
	\draw (0,0)--(0.5,1);\draw (0,0)--(1,1);
	
	\draw (1,1)--(1,2)--(0,2)--(0,1);
	\draw (0.5,1.5)--(1,1.5);\draw (0.5,1)--(0.5,2);
	
	\draw (1,2)--(1,3)--(0,3)--(0,2);
	\draw (0.5,2)--(1,3);	\draw (0,2)--(1,3);
	
	\draw (1,3)--(1,4)--(0,4)--(0,3);
	\draw (0.5,3)--(0.5,4);\draw (0,3.5)--(0.5,3.5);
	\node at (0.15,0.75) {$2$};	
	\node at (0.53,0.75) {$0$};
	\node at (0.7,0.3) {$2$};
	\node at (0.75,1.75) {$2$};
	\node at (0.75,1.25) {$1$};
	\node at (0.25,1.5) {$1$};
	\node at (0.85,2.25) {$2$};	
	\node at (0.48,2.25) {$0$};
	\node at (0.3,2.7) {$2$};
	\node at (0.25,3.75) {$2$};
	\node at (0.25,3.25) {$1$};
	\node at (0.75,3.5) {$1$};
	\begin{scope}[shift={(0,4)}]
		
		\draw (1,0)--(1,1)--(0,1)--(0,0);
		\draw (0,0)--(0.5,1);\draw (0,0)--(1,1);
		
		\draw (1,1)--(1,2)--(0,2)--(0,1);
		\draw (0.5,1.5)--(1,1.5);\draw (0.5,1)--(0.5,2);
		
		\draw (1,2)--(1,3)--(0,3)--(0,2);
		\draw (0.5,2)--(1,3);	\draw (0,2)--(1,3);

		\node at (0.15,0.75) {$2$};	
		\node at (0.53,0.75) {$0$};
		\node at (0.7,0.3) {$2$};
		\node at (0.75,1.75) {$2$};
		\node at (0.75,1.25) {$1$};
		\node at (0.25,1.5) {$1$};

		\node at (0.48,2.25) {$0$};
		\node at (0.3,2.7) {$2$};

	\end{scope}

\end{scope}	

\begin{scope}	[shift={(-3,0)}]		
	\draw (0,0)--(0,-0.5);
	\draw (1,0)--(1,-0.5);
	\node at (0.5,-0.3) {$\vdots$};
	
	\draw (0,0)--(1,0)--(1,1)--(0,1)--(0,0);
	\draw (0,0)--(0.5,1);\draw (0,0)--(1,1);
	
	\draw (1,1)--(1,2)--(0,2)--(0,1);
	\draw (0.5,1.5)--(1,1.5);\draw (0.5,1)--(0.5,2);
	
	\draw (1,2)--(1,3)--(0,3)--(0,2);
	\draw (0.5,2)--(1,3);	\draw (0,2)--(1,3);
	
	\draw (1,3)--(1,4)--(0,4)--(0,3);
	\draw (0.5,3)--(0.5,4);\draw (0,3.5)--(0.5,3.5);
	\node at (0.15,0.75) {$2$};	
	\node at (0.53,0.75) {$0$};
	\node at (0.7,0.3) {$2$};
	\node at (0.75,1.75) {$2$};
	\node at (0.75,1.25) {$1$};
	\node at (0.25,1.5) {$1$};
	\node at (0.85,2.25) {$2$};	
	\node at (0.48,2.25) {$0$};
	\node at (0.3,2.7) {$2$};
	\node at (0.25,3.75) {$2$};
	\node at (0.25,3.25) {$1$};
	\node at (0.75,3.5) {$1$};
	\begin{scope}[shift={(0,4)}]
		
		\draw (1,0)--(1,1)--(0,1)--(0,0);
		\draw (0,0)--(0.5,1);\draw (0,0)--(1,1);	
		\node at (0.53,0.75) {$0$};
	\end{scope}
	
\end{scope}	

\begin{scope}	[shift={(-4,0)}]		
	\draw (0,0)--(0,-0.5);
	\draw (1,0)--(1,-0.5);
	\node at (0.5,-0.3) {$\vdots$};
	
	\draw (0,0)--(1,0)--(1,1)--(0,1)--(0,0);
	\draw (0,0)--(0.5,1);\draw (0,0)--(1,1);
	
	\draw (1,1)--(1,2)--(0,2)--(0,1);
	\draw (0.5,1.5)--(1,1.5);\draw (0.5,1)--(0.5,2);
	
	\draw (1,2)--(1,3)--(0,3)--(0,2);
	\draw (0.5,2)--(1,3);	\draw (0,2)--(1,3);

	\node at (0.15,0.75) {$2$};	
	\node at (0.53,0.75) {$0$};
	\node at (0.7,0.3) {$2$};
	\node at (0.75,1.25) {$1$};
	\node at (0.25,1.5) {$1$};
	\node at (0.48,2.25) {$0$};
	
\end{scope}	

\begin{scope}	[shift={(-5,0)}]		
	\draw (0,0)--(0,-0.5);
	\draw (1,0)--(1,-0.5);
	\node at (0.5,-0.3) {$\vdots$};
	
	\draw (0,0)--(1,0)--(1,1)--(0,1)--(0,0);
	\draw (0,0)--(0.5,1);\draw (0,0)--(1,1);
	
	\draw (1,1)--(1,2)--(0,2)--(0,1);
	\draw (0.5,1.5)--(1,1.5);\draw (0.5,1)--(0.5,2);
	
	\draw (1,2)--(1,3)--(0,3)--(0,2);
	\draw (0.5,2)--(1,3);	\draw (0,2)--(1,3);

	\node at (0.15,0.75) {$2$};	

	\node at (0.7,0.3) {$2$};

	\node at (0.25,1.5) {$1$};
	\node at (0.48,2.25) {$0$};

	\draw (0,0)--(-0.5,0);
\draw (0,1)--(-0.5,1);
\draw (0,2)--(-0.5,2);
\draw (0,3)--(-0.5,3);

\node at (-0.4,0.5) {$\cdots$};
\node at (-0.4,1.5) {$\cdots$};
\node at (-0.4,2.5) {$\cdots$};
\node at (3,-1.2) {(a)};
\end{scope}	
\begin{scope}	[shift={(10,0)}]	
\begin{scope}	[shift={(0,0)}]		
	\draw (0,0)--(0,-0.5);
	\draw (1,0)--(1,-0.5);
	\node at (0.5,-0.3) {$\vdots$};
	
	\draw (0,0)--(1,0)--(1,1)--(0,1)--(0,0);
	\draw (0,0)--(0.5,1);\draw (0,0)--(1,1);
	
	\draw (1,1)--(1,2)--(0,2)--(0,1);
	\draw (0.5,1.5)--(1,1.5);\draw (0.5,1)--(0.5,2);
	
	\draw (1,2)--(1,3)--(0,3)--(0,2);
	\draw (0.5,2)--(1,3);	\draw (0,2)--(1,3);
	
	\draw (1,3)--(1,4)--(0,4)--(0,3);
	\draw (0.5,3)--(0.5,4);\draw (0,3.5)--(0.5,3.5);
	\node at (0.15,0.75) {$2$};	
	\node at (0.53,0.75) {$0$};
	\node at (0.7,0.3) {$2$};
	\node at (0.75,1.75) {$2$};
	\node at (0.75,1.25) {$1$};
	\node at (0.25,1.5) {$1$};
	\node at (0.85,2.25) {$2$};	
	\node at (0.48,2.25) {$0$};
	\node at (0.3,2.7) {$2$};
	\node at (0.25,3.75) {$2$};
	\node at (0.25,3.25) {$1$};
	\node at (0.75,3.5) {$1$};
	\begin{scope}[shift={(0,4)}]
		
		\draw (1,0)--(1,1)--(0,1)--(0,0);
		\draw (0,0)--(0.5,1);\draw (0,0)--(1,1);
		
		\draw (1,1)--(1,2)--(0,2)--(0,1);
		\draw (0.5,1.5)--(1,1.5);\draw (0.5,1)--(0.5,2);
		
		\draw (1,2)--(1,3)--(0,3)--(0,2);
		\draw (0.5,2)--(1,3);	\draw (0,2)--(1,3);
		
		\draw (1,3)--(1,4)--(0,4)--(0,3);
		\draw (0.5,3)--(0.5,4);\draw (0,3.5)--(0.5,3.5);
		\node at (0.15,0.75) {$2$};	
		\node at (0.53,0.75) {$0$};
		\node at (0.7,0.3) {$2$};
		\node at (0.75,1.75) {$2$};
		\node at (0.75,1.25) {$1$};
		\node at (0.25,1.5) {$1$};
		\node at (0.85,2.25) {$2$};	
		\node at (0.48,2.25) {$0$};
		\node at (0.3,2.7) {$2$};
		\node at (0.25,3.75) {$2$};
		\node at (0.25,3.25) {$1$};
		\node at (0.75,3.5) {$1$};
	\end{scope}
	\begin{scope}[shift={(0,8)}]
		
		\draw (1,0)--(1,1)--(0,1)--(0,0);
		\draw (0,0)--(0.5,1);\draw (0,0)--(1,1);
		
		\draw (1,1)--(1,2)--(0,2)--(0,1);
		\draw (0.5,1.5)--(1,1.5);\draw (0.5,1)--(0.5,2);
		
		\draw (1,2)--(1,3)--(0,3)--(0,2);
		\draw (0.5,2)--(1,3);	\draw (0,2)--(1,3);
		
		\draw (1,3)--(1,4)--(0,4)--(0,3);
		\draw (0.5,3)--(0.5,4);\draw (0,3.5)--(0.5,3.5);
		\draw (1,4)--(1,5)--(0,5)--(0,4);\draw (0,4)--(1,5);\draw (0,4)--(0.5,5);
			\node at (0.15,4.75) {$2$};	
		\node at (0.7,4.3) {$2$};
		\node at (0.15,0.75) {$2$};	
		\node at (0.53,0.75) {$0$};
		\node at (0.7,0.3) {$2$};
		\node at (0.75,1.75) {$2$};
		\node at (0.75,1.25) {$1$};
		\node at (0.25,1.5) {$1$};
		\node at (0.85,2.25) {$2$};	
		\node at (0.48,2.25) {$0$};
		\node at (0.3,2.7) {$2$};
		\node at (0.25,3.75) {$2$};
		\node at (0.25,3.25) {$1$};
		\node at (0.75,3.5) {$1$};
	\end{scope}
\end{scope}	

\begin{scope}	[shift={(-1,0)}]		
	\draw (0,0)--(0,-0.5);
	\draw (1,0)--(1,-0.5);
	\node at (0.5,-0.3) {$\vdots$};
	
	\draw (0,0)--(1,0)--(1,1)--(0,1)--(0,0);
	\draw (0,0)--(0.5,1);\draw (0,0)--(1,1);
	
	\draw (1,1)--(1,2)--(0,2)--(0,1);
	\draw (0.5,1.5)--(1,1.5);\draw (0.5,1)--(0.5,2);
	
	\draw (1,2)--(1,3)--(0,3)--(0,2);
	\draw (0.5,2)--(1,3);	\draw (0,2)--(1,3);
	
	\draw (1,3)--(1,4)--(0,4)--(0,3);
	\draw (0.5,3)--(0.5,4);\draw (0,3.5)--(0.5,3.5);
	\node at (0.15,0.75) {$2$};	
	\node at (0.53,0.75) {$0$};
	\node at (0.7,0.3) {$2$};
	\node at (0.75,1.75) {$2$};
	\node at (0.75,1.25) {$1$};
	\node at (0.25,1.5) {$1$};
	\node at (0.85,2.25) {$2$};	
	\node at (0.48,2.25) {$0$};
	\node at (0.3,2.7) {$2$};
	\node at (0.25,3.75) {$2$};
	\node at (0.25,3.25) {$1$};
	\node at (0.75,3.5) {$1$};
	\begin{scope}[shift={(0,4)}]
		
		\draw (1,0)--(1,1)--(0,1)--(0,0);
		\draw (0,0)--(0.5,1);\draw (0,0)--(1,1);
		
		\draw (1,1)--(1,2)--(0,2)--(0,1);
		\draw (0.5,1.5)--(1,1.5);\draw (0.5,1)--(0.5,2);
		
		\draw (1,2)--(1,3)--(0,3)--(0,2);
		\draw (0.5,2)--(1,3);	\draw (0,2)--(1,3);
		
		\draw (1,3)--(1,4)--(0,4)--(0,3);
		\draw (0.5,3)--(0.5,4);\draw (0,3.5)--(0.5,3.5);
		\node at (0.15,0.75) {$2$};	
		\node at (0.53,0.75) {$0$};
		\node at (0.7,0.3) {$2$};
		\node at (0.75,1.75) {$2$};
		\node at (0.75,1.25) {$1$};
		\node at (0.25,1.5) {$1$};
		\node at (0.85,2.25) {$2$};	
		\node at (0.48,2.25) {$0$};
		\node at (0.3,2.7) {$2$};
		\node at (0.25,3.75) {$2$};
		\node at (0.25,3.25) {$1$};
		\node at (0.75,3.5) {$1$};
	\end{scope}
	\begin{scope}[shift={(0,8)}]
		
		\draw (1,0)--(1,1)--(0,1)--(0,0);
		\draw (0,0)--(0.5,1);\draw (0,0)--(1,1);
		
		\draw (1,1)--(1,2)--(0,2)--(0,1);
		\draw (0.5,1.5)--(1,1.5);\draw (0.5,1)--(0.5,2);
		
	\node at (0.75,1.75) {$2$};
\node at (0.75,1.25) {$1$};

		\node at (0.15,0.75) {$2$};	
		\node at (0.53,0.75) {$0$};
		\node at (0.7,0.3) {$2$};
		
		\node at (0.25,1.5) {$1$};
	\end{scope}
\end{scope}	

\begin{scope}	[shift={(-2,0)}]		
	\draw (0,0)--(0,-0.5);
	\draw (1,0)--(1,-0.5);
	\node at (0.5,-0.3) {$\vdots$};
	
	\draw (0,0)--(1,0)--(1,1)--(0,1)--(0,0);
	\draw (0,0)--(0.5,1);\draw (0,0)--(1,1);
	
	\draw (1,1)--(1,2)--(0,2)--(0,1);
	\draw (0.5,1.5)--(1,1.5);\draw (0.5,1)--(0.5,2);
	
	\draw (1,2)--(1,3)--(0,3)--(0,2);
	\draw (0.5,2)--(1,3);	\draw (0,2)--(1,3);
	
	\draw (1,3)--(1,4)--(0,4)--(0,3);
	\draw (0.5,3)--(0.5,4);\draw (0,3.5)--(0.5,3.5);
	\node at (0.15,0.75) {$2$};	
	\node at (0.53,0.75) {$0$};
	\node at (0.7,0.3) {$2$};
	\node at (0.75,1.75) {$2$};
	\node at (0.75,1.25) {$1$};
	\node at (0.25,1.5) {$1$};
	\node at (0.85,2.25) {$2$};	
	\node at (0.48,2.25) {$0$};
	\node at (0.3,2.7) {$2$};
	\node at (0.25,3.75) {$2$};
	\node at (0.25,3.25) {$1$};
	\node at (0.75,3.5) {$1$};
	\begin{scope}[shift={(0,4)}]
		
		\draw (1,0)--(1,1)--(0,1)--(0,0);
		\draw (0,0)--(0.5,1);\draw (0,0)--(1,1);
		
		\draw (1,1)--(1,2)--(0,2)--(0,1);
		\draw (0.5,1.5)--(1,1.5);\draw (0.5,1)--(0.5,2);
		
		\draw (1,2)--(1,3)--(0,3)--(0,2);
		\draw (0.5,2)--(1,3);	\draw (0,2)--(1,3);
		
		\draw (1,3)--(1,4)--(0,4)--(0,3);
		\draw (0.5,3)--(0.5,4);\draw (0,3.5)--(0.5,3.5);
		\node at (0.15,0.75) {$2$};	
		\node at (0.53,0.75) {$0$};
		\node at (0.7,0.3) {$2$};
		\node at (0.75,1.75) {$2$};
		\node at (0.75,1.25) {$1$};
		\node at (0.25,1.5) {$1$};
		\node at (0.85,2.25) {$2$};	
		\node at (0.48,2.25) {$0$};
		\node at (0.3,2.7) {$2$};
		\node at (0.25,3.75) {$2$};
		\node at (0.25,3.25) {$1$};
		\node at (0.75,3.5) {$1$};
	\end{scope}
	\begin{scope}[shift={(0,8)}]
		
		\draw (1,0)--(1,1)--(0,1)--(0,0);\draw (0,0)--(1,1);
		\node at (0.7,0.3) {$2$};
	\end{scope}
\end{scope}	
\begin{scope}	[shift={(-3,0)}]		
	\draw (0,0)--(0,-0.5);
	\draw (1,0)--(1,-0.5);
	\node at (0.5,-0.3) {$\vdots$};
	
	\draw (0,0)--(1,0)--(1,1)--(0,1)--(0,0);
	\draw (0,0)--(0.5,1);\draw (0,0)--(1,1);
	
	\draw (1,1)--(1,2)--(0,2)--(0,1);
	\draw (0.5,1.5)--(1,1.5);\draw (0.5,1)--(0.5,2);
	
	\draw (1,2)--(1,3)--(0,3)--(0,2);
	\draw (0.5,2)--(1,3);	\draw (0,2)--(1,3);
	
	\draw (1,3)--(1,4)--(0,4)--(0,3);
	\draw (0.5,3)--(0.5,4);\draw (0,3.5)--(0.5,3.5);
	\node at (0.15,0.75) {$2$};	
	\node at (0.53,0.75) {$0$};
	\node at (0.7,0.3) {$2$};
	\node at (0.75,1.75) {$2$};
	\node at (0.75,1.25) {$1$};
	\node at (0.25,1.5) {$1$};
	\node at (0.85,2.25) {$2$};	
	\node at (0.48,2.25) {$0$};
	\node at (0.3,2.7) {$2$};
	\node at (0.25,3.75) {$2$};
	\node at (0.25,3.25) {$1$};
	\node at (0.75,3.5) {$1$};
	\begin{scope}[shift={(0,4)}]
		
		\draw (1,0)--(1,1)--(0,1)--(0,0);
		\draw (0,0)--(0.5,1);\draw (0,0)--(1,1);
		
		\draw (1,1)--(1,2)--(0,2)--(0,1);
		\draw (0.5,1.5)--(1,1.5);\draw (0.5,1)--(0.5,2);
		
		\draw (1,2)--(1,3)--(0,3)--(0,2);
		\draw (0.5,2)--(1,3);	\draw (0,2)--(1,3);
		
		\draw (1,3)--(1,4)--(0,4)--(0,3);
		\draw (0.5,3)--(0.5,4);\draw (0,3.5)--(0.5,3.5);
		\node at (0.15,0.75) {$2$};	
		\node at (0.53,0.75) {$0$};
		\node at (0.7,0.3) {$2$};
		\node at (0.75,1.75) {$2$};
		\node at (0.75,1.25) {$1$};
		\node at (0.25,1.5) {$1$};
		\node at (0.48,2.25) {$0$};
		\node at (0.3,2.7) {$2$};
		\node at (0.25,3.25) {$1$};
	\end{scope}

\end{scope}	

\begin{scope}	[shift={(-4,0)}]		
	\draw (0,0)--(0,-0.5);
	\draw (1,0)--(1,-0.5);
	\node at (0.5,-0.3) {$\vdots$};
	
	\draw (0,0)--(1,0)--(1,1)--(0,1)--(0,0);
	\draw (0,0)--(0.5,1);\draw (0,0)--(1,1);
	
	\draw (1,1)--(1,2)--(0,2)--(0,1);
	\draw (0.5,1.5)--(1,1.5);\draw (0.5,1)--(0.5,2);
	
	\draw (1,2)--(1,3)--(0,3)--(0,2);
	\draw (0.5,2)--(1,3);	\draw (0,2)--(1,3);
	
	\draw (1,3)--(1,4)--(0,4)--(0,3);
	\draw (0.5,3)--(0.5,4);\draw (0,3.5)--(0.5,3.5);
	\node at (0.15,0.75) {$2$};	
	\node at (0.53,0.75) {$0$};
	\node at (0.7,0.3) {$2$};
	\node at (0.75,1.75) {$2$};
	\node at (0.75,1.25) {$1$};
	\node at (0.25,1.5) {$1$};
	\node at (0.85,2.25) {$2$};	
	\node at (0.48,2.25) {$0$};
	\node at (0.3,2.7) {$2$};
	\node at (0.25,3.75) {$2$};
	\node at (0.25,3.25) {$1$};
	\node at (0.75,3.5) {$1$};
	\begin{scope}[shift={(0,4)}]
		
		\draw (1,0)--(1,1)--(0,1)--(0,0);
		\draw (0,0)--(0.5,1);\draw (0,0)--(1,1);
		\node at (0.53,0.75) {$0$};
		\node at (0.7,0.3) {$2$};
	\end{scope}
	
\end{scope}	

\begin{scope}	[shift={(-5,0)}]		
	\draw (0,0)--(0,-0.5);
	\draw (1,0)--(1,-0.5);
	\node at (0.5,-0.3) {$\vdots$};
	
	\draw (0,0)--(1,0)--(1,1)--(0,1)--(0,0);
	\draw (0,0)--(0.5,1);\draw (0,0)--(1,1);
	
	\draw (1,1)--(1,2)--(0,2)--(0,1);
	\draw (0.5,1.5)--(1,1.5);\draw (0.5,1)--(0.5,2);
	
	\draw (1,2)--(1,3)--(0,3)--(0,2);
	\draw (0.5,2)--(1,3);	\draw (0,2)--(1,3);
	
	\draw (1,3)--(1,4)--(0,4)--(0,3);
	\draw (0.5,3)--(0.5,4);\draw (0,3.5)--(0.5,3.5);
		\node at (0.53,0.75) {$0$};
	\node at (0.15,0.75) {$2$};	
	\node at (0.7,0.3) {$2$};
	\node at (0.75,1.75) {$2$};
	\node at (0.75,1.25) {$1$};
	\node at (0.25,1.5) {$1$};
	\node at (0.85,2.25) {$2$};	

	\node at (0.3,2.7) {$2$};

	\node at (0.75,3.5) {$1$};
	\begin{scope}[shift={(0,4)}]
		
		\draw (1,0)--(1,1)--(0,1)--(0,0);
		\draw (0,0)--(0.5,1);\draw (0,0)--(1,1);
		\node at (0.53,0.75) {$0$};
	\end{scope}
	
\end{scope}	

\begin{scope}	[shift={(-6,0)}]		
	\draw (0,0)--(0,-0.5);
	\draw (1,0)--(1,-0.5);
	\node at (0.5,-0.3) {$\vdots$};
	
	\draw (0,0)--(1,0)--(1,1)--(0,1)--(0,0);
	\draw (0,0)--(0.5,1);\draw (0,0)--(1,1);
	
	\draw (1,1)--(1,2)--(0,2)--(0,1);
	\draw (0.5,1.5)--(1,1.5);\draw (0.5,1)--(0.5,2);
	\node at (0.53,0.75) {$0$};
	\node at (0.7,0.3) {$2$};
	\node at (0.75,1.75) {$2$};
	\node at (0.75,1.25) {$1$};

	\draw (0,0)--(-0.5,0);
\draw (0,1)--(-0.5,1);
\draw (0,2)--(-0.5,2);

\node at (-0.4,0.5) {$\cdots$};
\node at (-0.4,1.5) {$\cdots$};
\node at (3.5,-1.2) {(b)};
\end{scope}	

\end{scope}
			\end{tikzpicture}
	\end{center}
	\begin{center}
	\begin{tikzpicture}[scale=0.7]
		\begin{scope}	[shift={(0,0)}]		
			\draw (0,0)--(0,-0.5);
			\draw (1,0)--(1,-0.5);
			\node at (0.5,-0.3) {$\vdots$};
			
			\draw (0,0)--(1,0)--(1,1)--(0,1)--(0,0);
			\draw (0,0)--(0.5,1);\draw (0,0)--(1,1);
			
			\draw (1,1)--(1,2)--(0,2)--(0,1);
			\draw (0.5,1.5)--(1,1.5);\draw (0.5,1)--(0.5,2);
			
			\draw (1,2)--(1,3)--(0,3)--(0,2);
			\draw (0.5,2)--(1,3);	\draw (0,2)--(1,3);
			
			\draw (1,3)--(1,4)--(0,4)--(0,3);
			\draw (0.5,3)--(0.5,4);\draw (0,3.5)--(0.5,3.5);
			\node at (0.15,0.75) {$2$};	
			\node at (0.53,0.75) {$0$};
			\node at (0.7,0.3) {$2$};
			\node at (0.75,1.75) {$2$};
			\node at (0.75,1.25) {$1$};
			\node at (0.25,1.5) {$1$};
			\node at (0.85,2.25) {$2$};	
			\node at (0.48,2.25) {$0$};
			\node at (0.3,2.7) {$2$};
			\node at (0.25,3.75) {$2$};
			\node at (0.25,3.25) {$1$};
			\node at (0.75,3.5) {$1$};
			\begin{scope}[shift={(0,4)}]
				
				\draw (1,0)--(1,1)--(0,1)--(0,0);
				\draw (0,0)--(0.5,1);\draw (0,0)--(1,1);
				
				\draw (1,1)--(1,2)--(0,2)--(0,1);
				\draw (0.5,1.5)--(1,1.5);\draw (0.5,1)--(0.5,2);
				
				\draw (1,2)--(1,3)--(0,3)--(0,2);
				\draw (0.5,2)--(1,3);	\draw (0,2)--(1,3);
				
				\draw (1,3)--(1,4)--(0,4)--(0,3);
				\draw (0.5,3)--(0.5,4);\draw (0,3.5)--(0.5,3.5);
				\node at (0.15,0.75) {$2$};	
				\node at (0.53,0.75) {$0$};
				\node at (0.7,0.3) {$2$};
				\node at (0.75,1.75) {$2$};
				\node at (0.75,1.25) {$1$};
				\node at (0.25,1.5) {$1$};
				\node at (0.85,2.25) {$2$};	
				\node at (0.48,2.25) {$0$};
				\node at (0.3,2.7) {$2$};
				\node at (0.25,3.75) {$2$};
				\node at (0.25,3.25) {$1$};
				\node at (0.75,3.5) {$1$};
			\end{scope}
			\begin{scope}[shift={(0,8)}]
				
				\draw (1,0)--(1,1)--(0,1)--(0,0);
				\draw (0,0)--(0.5,1);\draw (0,0)--(1,1);
				
				\draw (1,1)--(1,2)--(0,2)--(0,1);
				\draw (0.5,1.5)--(1,1.5);\draw (0.5,1)--(0.5,2);
				
				\draw (1,2)--(1,3)--(0,3)--(0,2);
				\draw (0.5,2)--(1,3);	\draw (0,2)--(1,3);
				
				\draw (1,3)--(1,4)--(0,4)--(0,3);
				\draw (0.5,3)--(0.5,4);\draw (0,3.5)--(0.5,3.5);
	
				\node at (0.15,0.75) {$2$};	
				\node at (0.53,0.75) {$0$};
				\node at (0.7,0.3) {$2$};
				\node at (0.75,1.75) {$2$};
				\node at (0.75,1.25) {$1$};
				\node at (0.25,1.5) {$1$};
				\node at (0.85,2.25) {$2$};	
				\node at (0.48,2.25) {$0$};
				\node at (0.3,2.7) {$2$};
				\node at (0.25,3.25) {$1$};
				\node at (0.75,3.5) {$1$};
			\end{scope}
		\end{scope}	
		
		\begin{scope}	[shift={(-1,0)}]		
			\draw (0,0)--(0,-0.5);
			\draw (1,0)--(1,-0.5);
			\node at (0.5,-0.3) {$\vdots$};
			
			\draw (0,0)--(1,0)--(1,1)--(0,1)--(0,0);
			\draw (0,0)--(0.5,1);\draw (0,0)--(1,1);
			
			\draw (1,1)--(1,2)--(0,2)--(0,1);
			\draw (0.5,1.5)--(1,1.5);\draw (0.5,1)--(0.5,2);
			
			\draw (1,2)--(1,3)--(0,3)--(0,2);
			\draw (0.5,2)--(1,3);	\draw (0,2)--(1,3);
			
			\draw (1,3)--(1,4)--(0,4)--(0,3);
			\draw (0.5,3)--(0.5,4);\draw (0,3.5)--(0.5,3.5);
			\node at (0.15,0.75) {$2$};	
			\node at (0.53,0.75) {$0$};
			\node at (0.7,0.3) {$2$};
			\node at (0.75,1.75) {$2$};
			\node at (0.75,1.25) {$1$};
			\node at (0.25,1.5) {$1$};
			\node at (0.85,2.25) {$2$};	
			\node at (0.48,2.25) {$0$};
			\node at (0.3,2.7) {$2$};
			\node at (0.25,3.75) {$2$};
			\node at (0.25,3.25) {$1$};
			\node at (0.75,3.5) {$1$};
			\begin{scope}[shift={(0,4)}]
				
				\draw (1,0)--(1,1)--(0,1)--(0,0);
				\draw (0,0)--(0.5,1);\draw (0,0)--(1,1);
				
				\draw (1,1)--(1,2)--(0,2)--(0,1);
				\draw (0.5,1.5)--(1,1.5);\draw (0.5,1)--(0.5,2);
				
				\draw (1,2)--(1,3)--(0,3)--(0,2);
				\draw (0.5,2)--(1,3);	\draw (0,2)--(1,3);
				
				\draw (1,3)--(1,4)--(0,4)--(0,3);
				\draw (0.5,3)--(0.5,4);\draw (0,3.5)--(0.5,3.5);
				\node at (0.15,0.75) {$2$};	
				\node at (0.53,0.75) {$0$};
				\node at (0.7,0.3) {$2$};
				\node at (0.75,1.75) {$2$};
				\node at (0.75,1.25) {$1$};
				\node at (0.25,1.5) {$1$};
				\node at (0.85,2.25) {$2$};	
				\node at (0.48,2.25) {$0$};
				\node at (0.3,2.7) {$2$};
				\node at (0.25,3.75) {$2$};
				\node at (0.25,3.25) {$1$};
				\node at (0.75,3.5) {$1$};
			\end{scope}
			\begin{scope}[shift={(0,8)}]
				
				\draw (1,0)--(1,1)--(0,1)--(0,0);
				\draw (0,0)--(0.5,1);\draw (0,0)--(1,1);
				\node at (0.53,0.75) {$0$};
				\node at (0.7,0.3) {$2$};

			\end{scope}
		\end{scope}	
		
	\begin{scope}	[shift={(-2,0)}]		
		\draw (0,0)--(0,-0.5);
		\draw (1,0)--(1,-0.5);
		\node at (0.5,-0.3) {$\vdots$};
		
		\draw (0,0)--(1,0)--(1,1)--(0,1)--(0,0);
		\draw (0,0)--(0.5,1);\draw (0,0)--(1,1);
		
		\draw (1,1)--(1,2)--(0,2)--(0,1);
		\draw (0.5,1.5)--(1,1.5);\draw (0.5,1)--(0.5,2);
		
		\draw (1,2)--(1,3)--(0,3)--(0,2);
		\draw (0.5,2)--(1,3);	\draw (0,2)--(1,3);
		
		\draw (1,3)--(1,4)--(0,4)--(0,3);
		\draw (0.5,3)--(0.5,4);\draw (0,3.5)--(0.5,3.5);
		\node at (0.15,0.75) {$2$};	
		\node at (0.53,0.75) {$0$};
		\node at (0.7,0.3) {$2$};
		\node at (0.75,1.75) {$2$};
		\node at (0.75,1.25) {$1$};
		\node at (0.25,1.5) {$1$};
		\node at (0.85,2.25) {$2$};	
		\node at (0.48,2.25) {$0$};
		\node at (0.3,2.7) {$2$};
		\node at (0.25,3.75) {$2$};
		\node at (0.25,3.25) {$1$};
		\node at (0.75,3.5) {$1$};
		\begin{scope}[shift={(0,4)}]
			
			\draw (1,0)--(1,1)--(0,1)--(0,0);
			\draw (0,0)--(0.5,1);\draw (0,0)--(1,1);
			
			\draw (1,1)--(1,2)--(0,2)--(0,1);
			\draw (0.5,1.5)--(1,1.5);\draw (0.5,1)--(0.5,2);
			
			\draw (1,2)--(1,3)--(0,3)--(0,2);
			\draw (0.5,2)--(1,3);	\draw (0,2)--(1,3);

			\node at (0.15,0.75) {$2$};	
			\node at (0.53,0.75) {$0$};
			\node at (0.7,0.3) {$2$};
			\node at (0.75,1.25) {$1$};
			\node at (0.25,1.5) {$1$};
			\node at (0.48,2.25) {$0$};

		\end{scope}
	
	\end{scope}	
		
	\begin{scope}	[shift={(-3,0)}]		
		\draw (0,0)--(0,-0.5);
		\draw (1,0)--(1,-0.5);
		\node at (0.5,-0.3) {$\vdots$};
		
		\draw (0,0)--(1,0)--(1,1)--(0,1)--(0,0);
		\draw (0,0)--(0.5,1);\draw (0,0)--(1,1);
		
		\draw (1,1)--(1,2)--(0,2)--(0,1);
		\draw (0.5,1.5)--(1,1.5);\draw (0.5,1)--(0.5,2);
		
		\draw (1,2)--(1,3)--(0,3)--(0,2);
		\draw (0.5,2)--(1,3);	\draw (0,2)--(1,3);
		
		\draw (1,3)--(1,4)--(0,4)--(0,3);
		\draw (0.5,3)--(0.5,4);\draw (0,3.5)--(0.5,3.5);
		\node at (0.15,0.75) {$2$};	
		\node at (0.53,0.75) {$0$};
		\node at (0.7,0.3) {$2$};
		\node at (0.75,1.75) {$2$};
		\node at (0.75,1.25) {$1$};
		\node at (0.25,1.5) {$1$};
		\node at (0.85,2.25) {$2$};	
		\node at (0.48,2.25) {$0$};
		\node at (0.3,2.7) {$2$};
		\node at (0.25,3.75) {$2$};
		\node at (0.25,3.25) {$1$};
		\node at (0.75,3.5) {$1$};
		\begin{scope}[shift={(0,4)}]
			
			\draw (1,0)--(1,1)--(0,1)--(0,0);
			\draw (0,0)--(0.5,1);\draw (0,0)--(1,1);
			
			\draw (1,1)--(1,2)--(0,2)--(0,1);
			\draw (0.5,1.5)--(1,1.5);\draw (0.5,1)--(0.5,2);

			\node at (0.15,0.75) {$2$};	
			\node at (0.53,0.75) {$0$};
			\node at (0.7,0.3) {$2$};
			\node at (0.75,1.25) {$1$};

		\end{scope}
		
	\end{scope}

		\begin{scope}	[shift={(-4,0)}]		
			\draw (0,0)--(0,-0.5);
			\draw (1,0)--(1,-0.5);
			\node at (0.5,-0.3) {$\vdots$};
			
			\draw (0,0)--(1,0)--(1,1)--(0,1)--(0,0);
			\draw (0,0)--(0.5,1);\draw (0,0)--(1,1);
			
			\draw (1,1)--(1,2)--(0,2)--(0,1);
			\draw (0.5,1.5)--(1,1.5);\draw (0.5,1)--(0.5,2);
			
			\draw (1,2)--(1,3)--(0,3)--(0,2);
			\draw (0.5,2)--(1,3);	\draw (0,2)--(1,3);

			\node at (0.15,0.75) {$2$};	
			\node at (0.53,0.75) {$0$};
			\node at (0.7,0.3) {$2$};
			\node at (0.75,1.25) {$1$};
			\node at (0.25,1.5) {$1$};
			\node at (0.48,2.25) {$0$};
			
		\end{scope}	
		
		\begin{scope}	[shift={(-5,0)}]		
			\draw (0,0)--(0,-0.5);
			\draw (1,0)--(1,-0.5);
			\node at (0.5,-0.3) {$\vdots$};
			
			\draw (0,0)--(1,0)--(1,1)--(0,1)--(0,0);
			\draw (0,0)--(0.5,1);\draw (0,0)--(1,1);
			
			\draw (1,1)--(1,2)--(0,2)--(0,1);
			\draw (0.5,1.5)--(1,1.5);\draw (0.5,1)--(0.5,2);
			
			\draw (1,2)--(1,3)--(0,3)--(0,2);
			\draw (0.5,2)--(1,3);	\draw (0,2)--(1,3);

			\node at (0.15,0.75) {$2$};	
			
			\node at (0.7,0.3) {$2$};
			
			\node at (0.25,1.5) {$1$};
			\node at (0.48,2.25) {$0$};
			
			\draw (0,0)--(-0.5,0);
			\draw (0,1)--(-0.5,1);
			\draw (0,2)--(-0.5,2);
			\draw (0,3)--(-0.5,3);
			
			\node at (-0.4,0.5) {$\cdots$};
			\node at (-0.4,1.5) {$\cdots$};
			\node at (-0.4,2.5) {$\cdots$};
		\node at (3,-1.2) {(c)};	
		\end{scope}	
		\begin{scope}	[shift={(9,0)}]	
			\begin{scope}	[shift={(0,0)}]		
				\draw (0,0)--(0,-0.5);
				\draw (1,0)--(1,-0.5);
				\node at (0.5,-0.3) {$\vdots$};
				
				\draw (0,0)--(1,0)--(1,1)--(0,1)--(0,0);
				\draw (0,0)--(0.5,1);\draw (0,0)--(1,1);
				
				\draw (1,1)--(1,2)--(0,2)--(0,1);
				\draw (0.5,1.5)--(1,1.5);\draw (0.5,1)--(0.5,2);
				
				\draw (1,2)--(1,3)--(0,3)--(0,2);
				\draw (0.5,2)--(1,3);	\draw (0,2)--(1,3);
				
				\draw (1,3)--(1,4)--(0,4)--(0,3);
				\draw (0.5,3)--(0.5,4);\draw (0,3.5)--(0.5,3.5);
				\node at (0.15,0.75) {$2$};	
				\node at (0.53,0.75) {$0$};
				\node at (0.7,0.3) {$2$};
				\node at (0.75,1.75) {$2$};
				\node at (0.75,1.25) {$1$};
				\node at (0.25,1.5) {$1$};
				\node at (0.85,2.25) {$2$};	
				\node at (0.48,2.25) {$0$};
				\node at (0.3,2.7) {$2$};
				\node at (0.25,3.75) {$2$};
				\node at (0.25,3.25) {$1$};
				\node at (0.75,3.5) {$1$};
				\begin{scope}[shift={(0,4)}]
					
					\draw (1,0)--(1,1)--(0,1)--(0,0);
					\draw (0,0)--(0.5,1);\draw (0,0)--(1,1);
					
					\draw (1,1)--(1,2)--(0,2)--(0,1);
					\draw (0.5,1.5)--(1,1.5);\draw (0.5,1)--(0.5,2);
					
					\draw (1,2)--(1,3)--(0,3)--(0,2);
					\draw (0.5,2)--(1,3);	\draw (0,2)--(1,3);
					
					\draw (1,3)--(1,4)--(0,4)--(0,3);
					\draw (0.5,3)--(0.5,4);\draw (0,3.5)--(0.5,3.5);
					\node at (0.15,0.75) {$2$};	
					\node at (0.53,0.75) {$0$};
					\node at (0.7,0.3) {$2$};
					\node at (0.75,1.75) {$2$};
					\node at (0.75,1.25) {$1$};
					\node at (0.25,1.5) {$1$};
					\node at (0.85,2.25) {$2$};	
					\node at (0.48,2.25) {$0$};
					\node at (0.3,2.7) {$2$};
					\node at (0.25,3.75) {$2$};
					\node at (0.25,3.25) {$1$};
					\node at (0.75,3.5) {$1$};
				\end{scope}
				\begin{scope}[shift={(0,8)}]
					
					\draw (1,0)--(1,1)--(0,1)--(0,0);
					\draw (0,0)--(0.5,1);\draw (0,0)--(1,1);
					
					\draw (1,1)--(1,2)--(0,2)--(0,1);
					\draw (0.5,1.5)--(1,1.5);\draw (0.5,1)--(0.5,2);
					
					\draw (1,2)--(1,3)--(0,3)--(0,2);
					\draw (0.5,2)--(1,3);	\draw (0,2)--(1,3);
					
					\draw (1,3)--(1,4)--(0,4)--(0,3);
					\draw (0.5,3)--(0.5,4);\draw (0,3.5)--(0.5,3.5);
					\draw (1,4)--(1,5)--(0,5)--(0,4);\draw (0,4)--(1,5);\draw (0,4)--(0.5,5);
					\draw (1,5)--(1,6)--(0,6)--(0,5);\draw (0.5,5)--(0.5,6);
						\node at (0.25,5.5) {$1$};	
					\node at (0.15,4.75) {$2$};	
					\node at (0.7,4.3) {$2$};
					\node at (0.15,0.75) {$2$};	
					\node at (0.53,0.75) {$0$};
					\node at (0.7,0.3) {$2$};
					\node at (0.75,1.75) {$2$};
					\node at (0.75,1.25) {$1$};
					\node at (0.25,1.5) {$1$};
					\node at (0.85,2.25) {$2$};	
					\node at (0.48,2.25) {$0$};
					\node at (0.3,2.7) {$2$};
					\node at (0.25,3.75) {$2$};
					\node at (0.25,3.25) {$1$};
					\node at (0.75,3.5) {$1$};
				\end{scope}
			\end{scope}	
			
		\begin{scope}	[shift={(-1,0)}]		
			\draw (0,0)--(0,-0.5);
			\draw (1,0)--(1,-0.5);
			\node at (0.5,-0.3) {$\vdots$};
			
			\draw (0,0)--(1,0)--(1,1)--(0,1)--(0,0);
			\draw (0,0)--(0.5,1);\draw (0,0)--(1,1);
			
			\draw (1,1)--(1,2)--(0,2)--(0,1);
			\draw (0.5,1.5)--(1,1.5);\draw (0.5,1)--(0.5,2);
			
			\draw (1,2)--(1,3)--(0,3)--(0,2);
			\draw (0.5,2)--(1,3);	\draw (0,2)--(1,3);
			
			\draw (1,3)--(1,4)--(0,4)--(0,3);
			\draw (0.5,3)--(0.5,4);\draw (0,3.5)--(0.5,3.5);
			\node at (0.15,0.75) {$2$};	
			\node at (0.53,0.75) {$0$};
			\node at (0.7,0.3) {$2$};
			\node at (0.75,1.75) {$2$};
			\node at (0.75,1.25) {$1$};
			\node at (0.25,1.5) {$1$};
			\node at (0.85,2.25) {$2$};	
			\node at (0.48,2.25) {$0$};
			\node at (0.3,2.7) {$2$};
			\node at (0.25,3.75) {$2$};
			\node at (0.25,3.25) {$1$};
			\node at (0.75,3.5) {$1$};
			\begin{scope}[shift={(0,4)}]
				
				\draw (1,0)--(1,1)--(0,1)--(0,0);
				\draw (0,0)--(0.5,1);\draw (0,0)--(1,1);
				
				\draw (1,1)--(1,2)--(0,2)--(0,1);
				\draw (0.5,1.5)--(1,1.5);\draw (0.5,1)--(0.5,2);
				
				\draw (1,2)--(1,3)--(0,3)--(0,2);
				\draw (0.5,2)--(1,3);	\draw (0,2)--(1,3);
				
				\draw (1,3)--(1,4)--(0,4)--(0,3);
				\draw (0.5,3)--(0.5,4);\draw (0,3.5)--(0.5,3.5);
				\node at (0.15,0.75) {$2$};	
				\node at (0.53,0.75) {$0$};
				\node at (0.7,0.3) {$2$};
				\node at (0.75,1.75) {$2$};
				\node at (0.75,1.25) {$1$};
				\node at (0.25,1.5) {$1$};
				\node at (0.85,2.25) {$2$};	
				\node at (0.48,2.25) {$0$};
				\node at (0.3,2.7) {$2$};
				\node at (0.25,3.75) {$2$};
				\node at (0.25,3.25) {$1$};
				\node at (0.75,3.5) {$1$};
			\end{scope}
			\begin{scope}[shift={(0,8)}]
				
				\draw (1,0)--(1,1)--(0,1)--(0,0);
				\draw (0,0)--(0.5,1);\draw (0,0)--(1,1);
				
				\draw (1,1)--(1,2)--(0,2)--(0,1);
				\draw (0.5,1.5)--(1,1.5);\draw (0.5,1)--(0.5,2);
				
				\draw (1,2)--(1,3)--(0,3)--(0,2);
				\draw (0.5,2)--(1,3);	\draw (0,2)--(1,3);

				\node at (0.15,0.75) {$2$};	
				\node at (0.53,0.75) {$0$};
				\node at (0.7,0.3) {$2$};
				\node at (0.75,1.75) {$2$};
				\node at (0.75,1.25) {$1$};
				\node at (0.25,1.5) {$1$};

				\node at (0.48,2.25) {$0$};

			\end{scope}
		\end{scope}	
			
			\begin{scope}	[shift={(-2,0)}]		
				\draw (0,0)--(0,-0.5);
				\draw (1,0)--(1,-0.5);
				\node at (0.5,-0.3) {$\vdots$};
				
				\draw (0,0)--(1,0)--(1,1)--(0,1)--(0,0);
				\draw (0,0)--(0.5,1);\draw (0,0)--(1,1);
				
				\draw (1,1)--(1,2)--(0,2)--(0,1);
				\draw (0.5,1.5)--(1,1.5);\draw (0.5,1)--(0.5,2);
				
				\draw (1,2)--(1,3)--(0,3)--(0,2);
				\draw (0.5,2)--(1,3);	\draw (0,2)--(1,3);
				
				\draw (1,3)--(1,4)--(0,4)--(0,3);
				\draw (0.5,3)--(0.5,4);\draw (0,3.5)--(0.5,3.5);
				\node at (0.15,0.75) {$2$};	
				\node at (0.53,0.75) {$0$};
				\node at (0.7,0.3) {$2$};
				\node at (0.75,1.75) {$2$};
				\node at (0.75,1.25) {$1$};
				\node at (0.25,1.5) {$1$};
				\node at (0.85,2.25) {$2$};	
				\node at (0.48,2.25) {$0$};
				\node at (0.3,2.7) {$2$};
				\node at (0.25,3.75) {$2$};
				\node at (0.25,3.25) {$1$};
				\node at (0.75,3.5) {$1$};
				\begin{scope}[shift={(0,4)}]
					
					\draw (1,0)--(1,1)--(0,1)--(0,0);
					\draw (0,0)--(0.5,1);\draw (0,0)--(1,1);
					
					\draw (1,1)--(1,2)--(0,2)--(0,1);
					\draw (0.5,1.5)--(1,1.5);\draw (0.5,1)--(0.5,2);
					
					\draw (1,2)--(1,3)--(0,3)--(0,2);
					\draw (0.5,2)--(1,3);	\draw (0,2)--(1,3);

					\node at (0.15,0.75) {$2$};	
					\node at (0.53,0.75) {$0$};
					\node at (0.7,0.3) {$2$};
					\node at (0.75,1.75) {$2$};
					\node at (0.75,1.25) {$1$};
					\node at (0.25,1.5) {$1$};
					\node at (0.85,2.25) {$2$};	
					\node at (0.3,2.7) {$2$};
	
				\end{scope}

			\end{scope}	
			\begin{scope}	[shift={(-3,0)}]		
				\draw (0,0)--(0,-0.5);
				\draw (1,0)--(1,-0.5);
				\node at (0.5,-0.3) {$\vdots$};
				
				\draw (0,0)--(1,0)--(1,1)--(0,1)--(0,0);
				\draw (0,0)--(0.5,1);\draw (0,0)--(1,1);
				
				\draw (1,1)--(1,2)--(0,2)--(0,1);
				\draw (0.5,1.5)--(1,1.5);\draw (0.5,1)--(0.5,2);
				
				\draw (1,2)--(1,3)--(0,3)--(0,2);
				\draw (0.5,2)--(1,3);	\draw (0,2)--(1,3);
				
				\draw (1,3)--(1,4)--(0,4)--(0,3);
				\draw (0.5,3)--(0.5,4);\draw (0,3.5)--(0.5,3.5);
				\node at (0.15,0.75) {$2$};	
				\node at (0.53,0.75) {$0$};
				\node at (0.7,0.3) {$2$};
				\node at (0.75,1.75) {$2$};
				\node at (0.75,1.25) {$1$};
				\node at (0.25,1.5) {$1$};
				\node at (0.85,2.25) {$2$};	
				\node at (0.48,2.25) {$0$};
				\node at (0.3,2.7) {$2$};
				\node at (0.25,3.75) {$2$};
				\node at (0.25,3.25) {$1$};
				\node at (0.75,3.5) {$1$};
				\begin{scope}[shift={(0,4)}]
					
					\draw (1,0)--(1,1)--(0,1)--(0,0);
					\draw (0,0)--(0.5,1);\draw (0,0)--(1,1);

					\node at (0.53,0.75) {$0$};
					\node at (0.7,0.3) {$2$};

				\end{scope}
				
			\end{scope}	
			
			\begin{scope}	[shift={(-4,0)}]		
				\draw (0,0)--(0,-0.5);
				\draw (1,0)--(1,-0.5);
				\node at (0.5,-0.3) {$\vdots$};
				
				\draw (0,0)--(1,0)--(1,1)--(0,1)--(0,0);
				\draw (0,0)--(0.5,1);\draw (0,0)--(1,1);
				
				\draw (1,1)--(1,2)--(0,2)--(0,1);
				\draw (0.5,1.5)--(1,1.5);\draw (0.5,1)--(0.5,2);
				
				\draw (1,2)--(1,3)--(0,3)--(0,2);
				\draw (0.5,2)--(1,3);	\draw (0,2)--(1,3);
				
				\draw (1,3)--(1,4)--(0,4)--(0,3);
				\draw (0.5,3)--(0.5,4);\draw (0,3.5)--(0.5,3.5);
				\node at (0.15,0.75) {$2$};	
				\node at (0.53,0.75) {$0$};
				\node at (0.7,0.3) {$2$};
				\node at (0.75,1.75) {$2$};
				\node at (0.75,1.25) {$1$};
				\node at (0.25,1.5) {$1$};
		
				\node at (0.48,2.25) {$0$};
				\node at (0.3,2.7) {$2$};
				\node at (0.25,3.75) {$2$};
				\node at (0.25,3.25) {$1$};
		
			\draw (0,0)--(-0.5,0);
		\draw (0,1)--(-0.5,1);
		\draw (0,2)--(-0.5,2);
		\draw (0,3)--(-0.5,3);
			\draw (0,4)--(-0.5,4);
		\node at (-0.4,0.5) {$\cdots$};
		\node at (-0.4,1.5) {$\cdots$};
		\node at (-0.4,2.5) {$\cdots$};
			\node at (-0.4,3.5) {$\cdots$};	
					\node at (2.5,-1.2) {(d)};		
			\end{scope}	
	
		\end{scope}
	\end{tikzpicture}
\end{center}
\end{Ex}

\section{Young wall construction of the level-1 highest weight crystals}\label{Young wall realization}

For a Young wall $Y=(\cdots,y_1,y_0)$, we define the \textit{pre $i$-signature of $Y$}  by 
$$(\cdots,\text{sign}_i(y_1),\text{sign}_i(y_0)).$$

We cancel every $(+,-)$ pair in $(\cdots,\text{sign}_i(y_1),\text{sign}_i(y_0))$ to obtain a finite sequence of $-$'s follows $+$'s, reading from left to right. This sequence is called the \textit{$i$-signature of $Y$}.

\medskip

We define $\tilde{E}_i Y$ to be the Young wall obtained from $Y$
by removing the $i$-block corresponding to the right-most $-$ in the
$i$-signature of $Y$.
We define $\tilde{E}_i Y=0$ if there exists no $-$ in the $i$-signature of $Y$.

\medskip

We define $\tilde{F}_i Y$ to be the Young wall obtained from $Y$
by adding an $i$-block to the column corresponding to the left-most $+$
in the $i$-signature of $Y$.
We define $\tilde{F}_i Y = 0$ if there exists no $+$ in the $i$-signature of $Y$.

\medskip

The above rule is called the \textit{tensor product rule} for Young walls.

\vskip 2mm

Let $\mathcal Y(\lambda)$ ($\lambda(c)=1$) be the set of all reduced Young walls. We define the maps

\begin{equation*}
	\wt : \mathcal Y(\lambda) \longrightarrow P,\quad 
	\epsilon_i : \mathcal Y(\lambda) \longrightarrow \Z, \quad 
	\phi_i : \mathcal Y(\lambda) \longrightarrow \Z
\end{equation*}
by
\begin{equation*}
	\begin{aligned}\mbox{}
		\wt(Y) &= \lambda - \sum_{i\in I} k_i \alpha_i,\\
		\epsilon_i(Y) &= \text{the number of $-$'s in the $i$-signature of $Y$,}\\
		\phi_i(Y) &= \text{the number of $+$'s in the $i$-signature of $Y$,}
	\end{aligned}
\end{equation*}
where $k_i$ is the number of $i$-blocks in $Y$ that have been added to
the ground-state wall $Y_{\lambda}$.

\begin{Prop}\label{prop:close}
	For any $Y\in\mathcal Y(\lambda)$, we have
	\begin{equation*}
		\tilde{E}_iY\in \mathcal Y(\lambda)\cup \{0\},\quad \tilde{F}_iY\in\mathcal Y(\lambda)\cup \{0\}. 
	\end{equation*}
\end{Prop} 

\begin{proof}
Let $Y=(\cdots,y_{k+1},y_k,y_{k-1},\cdots,y_2,y_1)\in\mathcal Y(\lambda)$. We assume that $$\tilde{F}_iY=(\cdots,y_{k+1},z_k,y_{k-1},\cdots,y_2,y_1),$$ 
where $z_k$ is obtained by adding an $i$-block on $y_k$. Then in the $i$-signature of $Y$, there is at least one $+$ at the $k$'th position and there is no $+$ at $(k+1)$'th position. Moreover, there is no $-$ at $(k-1)$'th position.

Thus, $(y_{k+1},y_k)$ (resp. $(y_k,y_{k-1})$) and $(y_{k+1},z_k)$ (resp. $(z_k,y_{k-1})$) are in the same connect component.  By  Lemma \ref{constant}, we have
\begin{align*}
&H^{\mathrm{aff}}(y_{k+1}({|y_{k+1}|}_0)\otimes y_k({|y_k|}_0
))=H^{\mathrm{aff}}(y_{k+1}({|y_{k+1}|}_0)\otimes z_k({|z_k|}_0
))=0,\\
&H^{\mathrm{aff}}(y_{k}({|y_{k}|}_0)\otimes y_{k-1}({|y_{k-1}|}_0
))=H^{\mathrm{aff}}(z_{k}({|z_{k}|}_0)\otimes  y_{k-1}({| y_{k-1}|}_0
))=0.
\end{align*} 
Since the only columns adjacent to $z_k$ are $y_{k-1}$ and $y_{k+1}$, in $\tilde{F}_iY$, we have 
$$H^{\mathrm{aff}}(y_{j+1}({|y_{j+1}|}_0)\otimes y_j({|y_j|}_0
))=0$$ 
for $j\neq k-1,k$. 

By the above discussion, we can see that the Young wall $\tilde{F}_iY$ is reduced. We can prove that $\tilde{E}_iY$ is reduced in the same way.
\end{proof}

By Proposition \ref{prop:close}, we have the following theorem.

\begin{Thm}\label{thm:affine crystal structure}
	The maps $\wt : \mathcal Y(\lambda) \rightarrow P$,
	$\tilde{E}_i,\tilde{F}_i : \mathcal Y(\lambda) \rightarrow \mathcal Y(\lambda)\cup \{0\}$ 
	and $\epsilon_i,\phi_i : \mathcal Y(\lambda) \rightarrow \Z$ define an affine crystal
	structure on the set $\mathcal Y(\lambda)$.
\end{Thm}

\begin{Thm}\label{thm:main}
	There exists an isomorphism of $U_q(D_4^{(3)})$ (resp. $U_q(G_2^{(1)})$)-crystals
	\begin{equation*}
		\mathcal Y(\lambda) \stackrel{\sim} \longrightarrow B(\lambda)
		\ \ \text{given by} \ \ 
		Y_{\lambda} \longmapsto u_{\lambda},
	\end{equation*}
	where $u_{\lambda}$ is the highest weight vector in $B(\lambda)$.
\end{Thm}
\begin{proof}
We only give the proof for the case of $U_q(D_4^{(3)})$. The proof for the case of $U_q(G_2^{(1)})$ is similar.
	
By Proposition \ref{prop:path realization}, we only need to prove $\mathcal Y(\lambda)\cong\mathcal P(\lambda)$. By Proposition \ref{perfect isomorphism}, we define a map $\Phi:\mathcal Y(\lambda)\to\mathcal P(\lambda)$ as follows.
	\begin{align*}
		\Phi:\mathcal Y(\lambda)\to\mathcal P(\lambda),\quad (\cdots,y_1,y_0)\mapsto (\cdots,\psi(y_1),\psi(y_0)).
	\end{align*}
	Using the tensor product rule on Young walls, it is straightforward to 
	check that $\Phi$ commutes with the Kashiwara operators, i.e., 
	\begin{equation*}
		\tilde{e}_{i} \circ \Phi = \Phi \circ \tilde{E}_{i},\ \ \tilde{f}_{i} \circ \Phi = \Phi \circ \tilde{F}_{i} \ \ \text{for all} \ i \in I.
	\end{equation*}

Let $\mathbf{p}=(p(k))^{\infty}_{k=0}$ be a path in $\mathcal P(\lambda)$. Let $Y=(\cdots,y_1,y_0)$ and $Y'=(\cdots,z_1,z_0)$ be reduced Young walls in $\mathcal Y(\lambda)$ such that $$(\cdots,\psi(y_1),\psi(y_0))=(\cdots,\psi(z_1),\psi(z_0))=\mathbf{p}.$$
Let $d_i$ and $d_{i+1}$ $(i\geq 0)$ denote the Young columns corresponding to $p(i)$ and $p(i+1)$, respectively. We must have
\begin{equation*}
y_i=z_i=d_i,\quad y_{i+1}=z_{i+1}=d_{i+1}.
\end{equation*}
By Proposition \ref{adjacent column}, we have 
\begin{equation*}
|y_i|-|y_{i+1}|=|z_i|-|z_{i+1}|\ \text{for all}\ i\geq 0.
\end{equation*}
Thus we have $Y=Y'$. Hence the map $\Phi$ is an injective. 

For a given path $\mathbf{p}=(p(k))^{\infty}_{k=0}$, we can first draw  the reduced adjacent columns $(y_1,y_0)$ corresponding to $(p(1),p(0))$, and then we draw the reduced adjacent columns $(y_2,y_1)$ corresponding to $(p(2),p(1))$, and so on. Finally, we can obtain a reduced Young wall $(\cdots, y_1,y_0)$ corresponding to the path $\mathbf{p}$. Hence the map $\Phi$ is surjective. 

Combining the above discussions, our assertion holds as desired.
\end{proof}

The top part of the crystal  $\mathcal Y(\Lambda_0)$ for $U_q(D_4^{(3)})$ and the top part of the crystals $\mathcal Y(\Lambda_0)$ and $\mathcal Y(\Lambda_2)$ for $U_q(G_2^{(1)})$ are given as follows (see Figure~\ref{top part 0 D43}, Figure~\ref{top part 0 G21} and Figure~\ref{top part 2 G21}).

\newpage

\begin{figure}[H]
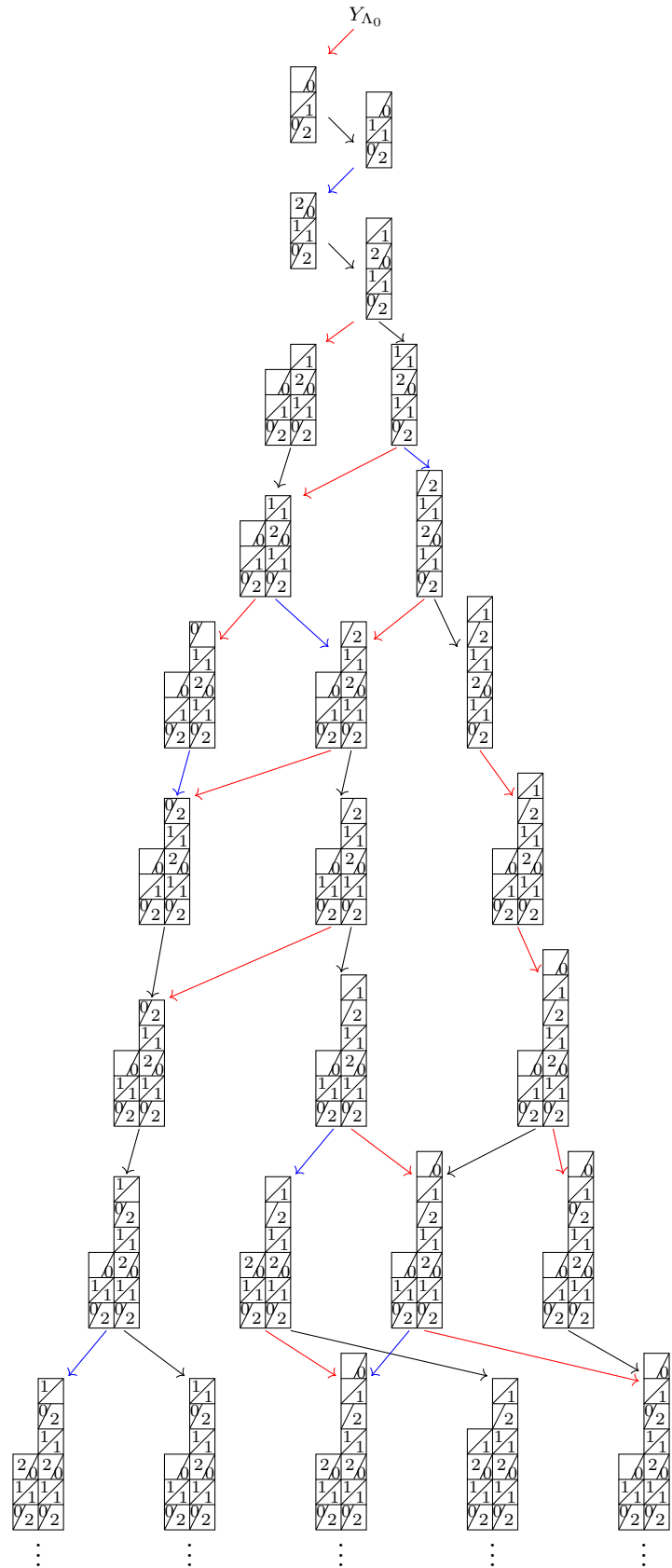

	\centering

								\vskip -2mm
								\caption{The top part of the crystal $\mathcal Y(\Lambda_0)$ for $U_q(D_4^{(3)})$}\label{top part 0 D43}
							\end{figure}

\newpage

\begin{figure}[H]
\centering
	%
\vskip -2mm
\caption{The top part of the crystal $\mathcal Y(\Lambda_0)$ for $U_q(G_2^{(1)})$}\label{top part 0 G21}
\end{figure}

\medskip

\newpage

\begin{figure}[H]
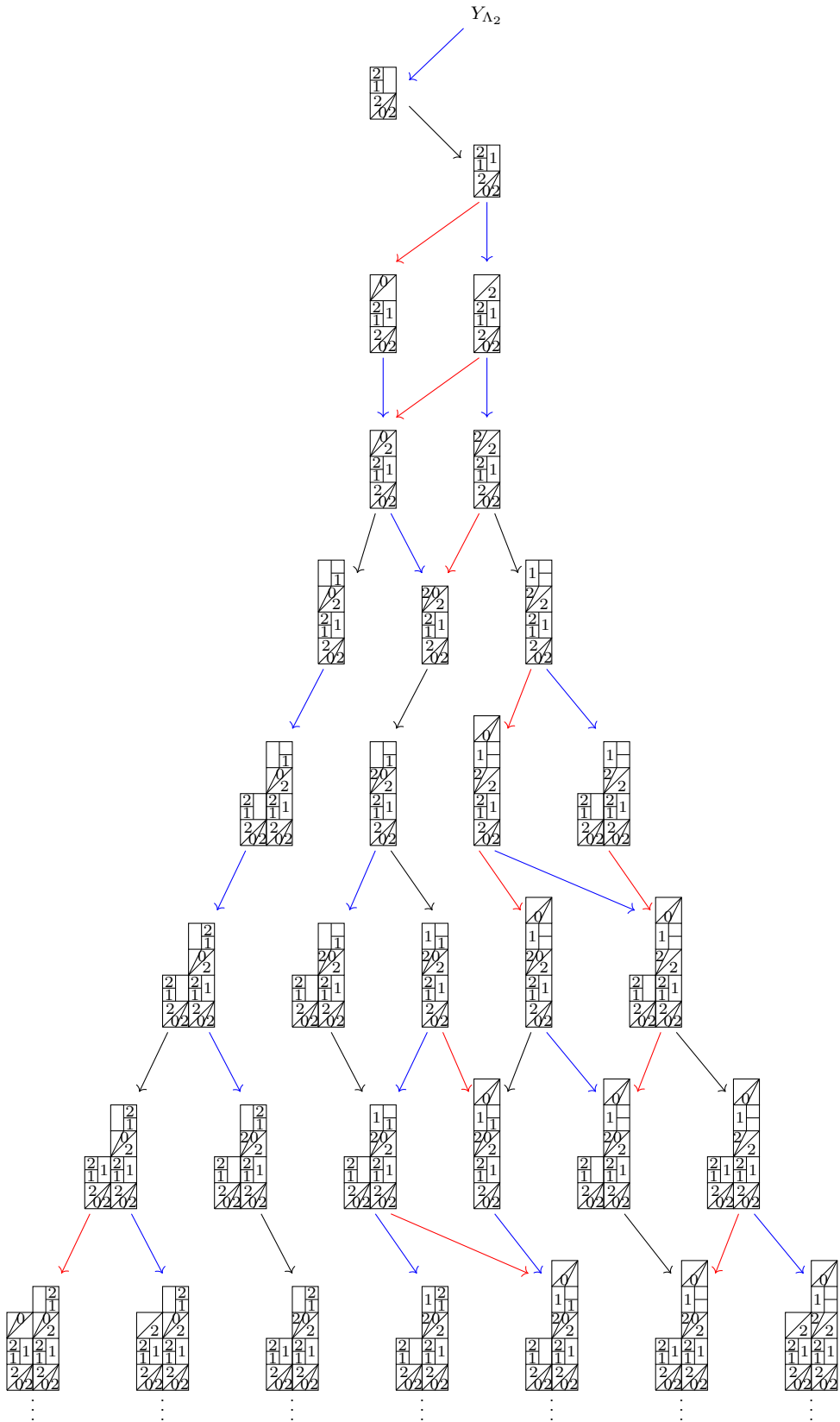

	\centering
	%
	\vskip -2mm
	\caption{The top part of the crystal $\mathcal Y(\Lambda_2)$ for $U_q(G_2^{(1)})$ }\label{top part 2 G21}
\end{figure}


\begin{thebibliography}{}	

\bibitem{HKL04}
J. Hong, S.-J.~Kang, H. Lee, \emph{ Young wall realization of crystal graphs for $U_q(C^{(1)}_n)$}, Comm. Math. Phys \textbf{244} (2004), 111--131.

\bibitem{JM11} R. L. Jayne, K. C. Misra,\emph{ On Demazure crystals for $U_q(G_2^{(1)})$}, Proc. Amer. Math. Soc.
\textbf{139} (2011), 2343-2356.

\bibitem{Kang03}
S.-J.~Kang, \emph{Crystal bases for quantum affine algebras and combinatorics of Young walls},
Proc. London Math. Soc. \textbf{86}
(2003), 29--69.

\bibitem{KMN1}
S-J. Kang, M. Kashiwara, K. C. Misra, T. Miwa, T. Nakashima,  A. Nakayashiki,
\emph{Affine crystals and vertex models},
Int. J. Mod. Phys. \textbf {7} (suppl. 1A) (1992), 449-484.

\bibitem{KMN2}
S-J. Kang, M. Kashiwara, K. C. Misra, T. Miwa, T. Nakashima,  A. Nakayashiki,
\emph{Perfect crystals of quantum affine Lie algebras},
Duke Math. J. \textbf{68} (1992), 499-607.

\bibitem{KL06}
S.-J. Kang, H.~Lee, \emph{Higher level affine crystals and {Y}oung walls},
Algebr. Represent. Theory \textbf{9} (2006), 593--632.

\bibitem{KL09}
S.-J. Kang, H. Lee, \emph{Crystal bases for quantum affine algebras and Young walls}, J. Algebra \textbf{322} (2009), 1979--1999.

\bibitem{Kas90}
M.~Kashiwara, \emph{Crystalizing the $q$-analogue of universal enveloping
	algebras}, Comm. Math. Phys. \textbf{133} (1990), 249--260.

\bibitem{Kas91}
M.~Kashiwara, \emph{On crystal bases of the $q$-analogue of universal enveloping
	algebras}, Duke Math. J. \textbf{63} (1991), 465--516.

\bibitem{KMOY} 
M. Kashiwara, K. C. Misra, M. Okado and D. Yamada,
\emph{Perfect crystals for $U_q(D_4^{(3)})$}, J. Algebra \textbf{317} (2007), 392-423.

\bibitem{KN94} M. Kashiwara, T. Nakashima, 
\emph{Crystal graphs for representations of the $q$-analogue of classical Lie algebras}, 
J. Algebra \textbf{165} (1994), 295--345.

\bibitem{KS19}
J. A. Kim, D. U. Shin, \emph{Strip bundle realization of the crystals over $U_q(G_2^{(1)})$}, Journal of Mathematical Physics \textbf{60} (2019), 111703.

\bibitem{KS22}
J. A. Kim, D. U. Shin, \emph{Strip bundle realization of the crystals over $U_q(D_4^{(3)})$}, Journal of Mathematical Physics \textbf{63} (2022), 121701.

\bibitem{M05} 
K. C. Misra, \emph{On Demazure crystals for $U_q(D_4^{(3)})$}, 
Contemporary Math. \textbf{392} (2005), 83-93.

\bibitem{MMO10} 
K.C. Misra, M. Mohamad, M. Okado,
\emph{Zero action on perfect crystals for $U_q(G_2^{(1)})$},
SIGMA \textbf{6} (2010), 022, (12 pages).

\bibitem{NZ97} T. Nakashima, A. Zelevinsky, \emph{Polyhedral realizations of crystal bases for quantized Kac-Moody algebras}, Adv. Math. \textbf{131} (1997), 253--278.

\bibitem{Nak99}
T. Nakashima, \emph{Polyhedral realizations of crystal bases for integrable highest
weight modules}, J. Algebra \textbf{219} (1999), 571-597.

\bibitem{Ya98}
S. Yamane,\emph{ Perfect crystals of $U_q(G_2^{(1)})$}, J. Algebra  \textbf{210} (1998), 440-486.

\end{thebibliography}
\end{document}